\documentclass[a4paper,12pt]{amsart}
\usepackage[latin2,utf8]{inputenc}

\usepackage{amsmath,amssymb,amsthm, amsfonts}
\usepackage{verbatim}
\usepackage{graphicx}
\usepackage{arydshln}
\usepackage{xcolor}
\usepackage{tikz}
\usepackage{marvosym}
\usepackage{marvosym}

 \numberwithin{equation}{section}
\numberwithin{figure}{section}
\numberwithin{table}{section}

\newcommand{\R}{\mathbb{R}}
\newcommand{\N}{\mathbb{N}}

\newcommand\Z{\mathbb{Z}}
\newcommand{\cC}{\mathcal{C}}
\newcommand{\cN}{\mathcal{N}}
\newcommand{\cR}{\mathcal{R}}
\renewcommand{\div}{{\rm div}\,}

\newcommand{\da}{\delta\! a}
\newcommand{\dn}{\delta\! n}
\newcommand{\dr}{\delta\! r}

\newcommand{\du}{\delta\! u}
\newcommand{\dX}{\delta\! X}

\newcommand{\dH}{\delta\! H}

\newcommand{\ep}{\varepsilon}
\renewcommand{\.}{\;\!}

\def\cC{{\mathcal C}}

\def\cF{{\mathcal F}}

\def\cH{{\mathcal H}}

\def\cL{{\mathcal L}}

\def\cP{{\mathcal P}}

\def\cS{{\mathcal S}}

\def\Id{\hbox{\rm Id}}

\renewcommand{\div}{\mbox{\rm div}\;\!}
\def\ddj{\dot\Delta_j}

\def\wh{\widehat}
\def\wt{\widetilde}

\newcommand{\Sum}{\displaystyle \sum}

\newcommand{\with}{\quad\hbox{with}\quad}
\newcommand{\andf}{\quad\hbox{and}\quad}

\newcommand{\f}
{\mathfrak f \,\!}

\newtheorem{thm}{Theorem}[section]
\newtheorem{lem}{Lemma}[section]

\newtheorem{rmk}{Remark}

\title[Euler system with nonlocal pressure]{The compressible Euler system with nonlocal pressure: global existence and relaxation}
\author{Raphael Danchin \& Piotr Bogusław Mucha }

\address{R. Danchin: Univ Paris Est Creteil, Univ Gustave Eiffel, CNRS, LAMA UMR8050, F-94010 Creteil, France}
\address{P.B. Mucha: Institute of Applied Mathematics and Mechanics, University of Warsaw, Banacha 2, 02-097 Warsaw, Poland}

\date{}

\begin{document}

\maketitle

\begin{abstract}   
 We here investigate  a modification of the compressible barotropic Euler system
with friction, involving a fuzzy nonlocal pressure term in place of the conventional one. This nonlocal term is parameterized by $\ep > 0$ and formally 
tends to the classical pressure when $\ep$ approaches zero.
The central challenge is to establish that this system is a reliable approximation of the classical compressible Euler system.
We establish the global existence and uniqueness of regular solutions 
in the neighborhood of the static state with density  $1$ and  null velocity. Our results are demonstrated independently of the parameter $\ep,$ which 
enable us to prove the  convergence of solutions 
to those of the classical Euler system.  Another consequence is 
the rigorous justification of the convergence  of the mass equation 
to various versions of the porous media equation 
 in the asymptotic limit where the friction tends to infinity.
Note that our results are demonstrated in the whole space, which 
 necessitates  to use the   $L^1(\mathbb{R}_+;\dot B^\sigma_{2,1}(\R^d))$ 
 spaces framework.
\end{abstract}

\vskip1cm

{\small 



}

\section{Introduction}

The phenomena of collective behavior are at the crossroads of various scientific disciplines and are currently the subject of active research.
They find their roots in diverse fields such as sociology, biology, and classical physics \cite{CS,OPA,TBD}. At the microscale level, these phenomena  are often described by simple Ordinary Differential Equations, as in e.g. the $N$-body problem. However, when the number of agents or particles becomes prohibitively large, such naive descriptions prove to be ineffective. Consequently, at the macroscale, it becomes suitable to adopt a hydrodynamical approach to model and understand these complex systems \cite{CCTT,DMPW,MO}.

This paper delves into the analysis of a modified version of the classical compressible Euler system, incorporating a nonlocal force designed to induce mass alignment among the constituent elements. 
This modification consists in replacing the classical pressure term by a non-local fuzzy approximation, which is designed to model the communication of each particle/agent with other particles located in a non-trivial neighborhood.

\medbreak
More precisely, we are concerned with the following class of systems in the whole space $\R^d$:
\begin{equation}\label{Kepsilon}
  \left\{  \begin{array}{l}
         \rho_t + \div (\rho u)=0 \\[1ex]
         \rho u_t + \rho u \cdot \nabla u  + \f \rho u = -\rho \nabla K_\ep \ast \rho.
    \end{array}\right.
\end{equation}
Above, $\rho=\rho(t,x)\in\R_+$ and $u=u(t,x)\in\R^d$ denote the
density and velocity functions of the studied ``matter'', respectively.
The positive real number $\f$ is the friction coefficient
and the family of smooth potentials $(K_\ep)_{\ep>0}$ is assumed to tend 
to  the Dirac measure at $0,$ when $\ep$ goes to $0.$ The convolution in the 
right-hand side of \eqref{Kepsilon} is taken with respect to space variables. 
Hence formally, in the limit, we obtain the 
following compressible Euler system with friction:
\begin{equation}\label{Euler}\left\{\begin{aligned}
&\rho_t + \div (\rho u)=0, \\
&\rho u_t + \rho u \cdot \nabla u  + \f \rho u +\frac 12 \nabla  \rho^2=0.\end{aligned}\right.
\end{equation}

Our primary goal is to establish the global well-posedness
of System \eqref{Kepsilon} supplemented with initial data
$(\rho_0,u_0)$ which are perturbations of the constant solution $(\rho,u)=(1,0).$
Because the nonlocal term $\nabla K_\ep\ast \rho$ is rather smooth, 
proving local well-posedness results in the case of sufficiently smooth data
bounded away from zero  presents no particular difficulty.  
Indeed,  the velocity satisfies  a damped Burgers equation with a smooth source term, that 
can be considered independently  of the density equation.
In this way however, it is  difficult to prove the 
global existence since, typically, the source term $\nabla K_\ep\ast\rho,$ albeit smooth,  causes a linear growth 
of $L^1$-in-time norms of $\nabla u.$ Back to the transport equation, it is thus impossible to get uniform bounds in time for the density, and thus to close the estimates for all positive time.  Likewise, getting a control independent of $\ep$ in this way is hopeless.
\smallbreak
The main difficulty is that our system  does not enter in the classical theory of hyperbolic equations. 
Even  for fixed values of parameters $\f$ and $\ep$, a nonstandard approach is necessitated. The main points of our analysis (that is also valid for more  general pressure functions  than 
$P(\rho)=\rho^2/2$) are the following:

\smallbreak

\textbullet \ 
A fundamental challenge arises from the essential requirement of $L^1$ time integrability  for $\nabla u,$ that is,
\begin{equation}\label{eq:Lip}
\int_0^\infty \| \nabla u \|_{L^\infty}\,  dt < \infty.
\end{equation}
This is the key to controlling for all time the transport
terms of  \eqref{Kepsilon},   namely  $u \cdot \nabla \rho$ and $u\cdot \nabla u.$ 
Given the hyperbolic nature of the system, \eqref{eq:Lip} can only be achieved thanks to the dissipative term $\f \rho u$. 
In the whole space context, there exists no inherent mechanism to induce rapid temporal decay (we shall showcase  below that there is
no `spectral' gap for the linearized system).  
A way to overcome  the difficulty is to use  the framework of homogeneous Besov  spaces of type $\dot{B}^s_{2,1}(\mathbb{R}^d)$. Here, the factor `1' 
will enable us  to attain the  $L^1$ integrability over time, while
the factor `2'  reflects the fact that our framework is  related to the $L^2$ space, in keeping with the quasilinear hyperbolic nature of the system.

\smallbreak

\textbullet \
In order to achieve global results with  some uniformity with respect to $\ep,$ the mathematical analysis is   subtle. 
In fact, instead of helping, the smoothing kernel 
$K_\ep$ destroys the nice partially dissipative symmetrizable structure
of \eqref{Euler}. The so-called Shizuta-Kawashima condition (first pointed
out in \cite{Ka84}) is not satisfied, and the more modern approach 
of Beauchard-Zuazua \cite{BZ} (revisited in \cite{CBD2,D-rev})
cannot be used as is. 
Compared to \eqref{Euler}, the difficulty is that the operator 
$\nabla K_\ep$ provides less dissipation than the full gradient, as may be 
 already observed  on the following linearization of \eqref{Kepsilon}:
\begin{equation}\label{eq:eulerlin}
\left\{  \begin{array}{l}   a_t + \div u=0, \\[1ex]
         u_t +  u  +  \nabla K_\ep a=0.
    \end{array}\right.     
\end{equation}
In Fourier variables, the matrix of the system reads
reads $$\begin{pmatrix}0&i\xi\\i\xi^T\wh K_\ep(\xi)&1\end{pmatrix},\qquad \xi\in\R^d.
$$
The eigenvalues are $1$ with multiplicity $d-1$ (incompressible part of $u$)
and: 
\begin{itemize}
\item[] $\lambda^\pm(\xi)=\frac12\Bigl(1\pm\sqrt{1-4|\xi|^2\wh K_\ep(\xi)}\Bigr)\:$ if $\:4|\xi|^2\wh K_\ep(\xi)\leq1$;
\smallbreak
\item[] $\lambda^\pm(\xi)=\frac12\Bigl(1\pm i\sqrt{4|\xi|^2\wh K_\ep(\xi)-1}\Bigr)\:$ if $\:4|\xi|^2\wh K_\ep(\xi)\geq1.$
\end{itemize}\smallbreak
The Euler situation corresponds to $\ep=0,$ that is $\wh K_0\equiv 1.$
We then have two distinct regimes: low frequencies with one parabolic mode
and $d$ damped modes, and high frequencies with only damped modes. 

If $\ep>0,$ then the regime where $4|\xi|^2\wh K_\ep(\xi)<1$
is likely to include arbitrarily high  frequencies, since  the functions $\wh K_\ep$
that we will consider here have 
algebraic decay at  $\infty.$ Furthermore, for small values of $|\xi|^2\wh K_\ep(\xi)$
we  have $\lambda^-(\xi)\simeq |\xi|^2\wh K_\ep(\xi)$
that is, a \emph{degenerate} parabolic mode. 
A key observation is that in this regime the combination $w:=u+\nabla K_\ep a$ 
(often referred to in this article as the `damped mode')
tends to undergo an exponential dissipation.
\smallbreak 

\textbullet \
An essential requirement in our study is the establishment of uniform dependence on the parameter $\ep$. This is clearly needed
for justifying rigorously the convergence to the Euler system 
\eqref{Euler} in the asymptotics $\ep\to0.$

Leveraging energy-based techniques, 
we succeed in  controlling the essential quantities required for our analysis,  uniformly as $\ep \to 0$. This enables us to precisely determine the diffusive limit of our system. It is worth noting that our approach and functional framework
for solving 
\eqref{Kepsilon} is inspired by the recent  paper \cite{CBD2}. 
However, the loss of symmetry caused by the 
kernel $K_\ep$ will entail a number of difficulties 
that will be described in detail in the next section. 
For older global existence results concerning 
System \eqref{Euler} and the relaxation limit, the reader may consult \cite{CG,JR,XK1,XK2}.

\smallbreak

\textbullet \
To recover the optimal 
information coming from the basic spectral analysis that we performed 
above for \eqref{eq:eulerlin}, it is convenient 
to localize the system by means of a dyadic decomposition in the Fourier space
(the so-called Littlewood-Paley decomposition) then to implement the method 
that was used in \cite{CBD2,CBD3} for \eqref{Euler}. 
There is one more difficulty: in the process, in order to compensate the loss
of symmetry with respect to $(a,u),$ 
one has somehow to look at $K_\ep a$ as an `independent' unknown. 
This  leads us to consider commutators of nonlinear terms with $K_\ep.$ A central objective lies in the meticulous control of these commutators, \emph{uniformly with respect to $\ep$.}
In this endeavor, we have  to extend the techniques delineated in \cite[Chap. 2]{BCD} to accommodate more intricate scenarios, wherein paraproduct operator and expansion techniques become indispensable for addressing higher-order terms.
Here, the key is to use a  Taylor expansion at order two;  
which necessitates  a control 
of $\|\nabla^2 u\|_{L^\infty}$. This leads us to use a dual level of regularity
while, for the classical compressible  Euler system, 
it is enough to control $\|\nabla u\|_{L^\infty},$
and thus to use only one level of regularity.  

\smallskip 

\textbullet \
The last part of our study concerns the relaxation limit  $\f \to \infty$.
A distinctive feature of our functional setting is that it allows 
to deduce the general case $\f>0$ from 
the particular case $\f=1$ by mere rescaling, provided parameter $\ep$ has been suitably modified. Then, the key to 
proving the strong convergence is to look at  $w$ defined above as the beneficial and dissipative component of our system\footnote{Here we can  draw an analogy with our use
in \cite{DM-adv} of \emph{the effective viscous flux} for  viscous compressible flows, so as to  justify the convergence to the inhomogeneous incompressible Navier-Stokes equations.}. 

Depending on the type of asymptotics we are looking at, we will justify rigorously the transition to porous media type equations, namely:
\begin{equation*}
\partial_t r - \div  (r \nabla K_\ep r) = 0 \ \hbox{ or }\  \partial_t n - \div (n\nabla n) = 0.
\end{equation*}
It is noteworthy that when $\ep > 0$, the resulting equation corresponds to some degenerate porous media equation, with no parabolic  smoothing-out effect.


\section{Derivation  from the particle system}

In order to have a better understanding of the model presented  in the introduction, let us delve into the interactions occurring among particles at the microscopic level. 
We therefore  look at  second-order agent models in their general formulation:
consider  a set of  $N$ identical particles, each of which is identified by the index $k$, ranging from $1$ to $N$. At any time $t$, particle $k$ occupies the position  $x_k(t)$ and moves with an instantaneous velocity $v_k(t)$.

In our analysis, we make the underlying assumption that communication between these particles solely depends on aggregation-repulsion effects, contingent upon the positions of the agents. Furthermore, we incorporate frictional effects into the model to govern and ensure the system's stability.
Consequently, denoting by $\f$ the (nonnegative) friction coefficient, the temporal evolution of both position $\{x_k\}$ and velocity $\{v_k\}$ for each particle, where $k$ spans the values from $1$ to $N$,
 is governed by the following system of equations:
\begin{equation}\label{eq:boltz}
    \left\{\begin{array}{l}
        \displaystyle \dot x_k= v_k  \\[2pt]
        \displaystyle \dot v_k = -\f  v_k -\frac{1}{N} \sum_{l\in \{1,\cdots,N\}} \nabla K_\ep (x_k-x_l).
    \end{array}\right.
\end{equation}
%
Changing the scale from micro to macro setting, jumping over the kinetic formulation,
leads to System \eqref{Kepsilon} (see details in Appendix).
%
Then, assuming that  $K_\ep \to \delta$ as $\ep \to 0,$ we formally 
obtain $\rho \,\nabla K_\ep \ast \rho \to \frac{1}{2} \nabla \rho^2,$ and thus  
the Euler system \eqref{Euler}.


A simple example of a family of potentials $(K_\ep)_{\ep>0}$ can be built from 
 the characteristic function of the ball, namely we set
 (for a suitable normalization constant $c_d$):
$$
    K_\ep(x):=c_d\,\ep^{-d}(1-\ep^{-1}|x|)\chi_{B(0,\ep)}(x)
\ \hbox{ so that }\ 
    \nabla K_\ep = c_d\,\ep^{-d-1} \chi_{B(0,\ep)} \frac{x}{|x|}\cdotp
$$
To better understand the effects modelled by this potential,  let us concentrate on the mono-dimensional case. Then
$$
    K'_\ep \ast \rho(x) =
    \ep^{-2}\int_{|z|\leq \ep} \frac{z}{|z|} \rho(x-z)\,dz= 
    \ep^{-2}\int_{|z|\leq \ep} {\rm sgn}(z) (\rho(x-z) -\rho(x))\, dz.
$$
We observe that the force term arising from
the integral on the right-hand side of the equation stems from the 
necessity of maintaining mass balance over the intervals $(-\ep, 0)$ and $(0, \ep)$. For
multidimensional systems, while the weightings may become somewhat more intricate, 
the underlying mechanism remains fundamentally unchanged.
To gain a visual insight into the impact of this nonlocal term, 
the reader may pay attention to  Fig. \ref{fig:1} below. \nobreak
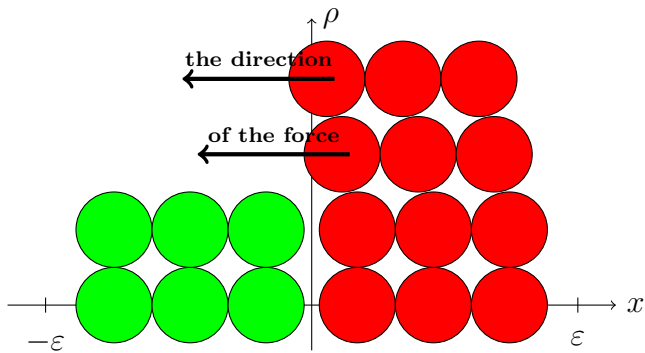
\begin{figure}[ht] 
  \centering
\begin{tikzpicture}
  \draw[->] (-4.0,0) -- (4,0) node[right] {$x$};
  \draw[->] (0,-0.6) -- (0,3.8) node[right] {$\rho$};
 \draw[-] (-3.5,-0.1) -- (-3.5,0.1);
 \node at (-3.5,-0.2) [below] {$-\varepsilon$};
\draw[-] (3.5,-0.1) -- (3.5,0.1);
  \node at (3.5,-0.2) [below] {$\varepsilon$};
  \foreach \x/\y in { -2.6/0, -1.6/0, -0.6/0, 
  -2.6/1, -1.6/1, -0.6/1} {
    \draw[fill=green] (\x,\y) circle (0.5);
  }
  \foreach \x/\y in {0.6/0, 1.6/0, 2.6/0,  
  0.6/1, 1.6/1, 2.6/1,  
  0.4/2, 1.4/2, 2.4/2, 
  0.2/3, 1.2/3, 2.2/3} {
    \draw[fill=red] (\x,\y) circle (0.5);
  }
  \draw[->, ultra thick] (0.5,2) -- (-1.5,2) node[midway, above] {\tiny \bf of the force};
  \draw[->, ultra thick] (0.3,3) -- (-1.7,3) node[midway, above] {{\tiny \bf the direction }};
\end{tikzpicture}
\caption{The mass contributed by the green balls on the segment $(-\ep, 0)$ exerts a comparatively lesser influence compared to that of the red balls located in the segment $(0, \ep)$. Consequently, the resultant force is oriented 
on the left. 
}
\label{fig:1}
\end{figure}

\goodbreak
The salient points of this analysis are valid in the  specific case
where the pressure is  of the form $P(\rho) \sim \rho^2$. To 
achieve more  general barotropic constitutive relations, one can introduce the density-induced communication protocol  of \cite{MT,MMP}. 
In that case we assume the communication between $k$-th and $l$-th agent to be 
of the form 
     ${\mathcal N}(K_\ep \ast \rho ) \nabla K_\ep (x_k-x_l),$
for some given function $\mathcal{N}$. 
At the level of the particle system,  ${\mathcal N}(K_\ep 
\ast \rho)$ measures the mass/number of particles in some vicinity of the agent $x_k$. At the hydrodynamical level, the convolution $K_\ep \ast \rho$ describes an average value of the density function $\rho$ in the vicinity of the examined point. Accordingly, the inclusion of the $\mathcal{N}$ factor serves to augment or diminish the influence of communication relative to the average density in the given region. In this way, the effects showed at Fig. \ref{fig:1} are  rescaled in terms of the magnitude of the mass in the considered neighborhood
 and our model System \eqref{Kepsilon} has to be replaced with the following more general one:
\begin{equation}\label{Kepsilon-PP}
  \left\{  \begin{array}{l}
         \rho_t + \div (\rho u)=0 \\[1ex]
         u_t +  u \cdot \nabla u  + \f  u + {\mathcal N}(K_\ep \ast \rho) \nabla K_\ep \ast \rho=0.
    \end{array}\right.
\end{equation}
 For $\ep$ tending to $0,$ we formally have
\begin{equation}\label{eq:P}
  \rho  {\mathcal N}(K_\ep \ast \rho) \nabla K_\ep  \ast \rho 
    \to \rho \mathcal{N}(\rho) \nabla \rho = \nabla P(\rho)
    \with P'(\rho)=\rho \mathcal{N}(\rho).
\end{equation}
Hence,  we get  the classical barotropic Euler system
\begin{equation}\label{Euler-P}\left\{\begin{aligned}
&\rho_t + \div (\rho u)=0, \\
& u_t +  u \cdot \nabla u  + \f  u + \mathcal{N}(\rho) \nabla  \rho =0.\end{aligned}\right. 
\end{equation}
Note that the classical  pressure law  $P(\rho) = \rho^\gamma$ ($\gamma\geq1$) 
(and thus the  isentropic Euler system with friction)
may be achieved if taking   $\mathcal{N}(\rho) = \rho^{\gamma -2},$ up to a multiplicative constant.   


\section{Results}

Before presenting our main results, some
definitions, assumptions  and notation are in order. 
Let us first specify our assumptions on the family 
$(K_\ep)_{\ep>0}.$
Since our approach  bases essentially on the Fourier transform,  the convergence of $K_\ep$
to the Dirac measure can be equivalently seen  as 
    $ \widehat{K}_{\ep}  \to 1 \mbox{ locally on } \R^d$.
Our analysis requires
$\wh K_\ep$ to keep
the same order of magnitude 
inside any annulus $\{2^{j-1}\leq |\xi|\leq 2^{j+1}\} \subset \R^d$ with $j\in\Z.$ 
In fact, we shall assume throughout that 
$$K_\ep=L_\ep\ast L_\ep$$ with $L_\ep$ a  real valued function such that $\wh L_\ep$ 
is nonincreasing with range in $[0,1],$ satisfies $\wh L_\ep(0)=1$ and,  for some $\kappa>0,$ 
\begin{equation}\label{eq:Kep}
\begin{aligned}&\sup_{\ep>0} \bigl(\|L_\ep\|_{L^1}+\|z\nabla L_\ep\|_{L^1}
+\|(z\otimes z)\nabla^2L_\ep\|_{L^1}\bigr)<\infty,\\
\ \kappa \wh L_\ep(\xi)\leq& \wh L_\ep(2\xi)\leq\kappa^{-1} \wh L_\ep(\xi)
\!\andf\! \xi_k\partial_\ell\wh L_\ep(\xi)\lesssim \wh L_\ep(\xi),\quad\! 1\leq k,\ell\leq d,\  \xi\in\R^d,\ \ep>0.\end{aligned}
\end{equation}
The above condition rules out sharp spectral cut-off or Gaussian functions.
A simple example of a family $(K_\ep)_{\ep>0}$ satisfying \eqref{eq:Kep}
is  to take $\widehat{K}_\ep(\xi) = \widehat{K}(\ep \xi)$ with
\begin{equation}\label{def:Kpe}
    \widehat{K}(\xi)= (\wh L(\xi))^2\andf 
    \widehat{L}(\xi)= \frac{1}{(1+|\xi|^2)^{m/2}}\cdotp
\end{equation}
For $m>d$,  one can  show from the standard properties of Fourier transform that \eqref{eq:Kep} is indeed satisfied.
\medbreak
Let us next introduce  the 
Littlewood-Paley decomposition on which on entire
analysis is based. 
Fix  a smooth function $\phi: \R_+ \to [0,1]$ supported in $\{1/2 \leq r \leq 2\}$ such that
$$
    \sum_{k\in \mathbb{Z}} \phi(2^{-k} r) = 1 \quad\hbox{for all }\  r>0.
$$
Setting $ \varphi (\xi):= \phi(|\xi|)$ for all $\xi\in\R^d,$
one can define a homogeneous Littlewood-Paley decomposition $\{\dot \Delta_k\}_{k\in \mathbb{Z}}$
over the space $\R^d$ in the following way:
$$
    \dot \Delta_k u := \varphi(2^{-k}D)u =
    {\mathcal{F}}^{-1}(\varphi 2^{-k} \cdot) \mathcal{F}u) 
    \with iD:=(\partial_{x_1},...,\partial_{x_d}) \ \hbox{ for }\ u\in\cS'(\R^d).
$$
Homogeneous Besov `norms' are  defined as follows for all $s\in\R$ and $1\leq p,q\leq\infty$:
$$
    \|u\|_{\dot B^s_{p,q}(\R^d)}:=
    \left\|2^{sk}\|\dot \Delta_k u\|_{L^p(\R^d)} \right\|_{\ell^q(\mathbb{Z})}.
$$
Actually, as $\|P\|_{\dot B^s_{p,q}}=0$ for any polynomial on $\R^d$, in the general tempered distribution 
setting $\|\cdot\|_{\dot B^s_{p,q}}$ is just a semi-norm. 
To get around the problem, we proceed as in \cite{BCD}, adopting the following definition:
$$
    \dot B^s_{p,q}(\R^d) := \left\{ 
    u \in \mathcal{S}'_h(\R^d): \| u\|_{\dot B^s_{p,q}} < \infty\right\},
$$
where  $\mathcal{S}'_h(\R^d)$ is the set of  tempered distributions
such that for all $\theta \in C^\infty_c (\R^d)$ we have
\begin{equation}\label{cond:LF}
    \lim_{\lambda \to \infty} \theta(\lambda D)u=0
    \mbox{ \ in \ } L^\infty(\R^d).
\end{equation}
Next, in accordance with our preceding 
spectral analysis of the linear system \eqref{eq:eulerlin}, we introduce the following notation
 where the value of the small positive absolute constant $\nu_0$ will be specified later in the paper:  
\begin{eqnarray}\label{eq:fucknot}
&\|z\|^\ell_{\dot B^\sigma_{2,1}}:=\Sum_{2^j\wh L_\ep(2^j)<\nu_0} 2^{j\sigma}\|\ddj z\|_{L^2}
\andf\|z\|^h_{\dot B^\sigma_{2,1}}:=\sum_{2^j\wh L_\ep(2^j)\geq \nu_0} 2^{j\sigma}\|\ddj z\|_{L^2},\\
\label{eq:fucknotbis}
&z^\ell:=\Sum_{2^j\wh L_\ep(2^j)<\nu_0} \ddj z
\andf z^h:=\sum_{2^j\wh L_\ep(2^j)\geq \nu_0} \ddj z.
\end{eqnarray}
Note that this decomposition of frequencies 
does not quite correspond to what will be sometimes called, improperly, in the paper 
\emph{low and high frequencies.} As said before, the fact that  $2^j\wh L_\ep(2^j)<\nu_0$
does not exclude  large values of $2^j.$
\smallbreak

Let us finally introduce the functional spaces that will be used for solving \eqref{Kepsilon}:
for all $\sigma \in\R$ and kernel $K_\ep=L_\ep\star L_\ep$
satisfying \eqref{eq:Kep}, 
the  space $E^\sigma_{K_\ep}$ stands for the 
set of all pairs $(a,u)\in \cC_b(\R_+;
\dot B^{\sigma-1}_{2,1}\cap \dot B^{\sigma}_{2,1})$
such that, in addition:
\begin{multline}\label{def:E}
(\nabla u,\nabla L_\ep\star a)\in  \cC_b(\R_+;B^{\sigma}_{2,1}),
\quad
(u,\nabla u)\in L^1(\R_+;\dot B^\sigma_{2,1})\\\andf 
\int_0^t\bigl(\|(K_\ep a,\nabla K_\ep\ast a)\|^\ell_{\dot B^{\sigma+1}_{2,1}}
+\|\nabla L_\ep\ast a\|_{\dot B^{\sigma}_{2,1}}^h\bigr)d\tau<\infty.
\end{multline}
The version of  $E^\sigma_{K_\ep}$ 
corresponding to the case where $K_\ep\ast$ is the
identity operator will be considered 
for solving the Euler system \eqref{Euler}. We shall
denote it by just  $E^\sigma.$
\medbreak
We are now ready to state our main global existence result for System \eqref{Kepsilon}:
\begin{thm}\label{thm:GWP}  Assume that $d\geq2$ and consider initial data $\rho_0=1+a_0$ and $u_0$
such that  
$$u_0\in \dot B^{\frac d2}_{2,1}\cap \dot B^{\frac d2+2}_{2,1},
\quad a_0\in \dot B^{\frac d2}_{2,1}\cap \dot B^{\frac d2+1}_{2,1}\andf
\nabla^2L_\ep a_0 \in \dot B^{\frac d2}_{2,1}.$$
There exists an absolute  positive constant $\alpha_0$ such that if 
\begin{equation}\label{eq:smalldata}
\|u_0\|_{\dot B^{\frac d2+1}_{2,1}\cap \dot B^{\frac d2+2}_{2,1}}
+\|a_0\|_{\dot B^{\frac d2}_{2,1}\cap \dot B^{\frac d2+1}_{2,1}}
+\|\nabla^2L_\ep a_0 \|_{\dot B^{\frac d2}_{2,1}}\leq\alpha_0,
\end{equation}
then System \eqref{Kepsilon} with $\f=1$ supplemented with initial data 
$(\rho_0,u_0)$ admits a unique global classical
solution $(\rho,u)$ such that  $(a,u)$ with 
$a:=\rho-1$  belongs to the space $E_{K_\ep}^{\frac d2+1}$
defined in \eqref{def:E}.
Furthermore, there exists a constant $C$ independent of $\ep$
such that  for all $t\in\R_+,$ 
\begin{multline}\label{est:1} \|(a,\nabla a,\nabla^2L_\ep a)(t)\|_{\dot B^{\frac d2}_{2,1}}
+\|(u,\nabla u)(t)\|_{\dot B^{\frac d2+1}_{2,1}}+
\int_0^t \bigl(\|(u,\nabla u)\|_{\dot B^{\frac d2+1}_{2,1}} 
\\+\|(K_\ep a,\nabla K_\ep a)\|^\ell_{\dot B^{\frac d2+2}_{2,1}}
+\|\nabla L_\ep a\|_{\dot B^{\frac d2+1}_{2,1}}^h\bigr)d\tau
 \leq C\bigl(\|(a_0,\nabla a_0,\nabla^2L_\ep a_0)\|_{\dot B^{\frac d2}_{2,1}}
+\|(u_0,\nabla u_0)\|_{\dot B^{\frac d2+1}_{2,1}}\bigr)\cdotp\end{multline}
In addition, setting    $w = u +  \nabla K_\ep a,$ we have
\begin{multline}\label{est:2} 
\|(u,w)(t)\|_{\dot B^{\frac d2}_{2,1}} 
+ \int_0^t \|w\|_{\dot B^{\frac d2}_{2,1}}\,d\tau
+ \biggl(\int_0^t  \|u\|_{\dot B^{\frac d2}_{2,1}}^2\,d\tau\biggr)^{1/2}
\\\leq C\bigl(\|(a_0,\nabla a_0,\nabla^2L_\ep a_0)\|_{\dot B^{\frac d2}_{2,1}}
+\|(u_0,\nabla u_0,\nabla^2 u_0)\|_{\dot B^{\frac d2}_{2,1}}\bigr)\cdotp
 \end{multline}
\end{thm}
Several important remarks are in order:
\begin{itemize}
\item[--] The above statement is valid for any $\ep>0,$ \emph{with constants
$\alpha_0$ and $C$ independent of~$\ep$}. 
We stated only the case $\f=1$ for simplicity.  However, whenever the family $(L_\ep)_{\ep>0}$
is given by $L_\ep=\ep^{-d} L(\ep^{-1}\cdot),$ then the rescaling
\begin{equation}\label{eq:keyrescaling}
\rho(t,x)=\wt\rho(\f t,\f x)\andf u(t,x)=\wt u(\f t,\f x)
\end{equation}
 transforms the case $(\f,\ep)$ into  the case $(1,\ep\f).$ 
Hence, one may deduce from the above theorem a global well-posedness
result that  is valid for any $\ep>0$ and $\f>0$ (see the beginning of Section \ref{s:limit}).

\item[--]  The integrability property of the damped mode $w$ is the key 
to proving strong convergence results in the asymptotics $\f\to\infty.$

\item[--]   Our approach is  appropriate for dealing with the \emph{multi-dimensional} case. 
 In the one-dimensional case, the above result is still valid
but the proof has to be slightly modified and it is 
very likely that  stronger results may be obtained by different techniques (see a similar problem in \cite{CCZ}). 

\item[--] A global existence result in the spirit of Theorem \ref{thm:GWP} 
can be established in the more general setting of System \eqref{Kepsilon-PP}
(see Subsection \ref{s:pressure}).
\end{itemize}
Let us quickly outline the proof of Theorem \ref{thm:GWP}.
The core consists in establishing a priori estimates 
in  the functional framework given above for the following linearization 
of \eqref{Kepsilon}: 
\begin{equation}\label{eq:eulerconv}
  \left\{  \begin{array}{l}
         a_t  + \div  u +v\cdot\nabla a+ b\,\div u=f, \\[1ex]
         u_t   +  u + v\cdot\nabla u +\nabla K_\ep \ast  a=g,
    \end{array}\right.
\end{equation}
where the given pair $(b,v)$  satisfies
\begin{equation}\label{eq:bv}
b_t+\div((1+b)v)=0.
\end{equation}
We consider this   linear system with \emph{variable coefficients} since just looking at \eqref{eq:eulerlin} with source terms 
cannot prevent a loss of derivatives. 
Here, we shall actually extend our analysis to the more general situation of 
 $(1+c)\nabla K_\ep\ast a$ in the second line of \eqref{eq:eulerconv}
and to a whole range of regularity exponents.  The first extension  is motivated by our wish to 
be able to consider more general pressure laws (see \eqref{eq:P}) and 
the second one, to have a ready-to-use result for proving  stability estimates (and thus uniqueness)
and the convergence from \eqref{Kepsilon} to \eqref{Euler} by the same stroke.
\smallbreak
Now,  to get optimal a priori estimates for \eqref{eq:eulerconv}, 
we adapt the method of \cite{CBD2}. This consists in, first, 
localizing \eqref{eq:eulerconv} by means of a Littlewood-Paley decomposition then:
\begin{itemize}
\item[--] proving estimates for the (degenerate) parabolic mode $a$
and the damped mode $w=u+\nabla K_\ep a$ rather than for $(a,u),$ 
in the regime of frequencies $\xi$ such that $|\xi|\wh L_\ep(\xi)\leq \nu_0$;
\item[--] considering a Lyapunov functional  \emph{depending  on the coefficient 
$b$} that encodes the information on $a,u,\nabla u,\nabla L_\ep a$
for frequencies such that  $|\xi|\wh L_\ep(\xi)\geq\nu_0.$
As for the Euler equation in \cite{CBD2}, the dependence of this 
functional on $b$ and $c$ is designed to exactly compensate the loss of derivative
coming from $b\,\div u$ and $c\nabla K_\ep\ast a$. This could be seen as a symmetrization of \eqref{eq:eulerconv} after spectral localization  by means of a Littlewood-Paley decomposition. 
\end{itemize}
Since, for proving global existence, we will have to take eventually $b=a$ and $v=u,$
 checking  at every step of the proof that 
only norms of $(b,v)$ that can be controlled in terms of the norms 
coming into play in Theorem \ref{thm:GWP} is fundamental. 
\smallbreak
The other steps of the proof are more standard: having at hand estimates 
for \eqref{eq:eulerconv} in the spaces $E^{\sigma}_{K_\ep},$
one can close the estimates globally for System \eqref{eq:Kep}
under Condition \eqref{eq:smalldata} in the space 
$E^{\frac d2+1}_{K_\ep}$ and prove  stability estimates  in $E^{\frac d2}_{K_\ep}.$  
These will enable us to prove the uniqueness part of the statement. 
As for the existence part, we first smooth out the data and prove
the existence of a sequence of local-in-time solutions $(a^{(n)},u^{(n)})_{n\in\N}$
with high Sobolev regularity. 
Combining our estimates in $E^{\frac d2+1}_{K_\ep}$ with a continuation criterion, 
we then succeed in proving that these smooth solutions are actually global, 
and that  $(a^{(n)},u^{(n)})_{n\in\N}$ is bounded in  $E^{\frac d2}_{K_\ep}.$  
Combining with functional analysis arguments allow to conclude to convergence, up to subsequence, 
to a solution of \eqref{Kepsilon} with the desired properties. 
\medbreak
Our second aim is to justify   that \eqref{Kepsilon} is indeed an approximation of \eqref{Euler}.
More precisely, we  show that
 the solution  of \eqref{Kepsilon} constructed above converges strongly and for all time for $\ep\to0$, to the unique solution of \eqref{Euler}. 
This is achieved in the following theorem that essentially follows from a variation on
stability estimates in $E^{\frac d2}_{K_\ep}.$
\begin{thm}\label{thm:conv} 
Assume,  in addition to \eqref{eq:Kep}, that  $L_\ep=\ep^{-d} L(\ep^{-1}\cdot)$ for a single
function $L.$
Consider initial data $(\rho_0=1+a_0,u_0)$ with $(a_0,u_0)$ in $\dot B^{\frac d2}_{2,1}\cap \dot B^{\frac d2+2}_{2,1}.$
There exists a universal constant $\alpha_0$ such that if 
\begin{equation}\label{eq:smallness2}\|(a_0,u_0)\|_{\dot B^{\frac d2}_{2,1}\cap \dot B^{\frac d2+2}_{2,1}}
\leq \alpha_0,\end{equation}
then, for all $\ep>0,$ System \eqref{Kepsilon} has a unique global 
solution $(\rho^\ep=1+a^\ep,u^\ep)$
with $(a^\ep,u^\ep)$ in $E^{\frac d2+1}_{K_\ep}$
and System \eqref{Euler} has a unique global 
solution $(\rho=1+a,u)$  with 
$(a,u)$ in $E^{\frac d2+1}.$
Furthermore, $$a^\ep\to a\ \hbox{ in }\ L^\infty_{loc}(\R_+;\dot B^{\frac d2+\alpha}_{2,1}),
\ \alpha\in[0,1)\andf 
u^\ep\to u\ \hbox{ in }\ L^\infty_{loc}(\R_+;\dot B^{\frac d2+\beta}_{2,1}),
\ \beta\in[0,2).$$
The above convergence holds uniformly on $\R_+$ if:
\begin{itemize}
    \item[--]  either $\eta\mapsto |\eta|^{-1}(\wh L(\eta)-1)$ is bounded;
\item[--] or  $a_0\in\dot B^{\frac d2-1}_{2,1}.$ 
In this case, we have $a^\ep$ and $a$ in $\cC_b(\R_+;\dot B^{\frac d2-1}_{2,1}).$ \end{itemize}
    \end{thm}
In the last part of the paper, we shall investigate  the 
high friction  limit $\f\to\infty$  for  \eqref{Kepsilon}. 
Our goal is to showcase the convergence of the (suitably rescaled) 
density toward a solution of  either the well-known porous   media equation
\begin{equation}\label{eq:pme}
\partial_tn-\div(n\nabla n)=0
\end{equation}
or of  the following  regularization of it: 
\begin{equation}\label{eq:pmeps}
\partial_tr-\div(r\nabla K_\ep\ast r)=0.
\end{equation}
These two equations can be guessed after performing the following   \emph{diffusive} change of variables in \eqref{Kepsilon}: 
\begin{equation}\label{eq:diffusive}
\rho(t,x)=\check \rho(\f^{-1} t,x)\andf u(t,x)=\f^{-1}\check u(\f^{-1} t,x).
\end{equation}
We get 
$$\left\{  \begin{array}{l}
\check \rho_t+\div (\check\rho\check u)=0, \\[1ex]
\f^{-2}\bigl(\check u_t + \check u\cdot\nabla \check u\bigr) +\check u
+\nabla K_{\ep}\ast\check \rho=0.\end{array}\right.$$
Hence, it can be expected that 
$\check u +\nabla K_{\ep}\ast\check \rho$ goes to $0$ when $\f$ tends to $\infty.$
Reverting to the mass equation and assuming that 
$\check\rho\to r,$  one can conclude that  $r$ satisfies \eqref{eq:pmeps}. 
In the same way, if both $\f\to\infty$ and $\ep\to0,$ then 
we will prove that $\check \rho\to n$
 with $n$ satisfying  \eqref{eq:pme}. 
\medbreak
The rest of the paper is organized as follows. In Section \ref{s:linear}, we establish a priori estimates for the  linear  System \eqref{eq:eulerconv}. To accommodate the general pressure case \eqref{Kepsilon-PP}, we 
actually replace the term $\nabla K_\ep\ast a$ with $(1+c)\nabla K_\ep\ast a$
for some given function $c$. At first reading however, setting $c\equiv 0$ is advisable, as it  corresponds to Theorems \ref{thm:GWP} and \ref{thm:conv}, 
(see Subsection \ref{s:pressure} for the general case).  The principal outcome, as presented in Theorem \ref{thm:linear}, furnishes a comprehensive estimate crucial for subsequent developments.
Section \ref{s:GWP} is dedicated to proving the existence of solutions. We outline the main steps of the construction procedure, followed by the proof of uniqueness, and ultimately, the convergence to the classical Euler system under the condition $\ep \to 0$. All these aspects rely on the estimates established in Theorem \ref{thm:linear}.
In Section \ref{s:limit}, we delve into relaxation results, as presented in (\ref{eq:pme})--(\ref{eq:diffusive}). We explore two types of relaxation, yielding modifications and classical versions of the porous equation.
 Subsection \ref{ss:com} is devoted to the study of various commutators, essential for our analysis. Subsequently, in Subsection \ref{s:pressure}, we examine the general case of pressure, emphasizing the distinctions from the original scenario.
Lastly, we provide motivation for transitioning from the particle system (\ref{eq:boltz}) to the hydrodynamical equations under consideration \eqref{Euler} and \eqref{Euler-P}.

\section{Study of a suitable linearized system}\label{s:linear}

This part is devoted to proving  a priori estimates for the following linear system: 
\begin{equation}\label{eq:eulerconvbis}\left\{  \begin{array}{l}
         a_t  +v\cdot\nabla a+ (1+b)\,\div u=f, \\[1ex]
         u_t   +  u + v\cdot\nabla u + (1+c)\nabla K_\ep   a=g,\\[1ex]
         a|_{t=0}=a_0, \qquad u|_{t=0}=u_0.
    \end{array}\right.
\end{equation}
Note that the system \eqref{eq:eulerconv} presented before corresponds to the special 
case $c=0.$ The reason for presenting here this more general class of systems
is motivated  by our desire to consider  \eqref{Kepsilon} with more general pressure laws (see Section \ref{s:pressure}). To short the notation from now we write $K_\ep a$ instead of  $K_\ep \ast a$.
\medbreak
Throughout this section, we assume that the (given) triple  $(b,c,v)$  satisfies the relation \eqref{eq:bv}
and the smallness condition\footnote{We assumed \eqref{eq:smallbc} for simplicity. 
A similar result holds if  $0<r\leq 1+b, 1+c < R$ for any real numbers 
$r$ and $R$:  it is just a matter of adapting
the definition of the Lyapunov functional in \eqref{eq:cLj} below, accordingly.} 
\begin{equation}\label{eq:smallbc}
\max\bigl(\|b\|_{L^\infty(0,T\times\R^d)},\|c\|_{L^\infty(0,T\times\R^d)}\bigr)\leq 1/4. 
\end{equation}
\begin{thm}\label{thm:linear} Let $\sigma$ be in the range $(1-d/2,1+d/2].$
Assume   that $a_0$ and $u_0$ are such that 
\begin{equation}\label{eq:data}
u_0\in\dot B^{\sigma}_{2,1}\cap\dot B^{\sigma+1}_{2,1},\quad
a_0\in \dot B^{\sigma-1}_{2,1}\cap\dot B^{\sigma}_{2,1}\andf L_\ep \nabla a_0\in\dot B^{\sigma}_{2,1},
\end{equation}
and that the source terms $f$ and $g$ verify 
\begin{equation}\label{eq:source}
g\in L^1(0,T; \dot B^{\sigma}_{2,1}\cap\dot B^{\sigma+1}_{2,1}),\quad
f \in  L^1(0,T; \dot B^{\sigma-1}_{2,1}\cap\dot B^{\sigma}_{2,1})
\andf L_\ep\star \nabla f \in L^1(0,T;\dot B^{\sigma}_{2,1}).
\end{equation}
Finally, let us assume that the triple $(b,c,v)$ satisfies \eqref{eq:bv} and \eqref{eq:smallbc} and has enough regularity.
Consider a smooth enough solution $(a,u)$ of \eqref{eq:eulerconv} on $[0,T]\times\R^d,$ and set
$$\begin{aligned}
 X^\sigma_{a,u}(t)&:=\|(a,\nabla a,\nabla^2L_\ep a)(t)\|_{\dot B^{\sigma-1}_{2,1}}
+\|(u,\nabla u)(t)\|_{\dot B^{\sigma}_{2,1}}\cr\andf
 H^\sigma_{a,u}(t)&:=\|(u,\nabla u)\|_{\dot B^{\sigma}_{2,1}} 
 +\|(\nabla K_\ep a,\nabla^2 K_\ep a)\|^\ell_{\dot B^{\sigma}_{2,1}}+\|\nabla L_\ep a\|_{\dot B^{\sigma}_{2,1}}^h.\end{aligned}$$
There exists a constant $C$ \emph{independent of $\ep$ and of $T$} such that 
for all $t\in[0,T),$ there holds:
\begin{multline}   \label{eq:linear}
X^\sigma_{a,u}(t)+\int_0^tH^\sigma_{a,u}\,d\tau
\leq C\biggl(X^\sigma_{a,u}(0) +\int_0^t X^\sigma_{f,g}\,d\tau
+\int_0^t \|\nabla v\|_{\dot B^{\frac d2}_{2,1}\cap
\dot B^{\frac d2+1}_{2,1}}X^\sigma_{a,u}\,d\tau\\
+\int_0^t\|b,\nabla b\|_{\dot B^{\frac d2}_{2,1}}\|u,\nabla u\|_{\dot B^{\sigma}_{2,1}}\,d\tau
\\+\int_0^t\Bigl(\|b,\nabla v,\nabla L_\ep b\|_{\dot B^\sigma_{2,1}}\|\nabla u\|_{L^\infty}+\|\nabla v\|_{\dot B^\sigma_{2,1}}\|\nabla L_\ep a\|_{L^\infty}\Bigr)d\tau\\
+\!\int_0^t\!\Bigl(\|c\|_{\dot B^{\frac d2}_{2,1}}
 \bigl(\|\nabla L_\ep a \|^h_{\dot B^{\sigma}_{2,1}} +\|\nabla K_\ep a \|^\ell_{\dot B^{\sigma}_{2,1}}
 +\|\nabla^2 K_\ep a\|^\ell_{\dot B^{\sigma}_{2,1}}\bigr) \\
 + \bigl(\|\nabla c\|_{\dot B^{\frac d2}_{2,1}}
 +\|c_t\!+\!\div((1\!+\!c)v)\|_{L^\infty}\bigr)\|\nabla L_\ep a\|_{\dot B^{\sigma}_{2,1}} \!+\! \|\nabla c\|_{\dot B^{\frac d2}_{2,1}}\|L_\ep a\|_{\dot B^{\sigma}_{2,1}}\\  
+\|c\|_{\dot B^\sigma_{2,1}}\bigl(\|\nabla L_\ep a\|^h_{L^\infty} +\|\nabla K_\ep a\|^\ell_{L^\infty} +\|\nabla^2 K_\ep a\|^\ell_{L^\infty}\bigr) +\|\nabla c\|_{\dot B^\sigma_{2,1}}\|\nabla L_\ep a\|_{L^\infty}\Bigr)d\tau\biggr),
\end{multline}
and the terms involving  the $L^\infty$  norm  of   $\nabla u$
or $\nabla L_\ep a$ and so on,  are not needed if $\sigma\leq d/2.$
\medbreak
If, in addition, $u_0$ belongs to $\dot B^{\sigma-1}_{2,1}$ and $g,$ to $L^1(0,T;\dot B^{\sigma-1}_{2,1})$
then we also have
\begin{multline}   \label{eq:linearb}
\|(u,w)(t)\|_{\dot B^{\sigma-1}_{2,1}} + \int_0^t \|w\|_{\dot B^{\sigma-1}_{2,1}}\,d\tau\lesssim X^\sigma_{a,u}(0)+\|u_0\|^\ell_{\dot B^{\sigma-1}_{2,1}} 
+ \int_0^t\bigl(\|g\|^\ell_{\dot B^{\sigma-1}_{2,1}} +  X^\sigma_{f,g}\bigr)d\tau
\\ +\int_0^t \|\nabla v\|_{\dot B^{\frac d2}_{2,1}\cap
\dot B^{\frac d2+1}_{2,1}}\bigl(\|u\|^\ell_{\dot B^{\sigma-1}_{2,1}} +X^\sigma_{a,u}\bigr)d\tau
+\int_0^t \|c\|_{\dot B^{\frac d2}_{2,1}}\|K_\ep a\|_{\dot B^{\sigma}_{2,1}}\,d\tau
+\hbox{last  5 lines of \eqref{eq:linear}},\end{multline}
where the `damped mode' $w$ is defined by 
\begin{equation}\label{def:w}
    w = u +  \nabla K_\ep a.
\end{equation}
\end{thm}
\begin{proof} In all that follows, $\{c_j\}_{j\in\Z}$   will denote a nonnegative sequence
with sum equal to $1,$ and we use  the notation 
$$z_j:=\ddj z, \qquad j\in\Z,\quad z\in\cS'(\R^d).$$

\subsection{First step: Low frequencies estimates} 

The first step consists in  estimating the `low frequencies' of $a$ at  level of regularity $\sigma-1,$ then of $w$ at level  $\sigma,$ where $w$ has been defined in \eqref{def:w}.
Estimates  for $u$ and $w$ at level $\sigma-1$ (that is, Inequality \eqref{eq:linearb})
are extra informations that can be obtained afterward. 

As a start, we look at the evolution of $(a,w),$ namely,
\begin{equation}\label{eq:aw}
  \left\{  \begin{array}{l}
         \!\!a_t  - \Delta K_\ep  a +v\cdot\nabla a=-\div w-b\.\div u, \\[1ex]
         \!\!w_t   +  w + v\cdot\nabla w =   [v,\nabla K_\ep]\cdot \nabla a - 
          \nabla K_\ep\div(w- \nabla K_\ep a)-\nabla K_\ep(b \. \div u) -c\nabla K_\ep a.
    \end{array}\right.
\end{equation}
Up to lower order terms, this is a diagonalization of System \eqref{eq:eulerconv}: $a$ may be seen as 
a (degenerate) parabolic mode, while $w$ is a damped mode. 
Now, to prove estimates in Besov spaces for $a$ and $w,$ the unavoidable first step
is to localize  \eqref{eq:aw} by means of $\ddj.$ We have:
  \begin{equation}\label{eq:awj}\left\{  \begin{array}{l}
         a_{j,t}  -   \Delta K_\ep  a_j +v\cdot\nabla a_j= [v,\ddj]\cdot\nabla a-\ddj\bigl(\div w+b\.\div u\bigr), \\[2ex]
         w_{j,t}   +  w_j + v\cdot\nabla w_j = [v,\ddj]\cdot\nabla w+  
         \ddj  \bigl([v,\nabla K_\ep]\cdot \nabla a\bigr)\\[1ex]
         \hspace{3cm}
          -  \ddj\nabla K_\ep\div(w-  \nabla K_\ep a)-  \ddj\nabla K_\ep(b\,\div u)
          -\ddj(c\nabla K_\ep a).
    \end{array}\right.\end{equation}

\subsubsection*{Estimate of $a$}

Taking the $L^2$ scalar product of the first equation  of \eqref{eq:awj} with $a_j$ and integrating by parts in the second and third terms on the left yields:
\begin{multline}\label{eq:aj}\frac 12\frac d{dt}\|a_j\|_{L^2}^2+\|\nabla L_\ep a_j\|_{L^2}^2=\frac12\int_{\R^d}(\div v) a_j^2\,dx\\
+\int_{\R^d} ([v,\ddj]\cdot\nabla a)\,a_j\,dx -\int_{\R^d}\ddj\bigl(\div w+b\,\div u\bigr)\,a_j\,dx. \end{multline}
 Provided $-d/2<\sigma-1<d/2+1,$  combining H\"older inequality, embedding
 \eqref{eq:keyembedding} and  the commutator estimate \eqref{com:I} ensures that
$$\frac12\int_{\R^d}(\div v) a_j^2\,dx+\int_{\R^d} ([v,\ddj]\cdot\nabla a)\,a_j\,dx\leq 
Cc_j2^{-j(\sigma-1)}\|\nabla v\|_{\dot B^{\frac d2}_{2,1}}\|a\|_{\dot B^{\sigma-1}_{2,1}}\|a_j\|_{L^2}.$$
Furthermore, for $-d/2<\sigma-1\leq d/2$, Cauchy-Schwarz inequality, 
the product law \eqref{eq:prod1} and the definition of the space 
$\dot B^{\sigma-1}_{2,1}$ guarantee that
$$\int_{\R^d}\ddj(b\,\div u)\,a_j\,dx\leq Cc_j2^{-j(\sigma-1)}
\|b\|_{\dot B^{\frac d2}_{2,1}}\|\div u\|_{\dot B^{\sigma-1}_{2,1}}\|a_j\|_{L^2}.$$
Now, owing to the spectral localization given by $\ddj,$ Bernstein 
inequality and \eqref{eq:Kep}, we have
\begin{equation}\label{eq:bernstein}
\|\nabla L_\ep a_j\|_{L^2}^2\approx 2^{2j}\wh L_\ep^2(2^j)\|a_j\|^2_{L^2}\approx 
\|\Delta K_\ep a_j\|_{L^2}\|a_j\|_{L^2}.
\end{equation}
Hence after `simplification by $\|a_j\|_{L^2}$' in \eqref{eq:aj}, 
integration  on $[0,t]$ and use of \eqref{eq:bernstein}, we get for some absolute 
constant $\kappa_0,$
$$\displaylines{\|a_j(t)\|_{L^2}+ \kappa\int_0^t  \|K_\ep\Delta a_j\|_{L^2}\,d\tau\leq \|a_{0,j}\|_{L^2}+\int_0^t\|\div w_j\|_{L^2}\,d\tau
\hfill\cr\hfill +C2^{-j(\sigma-1)}\int_0^t c_j\bigl(\|\nabla v\|_{\dot B^{\frac d2}_{2,1}}\|a\|_{\dot B^{\sigma-1}_{2,1}}
+\|\div u\|_{\dot B^{\sigma-1}_{2,1}}\|b\|_{\dot B^{\frac d2}_{2,1}}\bigr)d\tau.}$$
Then,  multiplying by $2^{j(\sigma-1)}$ and summing up on all $j$'s such that $2^j\wh L_\ep(2^j)<\nu_0$ gives
\begin{multline}\label{eq:LF1}
\|a(t)\|_{\dot B^{\sigma-1}_{2,1}}^\ell+\kappa_0\int_0^t  \|K_\ep\Delta a\|_{\dot B^{\sigma-1}_{2,1}}^\ell\,d\tau\leq 
\|a_0\|_{\dot B^{\sigma-1}_{2,1}}^\ell\\+\int_0^t\|\div w\|_{\dot B^{\sigma-1}_{2,1}}^\ell\,d\tau+C\int_0^t\bigl(\|\nabla v\|_{\dot B^{\frac d2}_{2,1}}\|a\|_{\dot B^{\sigma-1}_{2,1}}+\|\div u\|_{\dot B^{\sigma-1}_{2,1}}\|b\|_{\dot B^{\frac d2}_{2,1}}\bigr)\,d\tau.
\end{multline}
\subsubsection*{Estimate of $w$ at regularity level $\sigma$}

Let us take the $L^2$ scalar product of the second equation of \eqref{eq:awj}  with $w_j.$ Handling the terms containing to $v$
 as previously, simplifying by $\|w_j\|_{L^2}$ and integrating, we get for 
 $\sigma\in(-d/2,1+d/2],$
$$\displaylines{\|w_j(t)\|_{L^2}+ \int_0^t   \|w_j\|_{L^2}\,d\tau\leq \|w_{0,j}\|_{L^2}+C2^{-j\sigma}
\int_0^tc_j\|\nabla v\|_{\dot B^{\frac d2}_{2,1}}\|w\|_{\dot B^{\sigma}_{2,1}} \,d\tau
\hfill\cr\hfill+\int_0^t  \Bigl(\|\nabla K_\ep\div(w_j- \nabla K_\ep a_j)\|_{L^2} + 
\|\ddj\bigl([v,\nabla K_\ep]\cdot\nabla  a\bigr)\|_{L^2}
+ \|\ddj\nabla K_\ep(b\.\div u)\|_{L^2}\Bigr)d\tau\hfill\cr\hfill
+\int_0^t\|\ddj(c\nabla K_\ep a)\|_{L^2}\,d\tau.}
$$
Hence, multiplying by $2^{j\sigma}$ and summing on all $j$'s such that  $2^j\wh L_\ep(2^j)<\nu_0,$
\begin{multline}\label{eq:LF2}
\|w(t)\|_{\dot B^{\sigma}_{2,1}}^\ell+\int_0^t  \|w\|_{\dot B^{\sigma}_{2,1}}^\ell\,d\tau\leq 
\|w_0\|_{\dot B^{\sigma}_{2,1}}^\ell+\int_0^t  \|\nabla K_\ep\div(w- \nabla K_\ep a)\|_{\dot B^\sigma_{2,1}}^\ell\,d\tau\\
+C\int_0^t\|\nabla v\|_{\dot B^{\frac d2}_{2,1}}\|w\|_{\dot B^{\sigma}_{2,1}} \,d\tau
+\int_0^t\bigl(\|[v,\nabla K_\ep]\cdot\nabla a\|_{\dot B^\sigma_{2,1}}^\ell
+ \|\nabla K_\ep(b\.\div u)\|_{\dot B^\sigma_{2,1}}^\ell
+\|c\nabla K_\ep a\|_{\dot B^\sigma_{2,1}}^\ell\bigr)d\tau.
\end{multline}
Looking at \eqref{eq:fucknot}, we see that
\begin{equation}\label{eq:LF3}
 \|\nabla K_\ep\div(w-  \nabla K_\ep a)\|_{\dot B^{\sigma}_{2,1}}^\ell \leq   C\nu_0^2  \|w-  \nabla  K_\ep a\|_{\dot B^{\sigma}_{2,1}}^\ell.
\end{equation}
For the last term of \eqref{eq:LF2}, using \eqref{eq:prod1} and the low frequency cut-off yields
\begin{equation}\label{eq:LF4}
\|\nabla K_\ep(b\,\div u)\|_{\dot B^{\sigma}_{2,1}}^\ell
\leq C\nu_0^2   \|b\,\div u\|_{\dot B^{\sigma-1}_{2,1}}^\ell
\leq C\nu_0^2  \|b\|_{\dot B^{\frac d2}_{2,1}}\|\div u\|_{\dot B^{\sigma-1}_{2,1}}.
\end{equation}
To handle the commutator term, we use $K_\ep=L_\ep^2.$ This enables us to write that:
$$[v,\nabla K_\ep]\cdot\nabla a=[v,\nabla L_\ep]\cdot\nabla L_\ep a+\nabla L_\ep[v,L_\ep]\cdot\nabla a.$$
Therefore, thanks to Inequalities  \eqref{com:II}  and \eqref{com:IIb}
with $c=v^k$ (for $k=1,\cdots,d$)  and $h=\nabla L_\ep a$ or $\nabla a,$ respectively, we have
$$\begin{aligned}
\|[v,\nabla K_\ep]\cdot\nabla a\|_{\dot B^\sigma_{2,1}}^\ell &\lesssim
\|[v,\nabla L_\ep]\cdot\nabla L_\ep a\|_{\dot B^\sigma_{2,1}}^\ell
+\|\nabla L_\ep[v,L_\ep]\cdot\nabla a\|_{\dot B^\sigma_{2,1}}^\ell\\
&\lesssim
\|[v,\nabla L_\ep]\cdot\nabla L_\ep a\|_{\dot B^\sigma_{2,1}}
+\nu_0\|[v,L_\ep]\cdot\nabla a\|_{\dot B^\sigma_{2,1}}\\
&\lesssim \|\nabla v\|_{\dot B^{\frac d2}_{2,1}}\|\nabla L_\ep a\|_{\dot B^\sigma_{2,1}}+
\|\nabla v\|_{\dot B^{\sigma}_{2,1}}\|\nabla L_\ep a\|_{L^\infty}
+\nu_0\|\nabla v\|_{\dot B^{\frac d2}_{2,1}} \|a\|_{\dot B^\sigma_{2,1}},
\end{aligned}$$
and the second term in the right-hand side is not needed if $\sigma\leq d/2.$
\smallbreak
Finally, we have 
$$\|c\nabla K_\ep a\|_{\dot B^\sigma_{2,1}}^\ell\lesssim 
\|c\|_{\dot B^{\frac d2}_{2,1}}\|\nabla K_\ep a\|_{\dot B^\sigma_{2,1}}+
\|c\|_{\dot B^{\sigma}_{2,1}}\|\nabla K_\ep a\|_{L^\infty}$$
and the second term in the right-hand side is not needed if $\sigma\leq d/2.$
In the end, reverting to 
\eqref{eq:LF2} yields 
\begin{multline}\label{eq:LF8}
\|w(t)\|_{\dot B^{\sigma}_{2,1}}^\ell+\int_0^t  \|w\|_{\dot B^{\sigma}_{2,1}}^\ell\,d\tau\leq 
\|w_0\|_{\dot B^{\sigma}_{2,1}}^\ell+\nu_0^2\int_0^t \|w- \nabla K_\ep a\|_{\dot B^{\sigma}_{2,1}}^\ell\,d\tau\\
+C\int_0^t\bigl(\nu_0^2 \|b\|_{\dot B^{\frac d2}_{2,1}}\|u\|_{\dot B^{\sigma}_{2,1}} +\|c\|_{\dot B^{\frac d2}_{2,1}}\|\nabla K_\ep a\|_{\dot B^{\sigma}_{2,1}}+\|c\|_{\dot B^\sigma_{2,1}}\|\nabla K_\ep a\|_{L^\infty}
\\+\|\nabla v\|_{\dot B^{\frac d2}_{2,1}}\|(w,a,\nabla L_\ep a)\|_{\dot B^{\sigma}_{2,1}}+
\|\nabla L_\ep a\|_{L^\infty}\|v\|_{\dot B^{\sigma+1}_{2,1}}\bigr)d\tau,
\end{multline}
where the terms with $L^\infty$ norms of $\nabla K_\ep a$ or $\nabla L_\ep a$ are  absent if $\sigma\leq d/2.$
\medbreak
Putting this inequality together with \eqref{eq:LF1} allows to absorb all the 
linear terms in the right-hand side provided that $\nu_0$ is chosen small enough.
We get \begin{multline}\label{eq:LF9}
\|a(t)\|_{\dot B^{\sigma-1}_{2,1}}^\ell+\|w(t)\|_{\dot B^{\sigma}_{2,1}}^\ell
+\frac12\int_0^t\|(K_\ep\nabla a,w)\|_{\dot B^{\sigma}_{2,1}}^\ell\,d\tau \lesssim 
\|a_0\|_{\dot B^{\sigma-1}_{2,1}}^\ell+\|w_0\|_{\dot B^{\sigma}_{2,1}}^\ell\\
+\int_0^t\!\Bigl(\|b\|_{\dot B^{\frac d2}_{2,1}}\|u\|_{\dot B^{\sigma}_{2,1}}
+\|c\|_{\dot B^{\frac d2}_{2,1}}\|\nabla K_\ep a\|_{\dot B^{\sigma}_{2,1}}+\|c\|_{\dot B^\sigma_{2,1}}\|\nabla K_\ep a\|_{L^\infty}\\
+\|\nabla v\|_{\dot B^{\frac d2}_{2,1}}\bigl(\|(a,\nabla a,\nabla^2L_\ep a)\|_{\dot B^{\sigma-1}_{2,1}}
+\|w\|_{\dot B^{\sigma}_{2,1}}\bigr)
+\|\nabla L_\ep a\|_{L^\infty}\|v\|_{\dot B^{\sigma+1}_{2,1}}\Bigr)d\tau.
\end{multline}
Again, the terms with $L^\infty$ norms of $\nabla K_\ep a$ or $\nabla L_\ep a$ are
not needed if $\sigma\leq d/2.$

\subsubsection*{Estimate of $w$ at regularity level $\sigma-1$}

Note that we also have 
$$\|\nabla K_\ep\div(w-\nabla K_\ep a)\|_{\dot B^{\sigma-1}_{2,1}}^\ell 
\leq C\bigl(\nu_0^2  \|w\|_{\dot B^{\sigma-1}_{2,1}}^\ell
+c\|\nabla K_\ep a\|_{\dot B^{\sigma}_{2,1}}^\ell\bigr)\cdotp$$
Hence, arguing as for  proving  \eqref{eq:LF8} but using this time that
$$\|\nabla K_\ep(b\,\div u)\|^\ell_{\dot B^{\sigma-1}_{2,1}}\leq\nu_0
\|b\,\div u\|^\ell_{\dot B^{\sigma-1}_{2,1}}\leq C\nu_0
\|b\|_{\dot B^{\frac d2}_{2,1}}\|u\|_{\dot B^{\sigma}_{2,1}},$$
 we  get if $\nu_0$ is small enough:
\begin{multline}\label{eq:LF10}
\|w(t)\|_{\dot B^{\sigma-1}_{2,1}}^\ell
+\frac12\int_0^t  \|w\|_{\dot B^{\sigma-1}_{2,1}}^\ell\,d\tau\leq 
\|w_0\|_{\dot B^{\sigma-1}_{2,1}}^\ell+\nu_0\int_0^t \|\nabla K_\ep a\|_{\dot B^{\sigma}_{2,1}}^\ell\,d\tau\\
+C\int_0^t\bigl(\nu_0^2 \|b\|_{\dot B^{\frac d2}_{2,1}}\|u\|_{\dot B^{\sigma}_{2,1}}
+\|\nabla v\|_{\dot B^{\frac d2}_{2,1}}\|w,a,\nabla L_\ep a\|_{\dot B^{\sigma-1}_{2,1}}\bigr)d\tau
+C\int_0^t \|c\|_{\dot B^{\frac d2}_{2,1}}\|\nabla K_\ep a\|_{\dot B^{\sigma-1}_{2,1}}\,d\tau.
\end{multline}
Note that  for all $\sigma'\in\R,$ we may write
\begin{equation}\label{eq:wu}\|w-u\|_{\dot B^{\sigma'}_{2,1}}^\ell\leq C  \|\nabla K_\ep a\|_{\dot B^{\sigma'}_{2,1}}^\ell \leq C\nu_0^2\|a\|_{\dot B^{\sigma'-1}_{2,1}}.\end{equation}
 This allows  to replace  $\|w\|_{\dot B^{\sigma}_{2,1}}^\ell$ by $\|u\|_{\dot B^{\sigma}_{2,1}}^\ell$
 in the left-hand side of \eqref{eq:LF9}  and 
  $\|w\|_{\dot B^{\sigma-1}_{2,1}}^\ell$ by $\|u\|_{\dot B^{\sigma-1}_{2,1}}^\ell$
in the first term of the left-hand side of \eqref{eq:LF10}, and thus to complete the proof of the low frequency
 parts of \eqref{eq:linear} and \eqref{eq:linearb}.

\begin{rmk}\label{r:fundamental}
 If $c=F(K_\ep a)$ for some smooth function $F$ vanishing at $0,$ 
 the last term of \eqref{eq:LF10}
 lacks time integrability. However, one can apply
\eqref{eq:LF9} with $\sigma-1$ instead of \eqref{eq:LF10} to bound  $u$ and $\nabla K_\ep a$  in $L^1(\R_+;\dot B^{\sigma-1}_{2,1}).$ We deduce a bound 
for  $L_\ep a$ in $L^2(\R_+;\dot B^{\sigma-1}_{2,1})).$ 
 \end{rmk}


\subsection{Second step: High frequencies estimates} 

We adapt the approach of \cite{CBD2} for the dissipative Euler system, 
and introduce the following ``Lyapunov" and ``dissipation rate" functionals

\begin{align}\label{eq:cLj}
\cL_j^2&:=\|(a_j,L_\ep a_j, u_j)\|_{L^2}^2-  2\int_{\R^d} a_j\div u_j\,dx
+2\int_{\R^d}(1\!+\!c)|\nabla L_\ep a_j|^2\,dx\\&\hspace{6cm}+  \int_{\R^d}
|\nabla \cP u_j|^2\,dx+   2\int_{\R^d}
(1\!+\!b)(\div u_j)^2\,dx, \\\label{eq:cHj}
   \cH_j^2&:=  \|u_j\|_{L^2}^2+\|\nabla\cP u_j\|_{L^2}^2+
   \int_{\R^d}(1\!+\!c)|\nabla L_\ep a_j|^2\,dx+ \int_{\R^d} (1\!+\!b)(\div u_j)^2\,dx.\nonumber
\end{align}
Note that, owing to  \eqref{eq:smallbc} and Young inequality, we have 
\begin{equation}\begin{aligned}\label{eq:approx}\cL_j\approx 
\|(\nabla L_\ep a_j,\nabla u_j)\|_{L^2} 
\approx \cH_j\quad\hbox{if }\ 2^j\wh L_\ep(2^j)\geq \nu_0,\andf\\
\cL_j\approx \|(a_j,u_j,\nabla u_j)\|_{L^2},\quad
 \cH_j\gtrsim  2^{j}\wh L_\ep(2^j)\cL_j
 \quad\hbox{if }\ 2^j\wh L_\ep(2^j)\leq \nu_0.
\end{aligned}\end{equation}
Hence, if multiplying $\cH_j$ and $\cL_j$ by $2^{j\sigma}$ and summing up on those $j$'s
such that $2^j\wh L_\ep(2^j)\geq \nu_0,$ one gets  the parts of $X^\sigma_{a,u}$
and $H^\sigma_{a,u}$ corresponding to  the high frequencies.
\medbreak
We  shall implement an energy method so has to compute 
the time derivatives of all the terms that constitute $\cL_j^2.$
To proceed,  the first step is of course to localize \eqref{eq:eulerconv} by means of $\ddj.$
However, in order to avoid loss of derivatives, one needs to be very careful how one writes 
the terms with nonconstant coefficients after localization. The idea is to obtain for $(a_j,u_j)$ a system with the same structure as \eqref{eq:eulerconv}, up to `manageable' commutator terms.
Having this in mind, a suitable way of writing the localized system is
 \begin{equation}\label{eq:eulerconvj}
  \left\{  \begin{array}{l}
         a_{j,t} + v\cdot\nabla a_j + \ddj\bigl((1+b)\div u\bigr)= [v,\ddj]\cdot\nabla a, 
         \\ u_{j,t} + v\cdot\nabla u_j+  u_j + \ddj((1+c)\nabla K_\ep a)
         =[v,\ddj]\cdot\nabla u.\\
    \end{array}\right.
\end{equation}
Let us first  look at  the time derivative of $\|(a_j,u_j)\|_{L^2}^2.$  
Taking the $L^2$ scalar product of \eqref{eq:eulerconvj} with $(a_j,u_j),$ 
and integrating by parts in the convection terms  gives: 
\begin{multline}\label{eq:conv1}\frac12\frac d{dt}\|(a_j,u_j)\|_{L^2}^2 +  \|u_j\|_{L^2}^2+\int_{\R^d}(\Id-K_\ep)a_j\,\div u_j\,dx+\int_{\R^d} u_j\cdot \ddj(c\nabla K_\ep a)\,dx
\\+\int_{\R^d} a_j\ddj(b\,\div u)\,dx
=\int_{\R^d} ([v,\ddj]\cdot\nabla a\,a_j+ ([v,\ddj]\cdot\nabla u)\cdot u_j)\,dx+\frac12\int_{\R^d}(a_j^2+ |u_j|^2)\div v\,dx.
\end{multline}
To eliminate the third term of the left-hand side, we need to look at  $\|L_\ep a_j\|_{L^2}^2.$
We have
$$
  L_\ep a_{j,t} + v\cdot\nabla L_\ep a_j + L_\ep\ddj\bigl((1+b)\div u\bigr)= [v,L_\ep\ddj]\cdot\nabla a.
$$
Hence, 
$$\displaylines{\frac12\frac d{dt}\|L_\ep a_j\|_{L^2}^2 + \int_{\R^d} L_\ep a_j\,L_\ep \div u_j\,dx
+ \int_{\R^d} L_\ep a_j \,L_\ep\ddj(b\,\div u)\,dx\hfill\cr\hfill
=\frac12\int_{\R^d} (L_\ep a_j)^2\div v\,dx+\int_{\R^d} [v,L_\ep\ddj]\cdot\nabla a\: L_\ep a_j\,dx.}$$
Remembering $L_\ep^2=K_\ep$ and ${}^t\!L_\ep=L_\ep,$ and adding up this relation to \eqref{eq:conv1} gives
\begin{equation}\label{eq:conv2}
\frac12\frac d{dt}\|(a_j,L_\ep a_j,u_j)\|_{L^2}^2 +  \|u_j\|_{L^2}^2+
\int_{\R^d} a_j \div u_j \,dx = I^1_j,
\end{equation}
where
\begin{multline*}
I^1_j=-\int_{\R^d} a_j\ddj(b\,\div u)\,dx
-\int_{\R^d} L_\ep a_j \,L_\ep\ddj(b\,\div u)\,dx+\int_{\R^d} [v,\ddj]\cdot\nabla a\,a_j\\+\int_{\R^d} ([v,\ddj]\cdot\nabla u)\cdot u_j\,dx
+\int_{\R^d} [v,L_\ep\ddj]\cdot\nabla a\: L_\ep a_j\,dx\\
+\frac12\int_{\R^d} \bigl(a_j^2+(L_\ep a_j)^2+|u_j|^2\bigr)\div v\,dx
+\int_{\R^d}\ddj(c\nabla K_\ep a)\cdot u_j\,dx.
\end{multline*}
Next,  to show the third term  of $\cH_j^2,$ one can 
compute the time derivative of $(a_j|\div u_j)_{L^2}.$
To do this, it is better to rewrite the equation for $a_j$ as follows: 
$$a_{j,t}+v\cdot\nabla a_j+(1+b)\div u_j=[v,\ddj]\cdot\nabla a
+[b,\ddj]\div u.$$
Then, using the fact that
$$\int_{\R^d} \bigl(a_j\div (v\cdot\nabla u_j)+\div u_j \,v\cdot\nabla a_j\bigr)dx
=\int_{\R^d} a_j\bigl({\rm Tr}(\nabla v\cdot\nabla u_j)-\div v\,\div u_j\bigr)dx,$$
and that
\begin{multline*}
\int_{\R^d} a_j\div\ddj(c\nabla K_\ep a)\,dx
=\int_{\R^d}a_j\div \ddj[c,L_\ep]\nabla L_\ep a\,dx\\
+\int_{\R^d}\nabla L_\ep a_j\cdot[c,\ddj]\nabla L_\ep a\,dx
-\int_{\R^d}c|\nabla L_\ep a_j|^2\,dx,\end{multline*}
we get
\begin{equation}\label{eq:conv3}
\frac d{dt} \int_{\R^d} a_j \div u_j\,dx +    \int_{\R^d} a_j \div u_j\,dx +\int_{\R^d}(1\!+\!b)(\div u_j)^2\,dx 
-\int_{\R^d}(1\!+\!c)|\nabla L_\ep a_j|^2\,dx=I^2_j,
\end{equation}
where
\begin{multline*}
I^2_j
= \int_{\R^d} a_j\bigl(\div v\,\div u_j-{\rm Tr}(\nabla v\cdot\nabla u_j)\bigr)dx
+\int_{\R^d}a_j\div \ddj[L_\ep,c]\nabla L_\ep a\,dx\\
+\int_{\R^d}\nabla L_\ep a_j\cdot[\ddj,c]\nabla L_\ep a\,dx
+\int_{\R^d}\bigl(a_j\,\div [v,\ddj]\cdot\nabla u +  \div u_j ( [v,\ddj]\nabla a+[b,\ddj]\div u)\bigr)dx.\end{multline*}
So, subtracting \eqref{eq:conv3} from \eqref{eq:conv2} eventually yields 
\begin{multline}\label{eq:conv2aa}
\frac12\frac d{dt}\Bigl(\|(a_j,L_\ep a_j,u_j)\|_{L^2}^2 -2\int_{\R^d} a_j \div u_j \, dx\Bigr) + \|u_j\|_{L^2}^2+ \int_{\R^d}(1\!+\!c)|\nabla L_\ep a_j|^2\,dx\\
-\int_{\R^d}(1\!+\!b)(\div u_j)^2\,dx 
 = I^1_j- I^2_j.
\end{multline}
Next, to handle the term with  $\|\nabla \cP u_j\|_{L^2}^2,$ we   apply  $\ddj\cP$ to the velocity equation and get
$$\cP u_{j,t}+\cP u_j +v\cdot\nabla\cP u_j=[v,\ddj\cP]\cdotp\nabla u-\cP\ddj(c\nabla K_\ep a),$$
which immediately implies after taking the scalar product with $\cP u_j$
and integrating by parts in the convection term:
\begin{multline}\label{eq:conv4}
\frac12\frac d{dt}\|\nabla\cP u_j\|_{L^2}^2+ \|\nabla\cP u_j\|_{L^2}^2 =I_j^3:=
\frac12\int_{\R^d}\div v\,|\nabla\cP u_j|^2\,dx \\-\int_{\R^d}\!\! {\rm Tr}(\nabla\cP u_j\cdot D\cP u_j\cdot Dv)\,dx
+\int_{\R^d} \!\!\nabla\cP u_j\cdot(\nabla [v,\ddj\cP]\cdotp\nabla u)\,dx
-\int_{\R^d}\!\!\nabla\cP u_j\cdot \nabla \cP\ddj(c\nabla K_\ep a)\,dx.
\end{multline}
Let us finally compute  the time derivative of 
$\|\sqrt{1\!+\! c}\,\nabla L_\ep a_j\|_{L^2}^2+\|\sqrt{1\!+\!b}\,\div u_j\|_{L^2}^2.$
First, taking the $L^2$ scalar product of the following relation
$$\nabla L_\ep a_{j,t}+\nabla L_\ep\ddj((1\!+\!b)\,\div u)+\nabla L_\ep\ddj (v\cdot\nabla a)=0$$
with $(1+ c)\nabla L_\ep a_j,$ and using the fact that
\begin{multline*}
\int_{\R^d}(1\!+\!c)\nabla L_\ep\ddj (v\cdot\nabla a)\cdot\nabla L_\ep a_j\,dx
=\int_{\R^d}(1\!+\!c)[\nabla L_\ep\ddj, v\cdot\nabla] a\cdot\nabla L_\ep a_j\,dx
\\-\frac12\int_{\R^d}|\nabla L_\ep a_j|^2\div\bigl((1\!+\! c)v\bigr)\,dx,\end{multline*}
we find that
\begin{multline}\label{eq:conv10}
\frac12\frac d{dt}\biggl(\int_{\R^d}(1\!+\! c)|\nabla L_\ep a_j|^2\,dx
-\int_{\R^d}\bigl(c_t+\div\bigl((1\!+\! c)v\bigr) |\nabla L_\ep a_j|^2\,dx\biggr)\\
+\int_{\R^d} (1\!+\!c)\nabla L_\ep\ddj((1\!+\!b)\div u)\cdot\nabla L_\ep a_j\,dx
+\int_{\R^d}(1\!+\!c)[\nabla L_\ep\ddj,v\cdot\nabla] a)\cdot\nabla L_\ep a_j\,dx=0.
\end{multline}
Next,  because
$$\div u_{j,t}+\div u_j+\div((1+c)\nabla K_\ep a_j)+\div (v\cdot\nabla u_j)=\div[v,\ddj]\cdot\nabla u+\div[c,\ddj]\cdot\nabla K_\ep a,$$
we discover that
\begin{multline}\label{eq:conv11}\frac12\biggl(\frac d{dt} \int_{\R^d}(1+b)(\div u_j)^2 dx
- \int_{\R^d} b_t\, (\div u_j)^2 \,dx \biggr)
+\int_{\R^d}(1+b)(\div u_j)^2\,dx 
\\+\int_{\R^d} ((1+b)\div u_j)
\div ( (1\!+\!c) \nabla K_\ep a_j)\,dx
+\int_{\R^d}(1+b)\div u_j\,\div (v\cdot\nabla u_j)\,dx\\ =\int_{\R^d}(1+b)\div u_j \,\div[v,\ddj]\cdot\nabla u\,dx.\end{multline}
Due to \eqref{eq:bv}, we have
$$
\int_{\R^d}(1\!+\!b)(\div u_j)\div(v\cdot\nabla u_j)\,dx-\frac12 \int_{\R^d}b_t(\div u_j)^2\,dx=\int_{\R^d} (1+b)\div u_j {\rm Tr}(Dv\cdot Du_j)\,dx.$$
Hence adding up \eqref{eq:conv10} and \eqref{eq:conv11} leads to 
\begin{equation}\label{eq:conv12}
\frac12\frac d{dt} \int_{\R^d} \bigl(|\nabla L_\ep a_j|^2 +(1+b)(\div u_j)^2\bigr)dx
+\int_{\R^d}(1+b)(\div u_j)^2\,dx=\sum_{i=0}^6 R_j^i
\end{equation}
$$\begin{aligned}
\with  R_j^0&:=\int_{\R^d}(1\!+\!b)\div u_j\,\div[c,\ddj]\nabla K_\ep a\,dx,\\
R_j^1&:=\int_{\R^d}(1\!+\!c)\nabla L_\ep[b,\ddj]\div u\cdot\nabla L_\ep a_j\,dx,\\
R_j^2&:=-\int_{\R^d}(1\!+\!c)
[\nabla L_\ep\ddj,v\cdot\nabla a]\cdot\nabla L_\ep a_j\,dx,\\
R_j^3&:=-\int_{\R^d} (1+b)\div u_j{\rm Tr}(Dv\cdot Du_j)\,dx ,\\
R_j^4&:=\int_{\R^d}(1+b)\: \div u_j\div[v,\ddj]\cdot\nabla u \,dx\\
R_j^5&:=\int_{\R^d}(1+b)\: \div u_j\div[L_\ep,c]\nabla L_\ep a_j\,dx.
\\{\andf R_j^6 }& := \frac12\int_{\R^d}
\bigl(c_t+\div\bigl((1\!+\!c)v\bigr)|\nabla L_\ep a_j|^2\,dx.
\end{aligned}$$
To get \eqref{eq:conv12}, the key point is 
the following cancellation property between the   
third term  of \eqref{eq:conv10} and the fourth term of  \eqref{eq:conv11}:
\begin{align*}
   \int_{\R^d} ((1\!+\!b)\div u_j)
\div ( (1\!+\!c) \nabla K_\ep a_j)\,dx &=
-R_j^5 - \int_{\R^d}( (1\!+\!c) \nabla L_\ep ((1+b)\div u_j)\cdot
\nabla L_\ep a_j) \,dx 
\\=-R_j^5&-R_j^1 -\int_{\R^d} (1\!+\!c) \nabla L_\ep \ddj ((1+b)\div u)
\cdot \nabla L_\ep a_j) \,dx.
\end{align*}
So finally, adding \eqref{eq:conv4} and twice  \eqref{eq:conv12} to \eqref{eq:conv2aa}, we discover that
\begin{equation}\label{eq:conv13}
    \frac{1}{2} \frac{d}{dt} \cL^2_j + \cH^2_j = I_j^1-I_j^2
    +I_j^3+2\sum_{i=0}^5 R_j^i.
\end{equation}
Now, it is just a matter of  bounding all the terms of the right-hand side. 
The most tricky part is  to estimate the commutator terms in $I_j^1,$  $R^1_j$ and $R^2_j.$ 
For expository purpose, we admit these estimates, the reader
being referred to  Subsection \ref{ss:com} for the proof.

\subsubsection*{Estimating $I_j^1$}
From H\"older inequality, we  have
$$\displaylines{I^1_j\leq \|a_j\|_{L^2}\|\ddj(b\,\div u)\|_{L^2}
+\|L_\ep a_j\|_{L^2} \|L_\ep\ddj(b\,\div u)\|_{L^2}
+\|[\ddj,v]\cdot\nabla a\|_{L^2}\|a_j\|_{L^2}\hfill\cr\hfill
+\|u_j\|_{L^2}\|\ddj(c\nabla K_\ep a)\|_{L^2}
+\| [v,\ddj]\cdot\nabla u\|_{L^2}\|u_j\|_{L^2}\hfill\cr\hfill
+\| [L_\ep\ddj, v]\cdot\nabla a\|_{L^2} \|L_\ep a_j\|_{L^2}
+\frac12(\|a_j\|_{L^2}^2+\|L_\ep a_j\|_{L^2}^2+\|u_j\|_{L^2}^2)\|\div v\|_{L^\infty}.}$$
The terms with $b\,\div u$ may be bounded thanks to the product laws 
\eqref{eq:prod1} and \eqref{eq:prod2} with $f=b$ and $g=\div u,$
and the commutators,  
by means of \eqref{com:I} and \eqref{com:III}.

In the end, we get
\begin{multline}\label{eq:Ij1}
I^1_j\leq Cc_j2^{-j\sigma}\bigl(\|\nabla v\|_{\dot B^{\frac d2}_{2,1}}\|(a,u)\|_{\dot B^{\sigma}_{2,1}}+\|b\|_{\dot B^{\frac d2}_{2,1}}\|\div u\|_{\dot B^{\sigma}_{2,1}}
+\|c\|_{\dot B^{\frac d2}_{2,1}}\|\nabla K_\ep a\|_{\dot B^\sigma_{2,1}}
\\ +\|c\|_{\dot B^\sigma_{2,1}}\|\nabla K_\ep a\|_{L^\infty}
+\|\div u\|_{L^\infty}\|b\|_{\dot B^{\sigma}_{2,1}}\bigr)\|(a_j,u_j)\|_{L^2},
\end{multline}
and the last two terms are not needed if $\sigma\leq d/2.$

\subsubsection*{Estimating $I_j^2$}

From H\"older  inequality, we infer that
\begin{multline*}I^2_j
\leq\|a_j\|_{L^2}\bigl(\|\nabla u_j\|_{L^2}\|\nabla v\|_{L^\infty}
+\|\div[v,\ddj]\cdot\nabla u\|_{L^2}
+\|\div \ddj[L_\ep,c]\nabla L_\ep a\|_{L^2}\bigr)\\
+\|\nabla L_\ep a_j\|_{L^2}\|[\ddj,c]\nabla L_\ep a\|_{L^2}
+ \|\div u_j\|_{L^2}\bigl(
\|[v,\ddj]\nabla a\|_{L^2}
+\|[b,\ddj]\div u\|_{L^2}\bigr)\cdotp\end{multline*}
The last two terms may be bounded by means of \eqref{com:I}.
For the one with $\div [v,\ddj]\cdot\nabla u,$ we use the
decomposition 
$$\partial_k [v,\ddj]\cdot\nabla u=\partial_k v\cdot\nabla u_j
-\ddj(\partial_kv\cdot\nabla u)+[v,\ddj]\nabla\partial_k u.$$
The $L^2$ norm of the
 last term of the right-hand side may be bounded according to \eqref{com:I}.
 The $L^2$ norm of the  first one is obviously bounded by $\|\nabla v\|_{L^\infty}\|\nabla u_j\|_{L^2}.$ 
 To bound the second term, one 
can take advantage of the  product laws \eqref{eq:prod1} and \eqref{eq:prod2}
with $f=\partial_kv$ and $g=\nabla u.$
Finally, thanks to \eqref{com:IIb}, we have
$$
\begin{aligned}
\|\div \ddj[L_\ep,c]\nabla L_\ep a\|_{L^2}&\lesssim 
c_j2^{-j\sigma}\|\div[L_\ep,c]\nabla L_\ep a\|_{\dot B^\sigma_{2,1}}\\
&\lesssim 
c_j2^{-j\sigma}\|[L_\ep,c]\nabla L_\ep a\|_{\dot B^{\sigma+1}_{2,1}}\\
&\lesssim c_j2^{-j\sigma}
\bigl(\|\nabla c\|_{\dot B^{\frac d2}_{2,1}}
\|L_\ep a\|_{\dot B^{\sigma+1}_{2,1}}+
\|\nabla c\|_{\dot B^{\sigma}_{2,1}}
\|\nabla L_\ep a\|_{L^\infty}\bigr)\end{aligned}$$
and, by virtue of  \eqref{com:I},
$$\|[\ddj,c]\nabla L_\ep a\|_{L^2}\lesssim c_j2^{-j\sigma}
\|\nabla c\|_{\dot B^{\frac d2}_{2,1}}\|L_\ep a\|_{\dot B^\sigma_{2,1}}
.$$
This  leads to 
\begin{multline}\label{eq:Ij2}
I_j^2\lesssim c_j2^{-j\sigma}\Bigl(
\|\div u_j\|_{L^2}
\bigl(\|\nabla v\|_{\dot B^{\frac d2}_{2,1}}
\|a\|_{\dot B^{\sigma}_{2,1}}+
\|\nabla b\|_{\dot B^{\frac d2}_{2,1}}
\|u\|_{\dot B^{\sigma}_{2,1}}\bigr)\\+
\|a_j\|_{L^2} \bigl(\|\nabla v\|_{\dot B^{\frac d2}_{2,1}}\|\nabla u\|_{\dot B^{\sigma}_{2,1}}
+  \|\nabla v\|_{\dot B^{\sigma}_{2,1}}\|\nabla u\|_{L^\infty}
+\|\nabla c\|_{\dot B^{\frac d2}_{2,1}}
\|L_\ep a\|_{\dot B^{\sigma+1}_{2,1}}+
\|\nabla c\|_{\dot B^{\sigma}_{2,1}}
\|\nabla L_\ep a\|_{L^\infty}\bigr)\\
+\|\nabla L_\ep a_j\|_{L^2}
\|\nabla c\|_{\dot B^{\frac d2}_{2,1}}\|L_\ep a\|_{\dot B^{\sigma}_{2,1}}\Bigr)\cdotp\end{multline}
and the  terms with $L^\infty$ norms are not needed if $\sigma\leq d/2.$

\subsubsection*{Estimating $I_j^3$}

We start with the obvious inequality:
$$I_j^3\leq C\|\nabla\cP u_j\|_{L^2}
\bigl(\|\nabla v\|_{L^\infty}\|\nabla\cP u_j\|_{L^2}
+\|\nabla[v,\ddj\cP]\cdot\nabla u\|_{L^2}
+\|\nabla\cP\ddj(c\nabla K_\ep a)\|_{L^2}\bigr)$$
and  write that
$$\partial_k[v,\ddj\cP]\cdot\nabla u=\partial_k v\cdot\nabla\cP  u_j
-\ddj\cP(\partial_kv\cdot\nabla u)+[v,\ddj\cP]\nabla\partial_k u.$$
The commutator with $v$ may be bounded as the similar term in $I_j^2.$ As for the last term of $I_j^3$, we observe that, since $\cP\nabla=0,$ we have
$$\nabla\cP(c\nabla K_\ep a)=[\nabla\cP,c]\nabla K_\ep a.$$
Hence this term may be handled by \eqref{com:II} with 
the constant operator $\cP$ instead of $L_\ep.$
We end up with
\begin{multline}\label{eq:Ij3}
I_j^3\leq Cc_j2^{-j\sigma}\bigl( 
\|\nabla v\|_{\dot B^{\frac d2}_{2,1}}\|\nabla u\|_{\dot B^{\sigma}_{2,1}}
+  \|\nabla v\|_{\dot B^{\sigma}_{2,1}}\|\nabla u\|_{L^\infty}\\
+ \|c\|_{\dot B^{\frac d2+1}_{2,1}}\|\nabla K_\ep a\|_{\dot B^\sigma_{2,1}}
+ \|c\|_{\dot B^{\sigma+1}_{2,1}}\|\nabla K_\ep a\|_{L^\infty}
\bigr)\|\nabla\cP u_j\|_{L^2},\end{multline}
and the terms with the $L^\infty$ norm are not needed if $\sigma\leq d/2.$

\subsubsection*{Estimating $R_j^0$}

To bound this term, it suffices to apply H\"older inequality
then to use \eqref{com:IV} with $L_0$ (that is, the identity 
operator), $b=c$ and $z=\nabla K_\ep a.$ We get
\begin{equation}\label{eq:Rj0}
R_j^0 \leq Cc_j2^{-j\sigma}
\|\div u_j\|_{L^2}\bigl(\|\nabla c\|_{\dot B^{\frac d2}_{2,1}}\|\nabla K_\ep a\|_{\dot B^{\sigma}_{2,1}}
+\|\nabla c\|_{\dot B^{\sigma}_{2,1}}\|\nabla K_\ep a\|_{L^\infty}
\bigr)\end{equation}
and the last term is  not needed if $\sigma\leq d/2.$

\subsubsection*{Estimating  $R_j^1$}

Remembering that $\|c\|_{L^\infty}$ is small and applying Inequality  \eqref{com:IV} to $z=\div u,$ we readily get
\begin{equation}\label{eq:Rj1}
R_j^1 \leq Cc_j2^{-j\sigma} \|\nabla L_\ep a_j\|_{L^2}
\bigl(\|\nabla b\|_{\dot B^{\frac d2}_{2,1}}\|\div u\|_{\dot B^{\sigma}_{2,1}}
+\|\div u\|_{L^\infty}\|\nabla L_\ep b\|_{\dot B^{\sigma}_{2,1}}
\bigr)\end{equation}
and the last term is  not needed if $\sigma\leq d/2.$

\subsubsection*{Estimating   $R^2_j$}  
Because we strive for bounds that are independent of $\ep,$
we   have to assume that $\nabla^2 v \in L^1(\R_+;L^\infty),$
that is one more space derivative than for the classical compressible Euler system. 
Now, leveraging  Inequality \eqref{eq:Rj2a}, we readily get
\begin{equation}\label{eq:Rj2}
 R_j^2 \leq Cc_j2^{-j\sigma} \|\nabla L_\ep a_j\|_{L^2}
\bigl(\|a\|_{\dot B^{\sigma}_{2,1}}\|v\|_{\dot B^{\frac d2+2}_{2,1}}\\
+\|v\|_{\dot B^{\frac d2+1}_{2,1}}\|\nabla L_\ep a\|_{\dot B^{\sigma}_{2,1}}\bigr)\cdotp\end{equation}

\subsubsection*{Estimating  $R_j^3$}  

Under assumption \eqref{eq:smallbc}, it is obvious that
\begin{equation}\label{eq:Rj3}R_j^3\lesssim \|\nabla v\|_{L^\infty}\|\nabla u_j\|_{L^2}^2
\lesssim c_j2^{-j\sigma}\|\nabla u_j\|_{L^2}\|\nabla u\|_{\dot B^\sigma_{2,1}}
\|\nabla v\|_{\dot B^{\frac d2}_{2,1}}.\end{equation}

\subsubsection*{Estimating  $R_j^4$}  

With the summation convention on repeated indices, we have
$$\div[v^m,\ddj]\partial_m u =\partial_k v^m \partial_m u^k_j-\ddj(\partial_kv^m\,\partial_m u^k)+[v^m,\ddj]\partial_m\div u.$$
So, using \eqref{com:I}  as well as  product laws \eqref{eq:prod1} and \eqref{eq:prod2}, we get
\begin{equation}\label{eq:Rj4}
R_j^4 \leq Cc_j2^{-j\sigma} \|\div u_j\|_{L^2}
\bigl(\|\nabla v\|_{\dot B^{\frac d2}_{2,1}}\|\nabla u\|_{\dot B^{\sigma}_{2,1}}
+\|\nabla u\|_{L^\infty}\|\nabla v\|_{\dot B^{\sigma}_{2,1}}\bigr),
\end{equation}
and the second term  may be omitted  if $\sigma\leq d/2.$

\subsubsection*{Estimating $R_j^5$}
By H\"older inequality and Condition \eqref{eq:smallbc}, we have
$$R_j^5\lesssim \|\div u_j\|_{L^2}\|\div[L_\ep,c]\nabla L_\ep a_j\|_{L^2}.$$
We observe that 
$$\div[L_\ep,c]\nabla L_\ep a_j=L_\ep\bigl(\nabla c\cdot\nabla L_\ep a_j\bigr)
-\nabla c\cdot L_\ep\nabla L_\ep a_j+[L_\ep,c]\Delta L_\ep a_j.$$
On the one hand, due to \eqref{eq:Kep}, it is obvious that 
$$\|L_\ep(\nabla c\cdot\nabla L_\ep a_j)\|_{L^2}
+\|\nabla c\cdot L_\ep\nabla L_\ep a_j\|_{L^2}\lesssim 
\|\nabla c\|_{L^\infty}\|\nabla L_\ep a_j\|_{L^2}.$$
On the other hand, \eqref{com:V} and Bernstein inequality guarantee that
$$\|[L_\ep,c]\Delta L_\ep a_j\|_{L^2}\lesssim 
\|\nabla c\|_{L^\infty}\|\nabla L_\ep a_j\|_{L^2}.$$
Hence 
\begin{equation}\label{eq:Rj5}
R_j^5 \leq Cc_j2^{-j\sigma} \|\div u_j\|_{L^2}
\|\nabla c\|_{L^\infty}\|\nabla L_\ep a\|_{\dot B^{\sigma}_{2,1}}.\end{equation}

\subsubsection*{Estimating  $R_j^6$}  Finally, we have  
\begin{align}\label{eq:Rj6}
R_j^6&\leq \frac12\| c_t\!+\!\div((1\!+\! c)v)\|_{L^\infty}
\|\nabla L_\ep a_j\|_{L^2}^2\nonumber\\&\leq \frac{c_j}2 \:2^{-j\sigma} 
\|c_t\! + \!\div((1\!+\! c)v)\|_{L^\infty}
\|\nabla L_\ep a_j\|_{L^2}\|\nabla L_\ep a\|_{\dot B^\sigma_{2,1}}.
\end{align}

\subsubsection*{Third step: Putting everything together}

Plugging \eqref{eq:Ij1}, \eqref{eq:Ij2}, \eqref{eq:Ij3}, \eqref{eq:Rj0},  \eqref{eq:Rj1},  \eqref{eq:Rj2},  \eqref{eq:Rj3}, \eqref{eq:Rj4}, \eqref{eq:Rj5} and \eqref{eq:Rj6}
in \eqref{eq:conv13} and remembering \eqref{eq:approx}, we arrive (for some universal constant
$\kappa_0$) at 
\begin{multline*}
  \frac12\frac d{dt}\cL_j^2+\kappa_0\min (1,2^{2j}\wh K_\ep(2^j))\cL_j^2\leq C c_j2^{-j\sigma}\cL_j 
  \Bigl(\|v\|_{\dot B^{\frac d2+1}_{2,1}}\|(a,\nabla L_\ep a, u,\nabla u)\|_{\dot B^{\sigma}_{2,1}}\\
  + \|(b,\nabla b)\|_{\dot B^{\frac d2}_{2,1}}\|\div u\|_{\dot B^{\sigma}_{2,1}}
  +\|\nabla b\|_{\dot B^{\frac d2}_{2,1}}\|u\|_{\dot B^{\sigma}_{2,1}}
  +\|v\|_{\dot B^{\frac d2+2}_{2,1}}\|a\|_{\dot B^\sigma_{2,1}}
   +\|\nabla u\|_{L^\infty}\|(b,\nabla L_\ep b,\nabla v)\|_{\dot B^{\sigma}_{2,1}}\\
   \bigl(\|\nabla c\|_{\dot B^{\frac d2}_{2,1}}+\|c_t\!+\!\div((1\!+\!c)v)\|_{L^\infty}\bigr)\|\nabla L_\ep a\|_{\dot B^{\sigma}_{2,1}}   
+ \|\nabla c\|_{\dot B^{\frac d2}_{2,1}}\|L_\ep a\|_{\dot B^{\sigma}_{2,1}}\\
+ \|c\|_{L^\infty}\|\nabla K_\ep a\|_{\dot B^\sigma_{2,1}}
+ \|c\|_{\dot B^\sigma_{2,1}}\|\nabla K_\ep a\|_{L^\infty} +\|\nabla c\|_{\dot B^\sigma_{2,1}}\|\nabla L_\ep a\|_{L^\infty})\Bigr), \end{multline*}
   where  $\kappa_0$ only depends on $c,$ and
     the terms involving  $L^\infty$ norms of $L_\ep a$ or $u$
          are not needed  if $\sigma\leq d/2.$
\medbreak   
Then,  simplifying by $\cL_j$ and integrating on $[0,t]$ yields
\begin{multline*}
\cL_j(t)+ \kappa_0\min (1,2^{2j}\wh K_\ep(2^j))\!\int_0^t\!\cL_j\,d\tau \leq  \cL_j(0)+ C 2^{-j\sigma}
\!\int_0^t\! c_j\Bigl(\|v\|_{\dot B^{\frac d2+1}_{2,1}}\|(a,\nabla L_\ep a, u,\nabla u)\|_{\dot B^{\sigma}_{2,1}}\\
  + \|b,\nabla b\|_{\dot B^{\frac d2}_{2,1}}\|\div u\|_{\dot B^{\sigma}_{2,1}}
  \!+\|\nabla b\|_{\dot B^{\frac d2}_{2,1}}\|u\|_{\dot B^{\sigma}_{2,1}}
  +\|v\|_{\dot B^{\frac d2\!+\!2}_{2,1}}\|a\|_{\dot B^\sigma_{2,1}}
   \!+\|\nabla u\|_{L^\infty}\|(b,\nabla L_\ep b,\nabla v)\|_{\dot B^{\sigma}_{2,1}}
 \Bigr)d\tau\\
   +C2^{-j\sigma}\int_0^tc_j\Bigl(\bigl(\|\nabla c\|_{\dot B^{\frac d2}_{2,1}}+\|c_t\!+\!\div((1\!+\!c)v)\|_{L^\infty}\bigr)\|\nabla L_\ep a\|_{\dot B^{\sigma}_{2,1}} + \|\nabla c\|_{\dot B^{\frac d2}_{2,1}}\|L_\ep a\|_{\dot B^{\sigma}_{2,1}}
\Bigr)d\tau\\  
+C2^{-j\sigma}\int_0^tc_j\Bigl(\|c\|_{L^\infty}\|\nabla K_\ep a\|_{\dot B^\sigma_{2,1}}
+ \|c\|_{\dot B^\sigma_{2,1}}\|\nabla K_\ep a\|_{L^\infty} +\|\nabla c\|_{\dot B^\sigma_{2,1}}\|\nabla L_\ep a\|_{L^\infty}
\Bigr)d\tau. \end{multline*}
 Multiplying by $2^{j\sigma},$ summing up on all $j\in\Z$ and using again \eqref{eq:approx}, 
we conclude that
\begin{multline}\label{eq:almostfinal}
\|(a,u,\nabla u)(t)\|_{\dot B^{\sigma}_{2,1}}\!+\! 
\|\nabla L_\ep a(t)\|^h_{\dot B^{\sigma}_{2,1}}\!+\! 
\int_0^t\Bigl(\|(\nabla^2 K_\ep u, \nabla^3 K_\ep u)\|_{\dot B^{\sigma}_{2,1}}^\ell\!+\!
\|(u,\nabla u)\|^h_{\dot B^{\sigma}_{2,1}}\Bigr)d\tau \\+
\int_0^t\bigl(\|\nabla^2K_\ep a\|_{\dot B^{\sigma}_{2,1}}^\ell+\|\nabla L_\ep a\|_{\dot B^{\sigma}_{2,1}}^h\bigr)d\tau 
\lesssim \|(a_0,u_0,\nabla u_0)\|_{\dot B^{\sigma}_{2,1}}+ 
\|\nabla L_\ep a_0\|^h_{\dot B^{\sigma}_{2,1}}\\
+\int_0^t\Bigl(\|v\|_{\dot B^{\frac d2+1}_{2,1}}\|(a,\nabla L_\ep a, u,\nabla u)\|_{\dot B^{\sigma}_{2,1}}
  + \|b,\nabla b\|_{\dot B^{\frac d2}_{2,1}}\|\div u\|_{\dot B^{\sigma}_{2,1}}\\
  +\|\nabla b\|_{\dot B^{\frac d2}_{2,1}}\|u\|_{\dot B^{\sigma}_{2,1}}
  +\|\nabla v\|_{\dot B^{\frac d2+1}_{2,1}}\|a\|_{\dot B^\sigma_{2,1}}
   +\|\nabla u\|_{L^\infty}\|(b,\nabla L_\ep b,\nabla v)\|_{\dot B^{\sigma}_{2,1}}
 \Bigr)d\tau\\  
  +\int_0^t\Bigl(\bigl(\|\nabla c\|_{\dot B^{\frac d2}_{2,1}}+\|c_t\!+\!\div((1\!+\!c)v)\|_{L^\infty}\bigr)\|\nabla L_\ep a\|_{\dot B^{\sigma}_{2,1}} + \|\nabla c\|_{\dot B^{\frac d2}_{2,1}}\|L_\ep a\|_{\dot B^{\sigma}_{2,1}}
\Bigr)d\tau\\  
+\int_0^t\Bigl(\|c\|_{L^\infty}\|\nabla K_\ep a\|_{\dot B^\sigma_{2,1}}
+ \|c\|_{\dot B^\sigma_{2,1}}\|\nabla K_\ep a\|_{L^\infty} +\|\nabla c\|_{\dot B^\sigma_{2,1}}\|\nabla L_\ep a\|_{L^\infty}\Bigr)d\tau.
\end{multline}
where the terms with $L^\infty$ norms of $\nabla L_\ep a$ or $u$ are not needed if
$\sigma\leq d/2.$
\medbreak
Let us finally exhibit the 
$L^1$-in-time control for $\|\nabla u\|^\ell_{\dot B^\sigma_{2,1}}$ 
(note that it is not a consequence of the control of $\|u\|^\ell_{\dot B^\sigma_{2,1}}$
since `low' frequencies need not to be low !). 
We write that
$$\partial_k(\partial_ku)+\partial_ku+v\cdot\nabla\partial_ku=-\partial_kv\cdot \nabla u
-\partial_k((1\!+\!c)\nabla K_\ep a).$$
So, localizing this equation by means of $\ddj$ then repeating essentially the same arguments as
before (here we only need the most basic commutator estimate \eqref{com:I}), we arrive at
\begin{multline}\label{eq:nablau}
\|\nabla u(t)\|_{\dot B^\sigma_{2,1}}^\ell + \int_0^t\|\nabla u\|_{\dot B^\sigma_{2,1}}^\ell\,d\tau
\leq \|\nabla u_0\|_{\dot B^\sigma_{2,1}}^\ell + \int_0^t\|\nabla((1\!+\!c)\nabla K_\ep a)\|_{\dot B^\sigma_{2,1}}^\ell\,d\tau\\+ C\int_0^t\|\nabla v\|_{\dot B^{\frac d2}_{2,1}}\|\nabla u\|_{\dot B^{\sigma}_{2,1}}\,d\tau
+ C\int_0^t\|\nabla u\|_{L^\infty}\|\nabla v\|_{\dot B^{\sigma}_{2,1}}\,d\tau,
\end{multline}
where, as usual,  the last term is not needed if $\sigma\leq d/2.$
\smallbreak
Note that $\|\nabla^2 K_\ep a\|_{\dot B^\sigma_{2,1}}^\ell$ can be
 controlled from \eqref{eq:almostfinal}.
To handle the term $\nabla(c\nabla K_\ep a)$, we use the decomposition
(recall notation \eqref{eq:fucknotbis}):
$$\partial_k(c\,\partial_m K_\ep a)=\partial_k c\,\partial_m K_\ep a+
c\partial_k\partial_m K_\ep a^\ell+
[c,\partial_k L_\ep]\partial_m L_\ep a^h
+\partial_k L_\ep(c\,\partial_m L_\ep a^h).$$
Taking advantage  of \eqref{eq:prod1}, \eqref{eq:prod2} and \eqref{com:II}, and using the low frequency cut-off in the last term,  we get (with the usual convention if $\sigma\leq d/2$):
$$\begin{aligned}
\|\partial_k c\,\partial_m K_\ep a\|_{\dot B^\sigma_{2,1}}&\lesssim 
\|\partial_k c\|_{\dot B^{\frac d2}_{2,1}}\|\partial_m K_\ep a\|_{\dot B^\sigma_{2,1}}+
\|\partial_k c\|_{\dot B^\sigma_{2,1}}\|\partial_m K_\ep a\|_{L^\infty},\\
\|c\partial_k\partial_m K_\ep a^\ell\|_{\dot B^\sigma_{2,1}}&\lesssim
\|c\|_{\dot B^{\frac d2}_{2,1}}\|\partial_k\partial_m K_\ep a^\ell\|_{\dot B^\sigma_{2,1}}+
\|c\|_{\dot B^\sigma_{2,1}}\|\partial_k\partial_m K_\ep a^\ell\|_{L^\infty}, \\
\|[c,\partial_k L_\ep]\partial_m L_\ep a^h\|_{\dot B^\sigma_{2,1}}&\lesssim \|\nabla c\|_{\dot B^{\frac d2}_{2,1}}\|\partial_m L_\ep a^h\|_{\dot B^\sigma_{2,1}}
+\|\nabla c\|_{\dot B^{\sigma}_{2,1}}\|\partial_m L_\ep a^h\|_{L^\infty},\\
\|\partial_k L_\ep(c\,\partial_m L_\ep a^h)\|_{\dot B^\sigma_{2,1}}^\ell&\lesssim 
\nu_0\bigl(\|c\|_{\dot B^{\frac d2}_{2,1}}\|\partial_m L_\ep a^h\|_{\dot B^\sigma_{2,1}}+\|c\|_{\dot B^\sigma_{2,1}}
\|\partial_m L_\ep a^h\|_{L^\infty}\bigr)\cdotp
\end{aligned}$$
Hence we have 
\begin{multline*}
\|\partial_k(c\,\partial_m K_\ep a)\|^\ell_{\dot B^\sigma_{2,1}}
\lesssim \|\nabla c\|_{\dot B^{\frac d2}_{2,1}}\|\nabla L_\ep a\|_{\dot B^\sigma_{2,1}}
+ \|\nabla c\|_{\dot B^{\sigma}_{2,1}}\|\nabla L_\ep a\|_{L^\infty}\\
+ \|c\|_{\dot B^{\sigma}_{2,1}}\bigl(\|\nabla^2 K_\ep a^\ell\|_{L^\infty}+
\|\nabla L_\ep a^h\|_{L^\infty}\bigr)
+\|c\|_{\dot B^{\frac d2}_{2,1}}\bigl(\|\nabla^2 K_\ep a^\ell\|_{\dot B^\sigma_{2,1}}+\|\nabla L_\ep a^h\|_{\dot B^\sigma_{2,1}}\bigr)\cdotp
\end{multline*}
Hence, putting  \eqref{eq:LF9}, \eqref{eq:almostfinal} and \eqref{eq:nablau} together gives  \eqref{eq:linear}.
\medbreak
Let us finally consider the case  where, in addition 
$u_0$ is in $\dot B^{\sigma-1}_{2,1}.$
The starting point is \eqref{eq:LF10}. Since the term with 
$\nabla K_\ep a$ in the right-hand side is controlled by \eqref{eq:almostfinal} and because 
$$\|w-u\|_{\dot B^{\sigma-1}_{2,1}}^\ell 
\leq C\|\nabla K_\ep a\|^{\ell}_{\dot B^{\sigma-1}_{2,1}}\leq
C\|a\|^{\ell}_{\dot B^{\sigma-1}_{2,1}}, $$
we have the low frequency part of  \eqref{eq:linearb}. 
The high frequency part just stems from the fact that one can bound $\|u\|^h_{\dot B^{\sigma-1}_{2,1}}$
and $\|\nabla K_\ep a\|^h_{\dot B^{\sigma-1}_{2,1}}$ by, say, 
$\|u\|^h_{\dot B^{\sigma}_{2,1}}$ and  $\|\nabla a\|^h_{\dot B^{\sigma}_{2,1}},$
and thus by means of \eqref{eq:linear}. In the end, we get \eqref{eq:linearb}.
\end{proof}


\section{Proving  well-posedness and convergence to Euler}\label{s:GWP}

    This section is devoted to proving Theorems \ref{thm:GWP} and \ref{thm:conv}.
    In the first subsection, we prove the existence part of Theorem \ref{thm:GWP} then, 
    in second subsection, the uniqueness part. 
    The end of the section is devoted to establishing the
    convergence of the solutions to \eqref{Kepsilon} to those of \eqref{Euler} for $\ep$ tending to $0.$

\subsection{Existence}

Before proving the existence part  of  Theorem \ref{thm:GWP}, 
let us quickly explain why the results of the previous section 
allow to close the estimates  for all time and uniformly with respect to $\ep$
in the desired functional space for any initial data satisfying \eqref{eq:smallness2}.

To this end, we consider a smooth solution $(\rho,u)$ of \eqref{Kepsilon} on $[0,T)\times\R^d$
such that $a:=\rho-1$ satisfies $|a|\leq 1/4.$
Then, applying Inequality \eqref{eq:linear}  to the system satisfied by $(a,u)$
(that is,  to \eqref{eq:eulerconv} with $b=a,$ $c=0$ and $v=u$)
with $\sigma=d/2+1$ we get some absolute constant $C$ such that for all $t\in[0,T),$ 
$$X^{\frac d2+1}_{a,u}(t)+\int_0^tH^{\frac d2+1}_{a,u}\,d\tau \leq C\biggl(X_{a,u}^{\frac d2+1}(0)+
\int_0^t \|\nabla u\|_{\dot B^{\frac d2}_{2,1}\cap
\dot B^{\frac d2+1}_{2,1}}X^{\frac d2+1}_{a,u}\,d\tau\biggr)\cdotp$$
Note indeed that the 2nd and 3rd lines of  \eqref{eq:linear} 
then reduce to just 
 $\|\nabla u\|_{\dot B^{\frac d2}_{2,1}\cap \dot B^{\frac d2+1}_{2,1}}X^{\frac d2+1}_{a,u}.$ 
 Now, it is obvious that 
 $$\|\nabla u\|_{\dot B^{\frac d2}_{2,1}\cap \dot B^{\frac d2+1}_{2,1}}\leq H^{\frac d2+1}_{a,u}.$$
Hence, we conclude by bootstrap that  the smallness condition
$$2C^2X_{a,u}^{\frac d2+1}(0)<1 $$ implies that 
\begin{equation}\label{eq:globest}
X^{\frac d2+1}_{a,u}(t)+\frac12\int_0^tH^{\frac d2+1}_{a,u}\,d\tau \leq C X_{a,u}^{\frac d2+1}(0)
\quad\hbox{for all }\ t\in[0,T).
\end{equation}
Let us next move to the rigorous proof of existence of a solution for \eqref{Kepsilon} under the assumptions
of Theorem \ref{thm:GWP}. 
For technical reasons, we will have assume that, in addition to \eqref{eq:Kep}, 
the kernel $K_\ep$ is such that  $\nabla K_\ep : L^2\to H^{1}.$
This is clearly achieved if  $\nabla^2 K_\ep$ is in $L^1_{loc}(\R^d)$
(since we already have   $y^2\nabla K_\ep\in L^1(\R^d)$), and we note that there exists $M_\ep\geq0$
such that for all $s\in\R,$
\begin{equation}\label{eq:Mep}
\|\nabla K_\ep\|_{\cL(H^s;H^{s+1})}\leq M_\ep.\end{equation}

\subsubsection*{Step 1. Solving Burgers equation with friction and smooth data}

Here we consider 
\begin{equation}\label{eq:burgers}
u_t+u\cdot\nabla u+u=f \end{equation}
supplemented with initial velocity $u_0\in H^{s+1}$ and source term
$f\in \cC(\R_+;H^{s+1})$ with $s>d/2.$

The classical theory of symmetric hyperbolic systems (see e.g. \cite[Chap. 4]{BCD}
guarantees that \eqref{eq:burgers} admits a unique maximal solution 
$$u\in \cC([0,T^*);H^{s+1})\cap \cC^1([0,T^*);H^s).$$
Furthermore, by combining an energy method and classical commutator estimates in Sobolev spaces,  we have
\begin{equation}\label{eq:uu}\|u(t)\|_{H^{s+1}}\leq 
\|u_0\|_{H^{s+1}} +\int_0^t \|f\|_{H^{s+1}}+C\int_0^t\|\nabla u\|_{L^\infty}\|u\|_{H^{s+1}}\,d\tau\quad\hbox{for all }\ t\in[0,T^*),\end{equation}
whence, remembering the Sobolev embedding $H^s\hookrightarrow L^\infty$ (for $s>d/2$),
$$\sup_{\tau\in[0,t]}\|u(\tau)\|_{H^{s+1}}\leq 
\|u_0\|_{H^{s+1}} +\int_0^t \|f\|_{H^{s+1}}+
Ct\sup_{\tau\in[0,t]}\|u(\tau)\|_{H^{s+1}}^2\quad\hbox{for all }\ t\in[0,T^*).$$
This guarantees that there exists some constant $c$ depending only on $s$ and $d,$ such that
$$
    T^*\geq \sup \biggl\{ t\geq0,\: 
    t\Bigl(\|u_0\|_{H^{s+1}} +\int_0^t \|f\|_{H^{s+1}}\,d\tau\Bigr)\leq c\biggr\}\cdotp$$

\subsubsection*{Step 2. Local existence for \eqref{Kepsilon} supplemented
with smooth data}

Fix some $R_0>0$ and  data $(a_0,u_0)\in H^s\times H^{s+1}$ with $s>d/2,$ such that
$$   \|a_0\|_{H^s}+\|u_0\|_{H^{s+1}}\leq R_0.$$
Our goal it to prove that there exists some $T>0$ such that 
\eqref{Kepsilon} has a solution 
\begin{equation}\label{eq:au}
(a,u)\in\cC([0,T];H^s\times H^{s+1})\cap \cC^1([0,T];H^{s-1}\times H^s).\end{equation}
To do so, we consider the map $\Psi: \wt a\longmapsto a$  where $a$ is the solution to the transport equation 
$$a_t +\div ((1+a)u)=0$$  and the transport field $u$  is  the solution in 
$\cC([0,T);H^{s+1})\cap \cC^1([0,T);H^s)$ to the damped Burgers equation
$$u_t+u+u\cdot\nabla u=-\nabla K_\ep\wt a.$$
Owing to \eqref{eq:Mep}, the existence of $u$ with the required regularity on some maximal time 
interval $[0,T^*)$ is guaranteed by the previous step.
Then, the existence of $a\in\cC([0,T^*);H^s)\cap\cC([0,T^*);H^{s-1})$ follows from the standard theory  of transport equations in Sobolev spaces.

We claim that one can find some $T\in(0,T^*)$ such that $\Psi$ maps  the closed ball 
$\bar B(0,R)$ of $\cC([0,T);H^s)$ to itself, with $R=2R_0+1.$ Indeed, 
combining an energy method and Gronwall lemma, 
it is easy to show that 
$$\|a(t)\|_{H^s}\leq e^{C\int_0^t\|u\|_{H^{s+1}}\,d\tau}\|a_0\|_{H^s}+
e^{C\int_0^t\|u\|_{H^{s+1}}\,d\tau}-1,\qquad t\in[0,T^*)$$
and, owing to \eqref{eq:uu} and \eqref{eq:Mep},
$$\|u(t)\|_{H^{s+1}}\leq e^{C\int_0^t\|u\|_{H^{s+1}}\,d\tau}
\biggl(\|u_0\|_{H^{s+1}} + M_\ep\int_0^t \|\wt a\|_{H^{s}}\,d\tau\biggl),\qquad t\in[0,T^*).$$
If  $T$ is taken small enough then we have
$$C\int_0^T\|u\|_{H^{s+1}}\,d\tau\leq \log 2,$$
and thus, if $\|\wt a\|_{L^\infty(0,T; H^s}\leq R,$ 
$$\sup_{t\in[0,T]}\|a(t)\|_{H^s}\leq 2R_0+1=R\andf 
\sup_{t\in[0,T]}\|u(t)\|_{H^{s+1}}\leq 2R_0 +2RM_\ep T.$$
Hence, to ensure  our claim, it suffices to choose $T$ such that
$$ 2RM_\ep T\leq 1\andf CTR \leq \log2.$$
Next, since $a_t=-\div((1+a)u),$  we readily have $a_t\in \cC([0,T];H^{s-1})$ and
$$\sup_{t\in[0,T]}\|a_t(t)\|_{H^s}\leq C(1+ \sup_{t\in[0,T]}\|a(t)\|_{H^s})
\sup_{t\in[0,T]}\|u(t)\|_{H^s}\leq CR(1+R).$$
Hence $(\Psi(\wt a))_t$ remains in a bounded set of $\cC([0,T];H^{s-1}).$
Remembering that the embedding of $H^{s}(\R^d)$ in $H^{s-1}(\R^d)$ is locally compact, 
Schauder theorem guarantees that $\Psi$ admits a fixed point $a$ in $L^\infty(0,T;H^s).$
Back to the equation of $u,$ we deduce that $u$ is in $\cC([0,T];H^{s+1})$
then, using once more the equation of $a,$ that $a$ is in $\cC([0,T];H^s).$
Finally, computing the time derivative of $a$ and $u$ from the equation, we conclude to 
\eqref{eq:au}.

\subsubsection*{Step 3. A continuation criterion}  

Let $T^*$ be the lifespan of the solution $(a,u)$ constructed in the previous step. 
On the one hand, applying \eqref{eq:uu} with $f=-\nabla K_\ep a,$ we see that
\begin{equation}\label{eq:uuu}\|u(t)\|_{H^{s+1}} \leq \|u_0\|_{H^{s+1}} +M_\ep\int_0^t \|a\|_{H^{s}}
+C\int_0^t\|\nabla u\|_{L^\infty}\|u\|_{H^{s+1}}\,d\tau\quad\hbox{for all }\ t\in[0,T^*).\end{equation}
On the other hand, using the standard estimates in Sobolev spaces for the transport equation
and the product law
$$ \|a\.\div u\|_{H^s}\lesssim \|a\|_{L^\infty}\|\div u\|_{H^s} + 
\|a\|_{H^s}\|\div u\|_{L^\infty},$$
we get for all $t\in[0,T^*),$
\begin{equation}\label{eq:aaa}\|a(t)\|_{H^{s}} \leq  \|a_0\|_{H^{s}} 
+C\int_0^t\|u\|_{H^{s+1}}(1+\|a\|_{L^\infty})\,d\tau+
C\int_0^t\|\nabla u\|_{L^\infty}\|a\|_{H^s}\,d\tau.\end{equation}
Putting \eqref{eq:uuu} and \eqref{eq:aaa} together, then using Gronwall lemma yields
for all $t\in[0,T^*),$
$$
\|a(t)\|_{H^{s}} + \|u(t)\|_{H^{s+1}} \leq  \bigl(\|a_0\|_{H^{s}}+\|u_0\|_{H^{s+1}}\bigr)
e^{(C+M_\ep)\int_0^t\bigl(1+\|a\|_{L^\infty}+\|\nabla u\|_{L^\infty}\bigr)d\tau},$$
whence the following blow-up criterion: 
\begin{equation}\label{eq:blow-up}
T^*<\infty\Longrightarrow \int_0^{T^*}\bigl(\|a\|_{L^\infty}+\|\nabla u\|_{L^\infty}\bigr)dt=\infty.\end{equation}

\subsubsection*{Step 4. Global existence for System \eqref{Kepsilon} with 
data in Sobolev spaces}

Fix a pair $(a_0,u_0)$ satisfying the smallness assumption of Theorem 
\ref{thm:GWP}, and consider a sequence    $(a_0^{(n)},  u_0^{(n)})_{n\in\N}$
of smooth initial data   such that
\begin{equation}\label{ini-data-app}
\begin{aligned}
  (a_0^{(n)},  u_0^{(n)}) &\to (a_0,u_0)
    \mbox{ \  in \ } \dot B^{d/2}_{2,1}\cap\dot B^{d/2+1}_{2,1}(\R^d)\\
    \andf (\nabla^2 K_\ep a_0^{(n)}, \nabla^2 u_0^{(n)} &\to 
 (\nabla^2 K_\ep a_0, \nabla^2 u_0)    \mbox{ \  in \ } \dot B^{d/2}_{2,1}(\R^d).
 \end{aligned}
\end{equation}
One can for instance set $a_0^{(n)}:= (\dot S_n-\dot S_{-n}) a_0$
and $u_0^{(n)}:= (\dot S_n-\dot S_{-n}) u_0$ so that
\eqref{eq:smalldata} is satisfied by   $(a_0^{(n)},  u_0^{(n)})$ for all $n\in\N.$
\medbreak
The previous steps guarantee that \eqref{Kepsilon} supplemented with 
initial data $(a_0^{(n)},  u_0^{(n)})$ has a unique maximal solution 
$(a^{(n)},u^{(n)})$ in, say, 
$$\cC([0,T^{(n)});H^{\frac d2+3}\times H^{\frac d2+4})\cap 
\cC^1([0,T^{(n)});H^{\frac d2+2}\times H^{\frac d2+3})).$$
We  thus have enough regularity to apply the estimates of Theorem 
\ref{thm:linear} with $\sigma=d/2+1,$ $a=b=a^{(n)},$
$u=v=u^{(n)}$ and $c=0.$ 
Following the proof of \eqref{eq:globest}, 
we conclude that if $\alpha_0$ in \eqref{eq:smalldata} is small enough, then
we have 
$$\sup_{t\in[0,T^{(n)})} 
X^{\frac d2+1}_{a^{(n)},u^{(n)}}(t)+\frac12\int_0^{T^{(n)}} H^{\frac d2+1}_{a^{(n)},u^{(n)}}(\tau)\,d\tau\leq C\alpha_0\quad\hbox{for all }
t\in[0,T^{(n)}).$$
This implies that both $a^{(n)}(t)$ and $\nabla u^{(n)}(t)$ remain in a bounded set of $L^\infty.$
Hence the blow-up criterion \eqref{eq:blow-up} ensures that $T^{(n)}=\infty.$

Note also that  the second estimate of Theorem \ref{thm:linear} provides
a uniform control on $u^{(n)}$ and $w^{(n)}$ in 
$L^\infty(\R_+;\dot B^{\frac d2}_{2,1}).$

\subsubsection*{Step 5. Passing to the limit} 

In the previous step, we constructed a sequence of smooth global solutions of
\eqref{Kepsilon} pertaining to smooth data, that is bounded in the space $E^{\frac d2+1}_{K_\ep},$
and thus in 
$$F^{\frac d2+1}_{K_\ep}:=\bigl\{(b,v)\in L^\infty(\R_+;\dot B^{\frac d2}_{2,1}\cap \dot B^{\frac d2+1}_{2,1})\times  
L^\infty(\R_+;\dot B^{\frac d2}_{2,1}\cap \dot B^{\frac d2+2}_{2,1}),\. 
\nabla^2 K_\ep \in L^\infty(\R_+;\dot B^{\frac d2}_{2,1})\bigr\}\cdotp$$
This latter space being the dual of some separable Banach space, one may deduce that there
exists $(a,u)$ in $F^{\frac d2+1}_{K_\ep}$ such that, 
up to an omitted extraction, we have 
$$(a^{(n)},u^{(n)})\rightharpoonup (a,u)\ \ \hbox{weak * in }\ F^{\frac d2+1}_{K_\ep}.$$
Furthermore, as 
$$u^{(n)}_t=-u^{(n)}-u^{(n)}\cdot\nabla u^{(n)}-\nabla K_\ep a^{(n)},$$
the sequence $(u^{(n)})_{n\in\N}$ is bounded in $L^\infty(\R_+;\dot B^{\frac d2}_{2,1})$ 
(use product laws \eqref{eq:prod1} and \eqref{eq:prod2}) and, similarly, 
$(a^{(n)})_{n\in\N}$ is bounded in $L^\infty(\R_+;\dot B^{\frac d2}_{2,1}).$
Using the fact that both sequences are, in particular, bounded in 
$L^\infty(\R_+;\dot B^{\frac d2}_{2,1}\cap \dot B^{\frac d2+1}_{2,1})$ and that the embedding 
from $\dot B^{\frac d2}_{2,1}\cap \dot B^{\frac d2+1}_{2,1}$ to $\dot B^{\frac d2}_{2,1}$ is locally 
compact, one discovers that for all $\theta\in \cC^\infty_c(\R^d),$ we have (again, up to extraction),
$$(\theta a^{(n)},\theta u^{(n)})\to (\theta a,\theta u) \quad\hbox{strongly in }\ \cC_b(\R_+;\dot B^{\frac d2}_{2,1}).$$
This allows to pass to the limit in \eqref{Kepsilon} in the sense of distributions, fingers in the nose. 

Furthermore, for all fixed $t\in\R_+,$ the sequence $(\nabla u^{(n)}(t))_{n\in\N}$ is bounded in 
$\dot B^{\frac d2}_{2,1}\cap \dot B^{\frac d2+1}_{2,1}$ hence must converge to some function 
$z(t)$ weakly * in $\dot B^{\frac d2}_{2,1}\cap \dot B^{\frac d2+2}_{2,1}.$
Combining with the above property of strong convergence, we deduce that we must have $z(t)=\nabla u(t).$
Then, using the properties of lower semi-continuity of the weak limit and Fatou lemma, one can  write
$$\begin{aligned}\int_{\R_+} \|\nabla u\|_{\dot B^{\frac d2}_{2,1}\cap\dot B^{\frac d2+1}_{2,1}}\,dt
&\leq \int_{\R_+} \liminf \|\nabla u^{(n)}\|_{\dot B^{\frac d2}_{2,1}\cap\dot B^{\frac d2+1}_{2,1}}\,dt\\
&\leq  \liminf\int_{\R_+} \|\nabla u^{(n)}\|_{\dot B^{\frac d2}_{2,1}\cap\dot B^{\frac d2+1}_{2,1}}\,dt
\leq C\alpha_0.\end{aligned}$$
Similar arguments may be employed to show that all the other $L^1$-in-time properties of the 
space $E^{\frac d2+1}_{K_\ep}$ are satisfied by $(a,u).$
Finally, the time continuity for $u$ with values in $\dot B^{\frac d2}_{2,1}\cap\dot B^{\frac d2+2}_{2,1}$ stems from the properties of the transport equation and of the fact that
(remember \eqref{eq:Mep}):
$$u_t+u+u\cdot\nabla u=-\nabla K_\ep a\in L^\infty(\R_+;\dot B^{\frac d2}_{2,1}\cap\dot B^{\frac d2+2}_{2,1}).$$
Similarly, 
$$ a_t+ u\cdot\nabla a = -(1+a)\div u \in L^1(\R_+;\dot B^{\frac d2}_{2,1}\cap\dot B^{\frac d2+1}_{2,1}),$$
and thus $a\in \cC_b(\R_+;\dot B^{\frac d2}_{2,1}\cap\dot B^{\frac d2+1}_{2,1}).$
This completes the proof of  existence in Theorem \ref{thm:GWP}.

\subsection{Uniqueness}

This part is devoted to proving the uniqueness part of Theorem \ref{thm:GWP}.
We consider two solutions $(\rho_1=1+a_1,u_1)$ and $(\rho_2=1+a_2,u_2)$ of \eqref{Kepsilon} 
with $(a_i,u_i)$ for $i=1,2$ in the space $E^{\frac d2+1}_{K_\ep}.$ 
Then, we observe that  $\da:=\rho_2-\rho_1$ and $\du:=u_2-u_1$ satisfy
$$\left\{\begin{aligned}
&\partial_t\da +u_1\cdot\nabla \da +\div \du+a_1\div \du=f:=-\du\cdot\nabla a_2-\da\div u_2,\\[1ex]
&\partial_t\du +u_1\cdot\nabla \du +\du +\nabla K_\ep \da=g:= -\du\cdot\nabla u_2.\end{aligned}\right.$$
Hence  $(\da,\du)$ satisfies 
a linear system of type \eqref{eq:eulerconv} with $v=u_1$ and $b=a_1$
and source terms $(f,g).$
Now, uniqueness on a finite interval $[0,T]$ will stem from   Inequality \eqref{eq:linear}
provided we have proved beforehand that  $(\da,\du)$ has
the regularity required in Theorem \ref{thm:linear} in the case $\sigma=d/2.$
After careful inspection of what is already known on $(a_1,u_1)$ and $(a_2,u_2),$
we see that it suffices  to check  that  
$$\da\in L^\infty(0,T;\dot B^{\frac d2-1}_{2,1}) \andf \du\in L^\infty(0,T;\dot B^{\frac d2}_{2,1}).$$
These two properties may be justified from the density and velocity equations
and   product laws in Besov spaces (that is, \eqref{eq:prod1}) which guarantee
that
$$\partial_t a_i \in L^1(0,T;\dot B^{\frac d2-1}_{2,1})
\andf \partial_t u_i+u_i \in L^1(0,T;\dot B^{\frac d2}_{2,1}),\qquad i=1,2.$$
 Hence, using  the short notation $\dX:=X^{\frac d2}_{\da,\du}$ and $\dH:=H^{\frac d2}_{\da,\du},$
 we have:
\begin{multline}\label{eq:dX}
\dX(t)+\int_0^t\dH\,d\tau\lesssim \dX(0)+\int_0^t\|a_1,\nabla a_1,\nabla u_1,\nabla^2 u_1\|_{\dot B^{\frac d2}_{2,1}}
\dX\,d\tau\\+\int_0^t\|f,\nabla f,\nabla^2 L_\ep f\|_{\dot B^{\frac d2-1}_{2,1}}\,d\tau
+\int_0^t\|g,\nabla g\|_{\dot B^{\frac d2}_{2,1}}\,d\tau.
    \end{multline}
Thanks to \eqref{eq:prod1}, we readily have
$$\begin{aligned}
\|f\|_{\dot B^{\frac d2-1}_{2,1}}&\lesssim \|\du\|_{B^{\frac d2}_{2,1}}
\|a_2\|_{\dot B^{\frac d2}_{2,1}}+\|\da\|_{B^{\frac d2-1}_{2,1}}
\|\div u_2\|_{\dot B^{\frac d2}_{2,1}},\\
\|f\|_{\dot B^{\frac d2}_{2,1}}&\lesssim
 \|\du\|_{B^{\frac d2}_{2,1}}
\|\nabla a_2\|_{\dot B^{\frac d2}_{2,1}}+
\|\da\|_{B^{\frac d2}_{2,1}}
\|\div u_2\|_{\dot B^{\frac d2}_{2,1}},\\
\|g\|_{\dot B^{\frac d2}_{2,1}}&\lesssim \|\du\|_{B^{\frac d2}_{2,1}}
\|\nabla u_2\|_{\dot B^{\frac d2}_{2,1}}, \\
\|\nabla g\|_{\dot B^{\frac d2}_{2,1}}&\lesssim \|\nabla \du\|_{B^{\frac d2}_{2,1}}
\|\nabla u_2\|_{\dot B^{\frac d2}_{2,1}}+\|\du\|_{B^{\frac d2}_{2,1}}
\|\nabla^2 u_2\|_{\dot B^{\frac d2}_{2,1}}.
\end{aligned}$$
Bounding $\nabla L_\ep f$ in $\dot B^{\frac d2}_{2,1}$ is a bit more tricky. 
To achieve it, we use the decompositions:
$$\begin{aligned}
\nabla L_\ep(\du\cdot\nabla a_2)&=[\nabla L_\ep,\du]\cdot\nabla a_2+\du\cdot\nabla^2L_\ep a_2,\\
\nabla L_\ep(\da\div u_2)&=[\nabla L_\ep,\div u_2]\da+\div u_2\,\nabla L_\ep\da.\\
\end{aligned}$$
Hence, taking advantage of \eqref{eq:prod1} and of \eqref{com:II}, we have
$$\begin{aligned}
\|\nabla L_\ep(\du\cdot\nabla a_2)\|_{\dot B^{\frac d2}_{2,1}}&\lesssim 
\|\nabla\du\|_{\dot B^{\frac d2}_{2,1}}\|\nabla a_2\|_{\dot B^{\frac d2}_{2,1}}+
\|\du\|_{\dot B^{\frac d2}_{2,1}}\|\nabla^2L_\ep a_2\|_{\dot B^{\frac d2}_{2,1}},\\
\|\nabla L_\ep(\da\div u_2)\|_{\dot B^{\frac d2}_{2,1}}&\lesssim 
\|\da\|_{\dot B^{\frac d2}_{2,1}}
\|\nabla\div u_2\|_{\dot B^{\frac d2}_{2,1}}+
\|\div u_2\|_{\dot B^{\frac d2}_{2,1}}\|\nabla L_\ep \da\|_{\dot B^{\frac d2}_{2,1}}.
\end{aligned}$$
Plugging all the above inequalities in \eqref{eq:dX} yields
$$\dX(t)+\int_0^t\dH\,d\tau\lesssim \dX(0)+\int_0^t\|a_1,\nabla a_1,\nabla u_1,\nabla^2 u_1,a_2,\nabla a_2,\nabla^2L_\ep a_2, \nabla u_2,\nabla^2u_2\|_{\dot B^{\frac d2}_{2,1}}
\dX\,d\tau.$$
Since the prefactor of $\dX$ in the right-hand side is indeed locally integrable 
on time, Gronwall lemma ensures $\dX\equiv0$ that is, uniqueness, if the initial 
data of the two solutions are the same ones, and, more generally,  
stability with respect to the data.  \qed


\subsection{Convergence to Euler}

Justifying it  is  an easy adaptation of the proof of uniqueness that has been presented
just above. 

Indeed, consider initial data $(\rho_0=1+a_0,u_0)$ such that the smallness condition \eqref{eq:smallness2}
is satisfied. 
Then, even if it means a slight change in $\alpha_0,$ Condition \eqref{eq:smalldata}
is satisfied for all small enough $\ep>0.$
Consequently, on the one hand, Theorem \ref{thm:GWP} provides us with a unique global solution
$(\rho_\ep=1+a_\ep,u_\ep)$ satisfying  the properties described therein.
On the other hand, by following faithfully the proof of estimates for 
 \eqref{eq:eulerconv} (formally replacing the convolution by $K_\ep$ with the identity operator
 everywhere) then adapting the proof of existence accordingly, one  
 gets a global solution $(\rho=1+a,u)$ for  Euler 
 system \eqref{Euler} such that\footnote{The global well-posedness result 
 of \cite{CBD2} only deasls with regularity $\dot B^{\frac d2}_{2,1}\cap \dot B^{\frac d2+1}_{2,1}.$} 
$(a,u)$ belongs to the space $E^{\frac d2+1}$
defined in \eqref{def:E}.
To prove the convergence of $(a_\ep,u_\ep)$ to $(a,u)$, let us look at the system
satisfied by $\da:=a-a_\ep$ and $\du:=u-u_\ep$:
$$\left\{\begin{aligned}
&\partial_t\da +u_\cdot\nabla \da +\div \du+a\div \du=f:=-\du\cdot\nabla a_\ep-\da\div u_\ep,\\
&\partial_t\du +u\cdot\nabla \du +\du +\nabla K_\ep \da=g:= -\du\cdot\nabla u_\ep
+\nabla(K_\ep-Id)a.\end{aligned}\right.$$
Compared to the proof of uniqueness, only the last term of $g$ is new.
All the other terms may be bounded as above after replacing $a_1$ and $u_1$
(resp. $a_2$ and $u_2$) by   $a$ and $u$ (resp. $a_\ep$ and $u_\ep$). 
However, as we strive for a global-in-time result of convergence, 
putting all the terms concerning $a$ or $a_\ep$ as prefactor of
$\dX$ is not suitable : we need to be a little more precise, so  we write that
\begin{multline*}
 \dX(t)+\int_0^t\dH\,d\tau\lesssim
 \dX(0)+\int_0^t\|\nabla u,\nabla^2 u,\nabla^2L_\ep a_\ep, \nabla u_\ep,\nabla^2u_\ep\|_{\dot B^{\frac d2}_{2,1}}
\dX\,d\tau\\+\int_0^t\|a,\nabla a,a_\ep,\nabla a_\ep\|_{\dot B^{\frac d2}_{2,1}}
\|\du,\nabla\du\|_{\dot B^{\frac d2}_{2,1}}\,d\tau
+\int_0^t \|\nabla(K_\ep-Id)a\|_{\dot B^{\frac d2}_{2,1}\cap \dot B^{\frac d2+1}_{2,1}}\,d\tau.   
\end{multline*}
Since $\|\du,\nabla\du\|_{\dot B^{\frac d2}_{2,1}}$ is a part of $\dH$ and 
$$\sup_{t\in\R_+}\|(a,\nabla a,a_\ep,\nabla a_\ep)(t)\|_{\dot B^{\frac d2}_{2,1}}\ \hbox{ is small},$$ 
the last but one term in the right-hand side may be absorbed by the left-hand side. 
Furthermore, the map $t\mapsto \|(\nabla u,\nabla^2 u,\nabla^2L_\ep a_\ep, \nabla u_\ep,\nabla^2u_\ep)(t)\|_{\dot B^{\frac d2}_{2,1}}$ is integrable on $\R_+$
and is also small. 
Hence, applying Gronwall lemma yields for all $t\in\R_+,$
$$\dX(t)+\int_0^t\dH\,d\tau\lesssim \dX(0)+\int_0^t \|\nabla(K_\ep-Id)a\|_{\dot B^{\frac d2}_{2,1}\cap \dot B^{\frac d2+1}_{2,1}}\,d\tau.$$
Finally, the properties of the solution $(a,u)$ ensure in particular that
$$\nabla a\in L^2(\R_+;\dot B^{\frac d2}_{2,1})\andf 
\nabla^2 a\in L^1(\R_+;\dot B^{\frac d2}_{2,1}).$$
Hence, by virtue of Lebesgue's dominated convergence theorem, we have
$$
\int_0^\infty \|\nabla(K_\ep-Id)a\|_{\dot B^{\frac d2+1}_{2,1}}\,d\tau\to 0\andf 
\int_0^t\|\nabla(K_\ep-Id)a\|_{\dot B^{\frac d2}_{2,1}}\,d\tau\to 0,\quad
\hbox{for all }\ t\in\R_+.
$$
From this, we conclude, among other, that
for $\ep \to 0$
$$(a_\ep-a)\to 0 \ \hbox{ in }\ L^\infty_{loc}(\R_+;\dot B^{\frac d2-1}_{2,1})
\andf (u_\ep-u)\to 0 \ \hbox{ in }\ L^\infty_{loc}(\R_+;\dot B^{\frac d2}_{2,1}).$$
Interpolating with the uniform bounds that are satisfied by $(a_\ep,u_\ep)$ and $(a,u),$
one can upgrade the convergence to e.g.
$$a_\ep\to a\ \hbox{ in }\ L^\infty_{loc}(\R_+;\dot B^{\frac d2+\alpha}_{2,1}),
\ \alpha\in[0,1)\andf 
u_\ep\to u\ \hbox{ in }\ L^\infty_{loc}(\R_+;\dot B^{\frac d2+\beta}_{2,1}),
\ \beta\in[0,2).$$
In order to have uniform-in-time convergence, one can  assume 
in addition that $a_0\in\dot B^{\frac d2-1}_{2,1}.$
Then, we have $\nabla a\in L^1(\R_+;\dot B^{\frac d2}_{2,1})$ and 
thus $\nabla(K_\ep-Id)a\to 0$ in $L^1(\R_+;\dot B^{\frac d2}_{2,1})$
as $\ep \to 0$.

To have a global result, another possibility  is to assume that
in addition to \eqref{eq:Kep},  we have $K_\ep=\ep^{-d}K(\ep^{-1})$
with $K$ such that
$\eta\mapsto |\eta|^{-1}(\wh K(\eta)-1)$ is bounded.
Indeed, one can write
$$\cF(\nabla(K_\ep-\Id))(\xi)=i\biggl(\frac{\wh K(\ep\xi)-1}{\ep|\xi|}\biggr)\:
\ep\xi|\xi|\wh a(\xi),$$
and thus $\nabla(K_\ep-\Id)={\mathcal O}(\ep)$ in $L^1(\R_+;\dot B^{\frac d2}_{2,1}).$
\medbreak
This completes the proof of  Theorem \ref{thm:conv}.


\section{On the high friction limit}\label{s:limit}

In the introductory part of the paper, we pointed out that, after
performing the diffusive rescaling \eqref{eq:diffusive}, the density 
formally tends to the solution of different avatars of the porous media 
equation. 
In this section, we aim at justifying rigorously this heuristics, getting
in the small data case, strong and global-in-time results of convergence.

As a preliminary step, let us present the results that can be deduced from 
Theorem \ref{thm:GWP}. Fix some data $(\rho_0=1+a_0,u_0)$ 
satisfying  the regularity requirements therein, and denote
by $(\wt\rho_0=1+\wt a_0,\wt u_0)$ the data corresponding to the 
 rescaling \eqref{eq:keyrescaling}. 
 If $(\wt a_0,\wt u_0)$ fulfills \eqref{eq:smalldata} \emph{with $\ep\f$ instead
 of $\ep$}, then Theorem \ref{thm:GWP} gives us a unique global solution $(1+\wt a,\wt u)$
satisfying \eqref{est:1} and \eqref{est:2} (with $\ep\f$). 
We have the following scaling properties for $\wt z(x)=z(\f^{-1}x):$
$$\|\wt z\|_{\dot B^{\frac d2+\sigma}_{2,1}} = \f^{-\sigma} \|z\|_{\dot B^{\frac d2+\sigma}_{2,1}} \andf L_{\f\ep}\wt z(x)= (L_\ep z)(\f^{-1}x).$$
Hence, reverting to the original variables, we deduce that provided
\begin{equation}\label{eq:smalldataf}\|(a_0,\f^{-1}\nabla a_0,\f^{-2}\nabla^2L_{\ep} a_0)\|_{\dot B^{\frac d2}_{2,1}}+\f^{-1}\|(u_0,\f^{-1}\nabla u_0)\|_{\dot B^{\frac d2+1}_{2,1}}\leq \alpha_0,\end{equation}
System \eqref{Kepsilon} has a unique global solution $(\rho=1+a,u)$ such that 
\begin{equation}\label{eq:Xf}
X(t)+\int_0^tH\,d\tau \leq C X(0)\end{equation}
with 
$$\displaylines{X(t):=\|(a,\f^{-1}\nabla a,\f^{-2}\nabla^2L_{\ep} a)(t)\|_{\dot B^{\frac d2}_{2,1}}+\f^{-1}\|(u,\f^{-1}\nabla u)(t)\|_{\dot B^{\frac d2+1}_{2,1}}\andf\cr
H(t):=\|(u,\f^{-1}\nabla  u)\|_{\dot B^{\frac d2+1}_{2,1}}+\f^{-1}\|\nabla^2 K_{\ep} a\|^\ell_{\dot B^{\frac d2}_{2,1}}+\f^{-2}\|\nabla^2 K_{\ep} a\|^\ell_{\dot B^{\frac d2+1}_{2,1}}
+\|(a,\f^{-1}\nabla L_{\ep}  a)\|_{\dot B^{\frac d2+1}_{2,1}}^h.}$$
Furthermore, if, in addition, $u_0$ belongs to $\dot B^{\frac d2}_{2,1},$ 
then  the damped mode $w:=u+\f^{-1}\nabla K_\ep a$ satisfies
\begin{equation}\label{eq:ud/2b}
\|u(t)\|_{\dot B^{\frac d2}_{2,1}} +\|w(t)\|_{\dot B^{\frac d2}_{2,1}} +
\f\int_0^t\|w\|_{\dot B^{\frac d2}_{2,1}}\,d\tau\leq C\bigl(\|u_0\|_{\dot B^{\frac d2}_{2,1}}^\ell+X(0)\bigr)\cdotp\end{equation}
Based on these uniform estimates, it will be rather easy to justify the high friction asymptotics 
pointed out in the introduction, after performing the diffusive rescaling  \eqref{eq:diffusive}.

\subsection{High relaxation limit for fixed $\ep$}

In this section, we justify the convergence of $(\check\rho,\check u)$ (obtained 
from $(\rho,u)$ and \eqref{eq:diffusive}) to $(r,-\nabla K_\ep r),$ with $r$ satisfying \eqref{eq:pmeps} supplemented with initial data $\rho_0.$
Our main result reads as follows:
\begin{thm}\label{thm:fric1} Fix some $\ep>0$ and data $(\rho_0=1+a_0,u_0)$ such that
$a_0$ and $u_0$ are in $\dot B^{\frac d2}_{2,1}\cap \dot B^{\frac d2+2}_{2,1}.$
There exists an absolute constant $\alpha_0$ such that, if 
\begin{equation}\label{eq:smalldata1}
\|a_0\|_{\dot B^{\frac d2}_{2,1}}\leq\alpha_0
\end{equation}
then, for all large enough $\f,$ System \eqref{Kepsilon} admits a unique global 
solution $(\rho_\f=1+a_\f,u_\f)$ satisfying \eqref{eq:Xf} and \eqref{eq:ud/2b}, 
and Equation \eqref{eq:pmeps} supplemented with initial data $\rho_0$ 
has a unique global solution $r=1+\wt r$ with 
$\wt r\in\cC_b(\R_+;\dot B^{\frac d2}_{2,1})$ and 
$\nabla^2K_\ep\wt r\in L^1(\R_+;\dot B^{\frac d2}_{2,1}).$
\smallbreak
Furthermore, if  $(\check \rho_\f,\check u_\f)$ is defined from $(\rho_\f,u_\f)$ by \eqref{eq:diffusive}, then we have
\begin{equation*} \|\check u_\f+\nabla K_\ep r\|_{L^1(\R_+;\dot B^{\frac d2}_{2,1})}+
    \|\check \rho_\f -r \|_{L^\infty(\R_+;\dot B^{\frac d2-1}_{2,1})}+\|\nabla^2 K_\ep( \check \rho_\f -r)\|_{L^1(\R_+;\dot B^{\frac d2-1}_{2,1})} \to 0 \mbox{ 
    \ as \ } \f \to \infty
\end{equation*}
with convergence rate $\f^{-1}.$
\end{thm}
\begin{proof}
Since $a_0$ and $u_0$ are in $\dot B^{\frac d2}_{2,1}\cap \dot B^{\frac d2+2}_{2,1}$
and  Hypothesis \eqref{eq:smalldata1} holds, we are guaranteed that  the smallness condition \eqref{eq:smalldataf} is satisfied for large enough $\f.$ Hence, as explained at the beginning of
Section \ref{s:limit}, there exists a unique global solution $(\rho_\f,u_\f)$ of
\eqref{Kepsilon} with the desired properties. 

Next, in terms of $\wt r:=r-1,$ Equation \eqref{eq:pmeps} reads
\begin{equation}\label{eq:a1}\partial_t\wt r-\nabla K_\ep \wt r\cdot\nabla \wt r-\Delta K_\ep \wt r
=\wt r \Delta K_\ep \wt r,\qquad \wt r|_{t=0}=a_0.
\end{equation}
This may be seen as a degenerate convection diffusion equation. We claim that 
there exists an absolute constant $c_0$ such that for all $t\geq0,$ we have
\begin{equation}\label{eq:a2}
\|\wt r(t)\|_{\dot B^{\frac d2}_{2,1}}+c_0\int_0^t\|\nabla^2K_\ep \wt r\|_{\dot B^{\frac d2}_{2,1}}
\leq \|\wt r_{0}\|_{\dot B^{\frac d2}_{2,1}}+C\int_0^t\|\nabla^2K_\ep \wt r\|_{\dot B^{\frac d2}_{2,1}}
\|\wt r\|_{\dot B^{\frac d2}_{2,1}}\,d\tau.\end{equation}
Indeed, localizing \eqref{eq:a1} by means of $\ddj$ gives
$$\partial_t\wt r_j-\nabla K_\ep \wt r\cdot\nabla \wt r_j-\Delta K_\ep \wt r_j=\ddj(\wt r\Delta K_\ep \wt r)-[\nabla K_\ep \wt r, \ddj]\cdot\nabla \wt r.$$
Hence, taking the $L^2$ scalar product with $\wt r_j$ and integrating by parts 
in the second and third term of the left-hand side:
$$\displaylines{
\frac 12\frac d{dt}\|\wt r_j\|_{L^2}^2+\frac 12\int_{\R^d} |\wt r_j|^2\Delta K_\ep \wt r\,dx
+\int_{\R^d} \nabla \wt r_j\cdot\nabla K_\ep \wt r_j\,dx\hfill\cr\hfill=\int_{\R^d}\wt r_j\,\ddj(\wt r\Delta K_\ep \wt r)\,dx
-\int_{\R^d} \wt r_j\,[\nabla K_\ep \wt r, \ddj]\cdot\nabla \wt r\,dx.}$$
So, using the usual integration procedure and \eqref{eq:bernstein}, we get
a universal positive constant $\kappa_0$
such that
$$\displaylines{\|\wt r_j(t)\|_{L^2} +\kappa_0\int_0^t\|\nabla^2 K_\ep \wt r_j\|_{L^2}\,d\tau
\leq \|\wt r_{j,0}\|_{L^2}\hfill\cr\hfill +
\int_0^t \|\ddj(\wt r\Delta K_\ep \wt r)\|_{L^2}\,d\tau+\int_0^t\|[\nabla K_\ep \wt r, \ddj]\cdot\nabla \wt r\|_{L^2}\,d\tau+\frac12\int_0^t\|\Delta K_\ep \wt r\|_{L^\infty}\|\wt r_j\|_{L^2}\,d\tau.}
$$
Taking advantage of \eqref{eq:prod1}, \eqref{eq:keyembedding} and \eqref{com:I}, we discover that
$$ \|\ddj(\wt r\Delta K_\ep \wt r)\|_{L^2}+\|[\nabla K_\ep \wt r, \ddj]\cdot\nabla \wt r\|_{L^2}+\|\Delta K_\ep \wt r\|_{L^\infty}\|\wt r_j\|_{L^2}\leq Cc_j2^{-j\frac d2}\|\nabla^2K_\ep \wt r\|_{\dot B^{\frac d2}_{2,1}}
\|\wt r\|_{\dot B^{\frac d2}_{2,1}},$$
which after multiplication by $2^{j\frac d2}$ and summation on $j\in\Z$ completes
the proof of \eqref{eq:a2}.
\smallbreak
Having this inequality at our disposal and assuming that $\alpha_0$ in \eqref{eq:smalldata1}
is small enough, one can use the fixed point theorem 
(e.g. adapting the proof for the incompressible Navier-Stokes equations given 
in \cite[Chap. 5]{BCD})  to solve \eqref{eq:a1} globally in time. We get a unique solution $\wt r$ in 
$\cC_b(\R_+;\dot B^{\frac d2}_{2,1})$ such that
$$\|\wt r(t)\|_{\dot B^{\frac d2}_{2,1}}+\kappa_0\int_0^t\|\nabla^2K_\ep \wt r\|_{\dot B^{\frac d2}_{2,1}}
\leq 2\|a_0\|_{\dot B^{\frac d2}_{2,1}},\qquad t\in\R_+.$$
\smallbreak
Let us  drop index $\f$ for better readability. In order to prove the last part of the theorem, we observe that 
$$\partial_t\check\rho -\div(\check\rho \nabla K_\ep \check\rho)=-\div(\check\rho \check w)
\with \check w:=\check u+\nabla K_\ep\check\rho.$$
The key to the proof is that  \eqref{eq:ud/2b} after rescaling implies that 
\begin{equation}\label{eq:dr-1}\int_0^t\|\check w\|_{\dot B^{\frac d2}_{2,1}}\,d\tau\leq C\f^{-1}\bigl(\|u_0\|_{\dot B^{\frac d2}_{2,1}}^\ell +\alpha_0\bigr),\qquad t>0.\end{equation}
 The difference $\dr:=\check\rho-r$  satisfies:
$$\partial_t\dr-\div(\check\rho\nabla K_\ep\dr)=-\div(\dr \nabla K_\ep r)-\div(\check\rho \check w).$$
Putting  $\check a:=\check\rho-1$ and remembering that $r=1+\wt r,$ the above equation may be rewritten:
\begin{equation}\label{eq:dr}\partial_t\dr
+\nabla K_\ep \wt r\cdot\nabla\dr
-\Delta K_\ep\dr= \div(\check a\nabla K_\ep\dr)
-\dr \Delta K_\ep \wt r-\div((1+\check a) \check w).\end{equation}
Localizing \eqref{eq:dr} by means of Littlewood-Paley decomposition, then arguing
as for proving \eqref{eq:a2} (with regularity index $d/2-1$ instead of $d/2$), we get  for all $t>0,$
\begin{multline}\label{eq:dr1}\|\dr(t)\|_{\dot B^{\frac d2-1}_{2,1}}+\kappa_0
\int_0^t\|\nabla^2 K_\ep\dr\|_{\dot B^{\frac d2-1}_{2,1}} \lesssim
\int_0^t\|\nabla^2 K_\ep \wt r\|_{\dot B^{\frac d2}_{2,1}}\|\dr\|_{\dot B^{\frac d2-1}_{2,1}}\,d\tau
\\+\|\div(\check a\nabla K_\ep\dr)\|_{L^1(\R_+;\dot B^{\frac d2-1}_{2,1})}
+ \|\dr \Delta K_\ep \wt r \|_{L^1(\R_+;\dot B^{\frac d2-1}_{2,1})}
+\|\div((1+\check a) \check w)\|_{L^1(\R_+;\dot B^{\frac d2-1}_{2,1})}.\end{multline}
According to the product law \eqref{eq:prod1}, and to  \eqref{eq:Xf}, we have:
$$\|\div(\check a\nabla K_\ep\dr)\|_{L^1(\R_+;\dot B^{\frac d2-1}_{2,1})}
\lesssim \|\check a\|_{L^\infty(\R_+;\dot B^{\frac d2}_{2,1})}
\|\nabla K_\ep\dr\|_{L^1(\R_+;\dot B^{\frac d2}_{2,1})}\lesssim\alpha_0 
\|\nabla^2 K_\ep\dr\|_{L^1(\R_+;\dot B^{\frac d2-1}_{2,1})},$$
so this term may be absorbed  by the left-hand side of \eqref{eq:dr1}.
\smallbreak
For the next term, we have:
$$ \|\dr \Delta K_\ep  \wt r \|_{L^1(\R_+;\dot B^{\frac d2-1}_{2,1})}
\lesssim  \|\dr\|_{L^\infty(\R_+;\dot B^{\frac d2-1}_{2,1})}
\|\Delta K_\ep \wt r \|_{L^1(\R_+;\dot B^{\frac d2}_{2,1})}
\lesssim\alpha_0 \|\dr\|_{L^\infty(\R_+;\dot B^{\frac d2-1}_{2,1})}.$$
Hence,  remembering \eqref{eq:dr-1}, Inequality \eqref{eq:dr1} implies that
$$ \|\dr\|_{L^\infty(\R_+;\dot B^{\frac d2-1}_{2,1})}+\|\nabla^2 K_\ep\dr\|_{L^1(\R_+;\dot B^{\frac d2-1}_{2,1})} \leq 
  C\f^{-1}\bigl(\|u_0\|_{\dot B^{\frac d2}_{2,1}}^\ell +\alpha_0\bigr)\cdotp   $$
Since $\check u=\check w-\nabla K_\ep\check\rho,$ the above inequality and \eqref{eq:dr-1}
imply that $\check u$ tends to the limit `velocity' $z:=-\nabla K_\ep r$ with 
convergence rate $\f^{-1}$ in $L^1(\R_+;\dot B^{\frac d2}_{2,1}).$
\end{proof}


\subsection{Convergence of the relaxed system to the porous media equation}

In this part, we want to justify the limit of solutions of Equation 
\eqref{eq:pmeps} to those of the porous media equation \eqref{eq:pme}, 
when $\ep$ goes to $0.$
Our main result is stated below:
\begin{thm}\label{thm:fric2}  Consider initial data $r_{\ep,0}$ and  $n_0$  such that 
 $\wt r_{\ep,0}:=r_{\ep,0}-1$ and $\wt n_0:= n_0-1$ are in $\dot B^{d/2}_{2,1}(\R^d).$ There exists an absolute constant  $\alpha_0$ such that if 
$$\max\bigl(\|\wt r_{\ep,0}\|_{\dot B^{d/2}_{2,1}}, \|\wt n_0\|_{\dot B^{d/2}_{2,1}}\bigr)
    \leq\alpha_0, $$
then Equations \eqref{eq:pmeps} and \eqref{eq:pme} have a unique global solution $r_\ep=1+\wt r_\ep$ and $n=1+\wt n$ with $r_\ep$ given by Theorem \ref{thm:fric1}
and $\wt n\in \cC_b(\R_+;\dot B^{\frac d2}_{2,1})\cap L^1(\R_+;\dot B^{\frac d2+2}_{2,1}),$
and we have
\begin{equation}\label{est:3}
    \|\wt r_\ep, \wt n\|_{L^\infty(\R_+;\dot B^{\frac d2}_{2,1})}+
\|\nabla^2K_\ep r_\ep, \nabla^2 n\|_{L^1(\R_+;\dot B^{\frac d2}_{2,1})}\leq C \| \wt r_{\ep,0} ,\wt n_0\|_{\dot B^{\frac d2}_{2,1}}.
\end{equation}
If, furthermore, $\wt r_{\ep,0}$ tends to $\wt n_0$ in $\dot B^{\frac d2}_{2,1},$ 
then we have\footnote{Unless stronger assumptions are made
on $K_\ep,$ we do not have any rate of convergence.}:
\begin{equation}\label{eq:convn} \wt r_\ep\to \wt n\quad\hbox{in}\quad L^\infty(\R_+;\dot B^{\frac d2}_{2,1}).\end{equation}\end{thm}
\begin{proof}
The existence of $r_\ep$ with the desired properties follows from Theorem \ref{thm:fric1}.
 Next, as for \eqref{eq:pmeps}, since $\wt n_0:=n_0-1$ is small in $\dot B^{\frac d2}_{2,1},$
it is easy to see by variations on the fixed point theorem 
that there exists a unique solution $n$ to \eqref{eq:pme} satisfying the properties mentioned in the 
above statement. 
Let us prove the convergence of $r_\ep$ to $n$.
Set $\dn:=\wt n-\wt r_\ep.$ We have:
$$\partial_t\dn-\div(n\nabla K_\ep \dn)=\div(\dn\nabla K_\ep r_\ep)+\div(n(\Id-K_\ep)\nabla n).$$
We rewrite this expression in the form of a degenerate convection diffusion equation as follows:
$$\partial_t\dn-\nabla\dn\cdot\nabla K_\ep r_\ep-\Delta K_\ep\dn=\nabla\wt n\cdot\nabla K_\ep\dn +
\wt n\Delta K_\ep\dn+\dn\Delta K_\ep r_\ep
+\div((1+\wt n)(\Id-K_\ep)\nabla n).$$
Hence, arguing as in the proof of Theorem \ref{thm:fric1},  we get 
\begin{multline}\label{eq:dn}
\|\dn(t)\|_{\dot B^{\frac d2}_{2,1}}+\kappa_0\int_0^t\|\nabla^2K_\ep\dn\|_{\dot B^{\frac d2}_{2,1}}\,d\tau
\leq \|\dn_{0}\|_{\dot B^{\frac d2}_{2,1}}
+\int_0^t\|\nabla^2K_\ep\wt r_\ep\|_{\dot B^{\frac d2}_{2,1}}\|\dn\|_{\dot B^{\frac d2}_{2,1}}\,d\tau
\\+ \int_0^t\|\nabla\wt n\cdot\nabla K_\ep\dn\|_{\dot B^{\frac d2}_{2,1}}\,d\tau 
+\int_0^t\|\wt n\Delta K_\ep\dn\|_{\dot B^{\frac d2}_{2,1}}\,d\tau
+\int_0^t\|\dn\Delta K_\ep \wt r_\ep\|_{\dot B^{\frac d2}_{2,1}}\,d\tau\\
+\int_0^t\|(1+\wt n)(\Id-K_\ep)\Delta \wt n\|_{\dot B^{\frac d2}_{2,1}}\,d\tau
+\int_0^t\|\nabla\wt n\cdot(\Id-K_\ep)\nabla \wt n\|_{\dot B^{\frac d2}_{2,1}}\,d\tau.    
\end{multline}
From product law \eqref{eq:prod1}, we have:
$$\begin{aligned}
\|\nabla\wt n\cdot\nabla K_\ep\dn\|_{\dot B^{\frac d2}_{2,1}}&\lesssim
\|\nabla\wt n\|_{\dot B^{\frac d2}_{2,1}}\|\nabla K_\ep\dn\|_{\dot B^{\frac d2}_{2,1}},\\
\|\wt n\Delta K_\ep\dn\|_{\dot B^{\frac d2}_{2,1}}&\lesssim
\|\wt n\|_{\dot B^{\frac d2}_{2,1}}\|\Delta K_\ep\dn\|_{\dot B^{\frac d2}_{2,1}},\\
\|\dn\Delta K_\ep \wt r_\ep\|_{\dot B^{\frac d2}_{2,1}}&\lesssim
\|\dn\|_{\dot B^{\frac d2}_{2,1}}\|\Delta K_\ep \wt r_\ep\|_{\dot B^{\frac d2}_{2,1}},\\
\|(1+\wt n)(\Id-K_\ep)\Delta \wt n\|_{\dot B^{\frac d2}_{2,1}}&\lesssim
\bigl(1+\|\wt n\|_{\dot B^{\frac d2}_{2,1}}\bigr) \|(\Id-K_\ep)\Delta \wt n\|_{\dot B^{\frac d2}_{2,1}},\\
\|\nabla\wt n\cdot(\Id-K_\ep)\nabla \wt n\|_{\dot B^{\frac d2}_{2,1}}&\lesssim
\|\nabla \wt n\|_{\dot B^{\frac d2}_{2,1}}
\|(\Id-K_\ep)\nabla \wt n\|_{\dot B^{\frac d2}_{2,1}}.
\end{aligned}$$
As we work with small solutions, the second, fourth  and fifth terms of the right-hand side above may be absorbed by the 
left-hand side of \eqref{eq:dn}. Next, by interpolation, we have
$$\begin{aligned}
\|\nabla\wt n\|_{\dot B^{\frac d2}_{2,1}}\|\nabla K_\ep\dn\|_{\dot B^{\frac d2}_{2,1}}&\lesssim
\|\wt n\|_{\dot B^{\frac d2}_{2,1}}^{1/2}\|K_\ep\dn\|_{\dot B^{\frac d2}_{2,1}}^{1/2}
\|\Delta\wt n\|_{\dot B^{\frac d2}_{2,1}}^{1/2}\|\Delta K_\ep\dn\|_{\dot B^{\frac d2}_{2,1}}^{1/2}\\
&\lesssim \|\wt n\|_{\dot B^{\frac d2}_{2,1}}\|\Delta K_\ep\dn\|_{\dot B^{\frac d2}_{2,1}}
+\|\Delta\wt n\|_{\dot B^{\frac d2}_{2,1}}\|\dn\|_{\dot B^{\frac d2}_{2,1}}.
\end{aligned}$$
Hence the corresponding term may also be absorbed with  the left-hand side of \eqref{eq:dn}.
\smallbreak
Finally, in light of Lebesgue's dominated convergence theorem, 
since $\Delta\wt n\in L^1(\R_+;\dot B^{\frac d2}_{2,1})$ and 
$\nabla\wt n\in \wt L^2(\R_+;\dot B^{\frac d2}_{2,1}),$
we have
$$\lim_{\ep\to0}\|(\Id-K_\ep)\Delta \wt n\|_{L^1(\R_+;\dot B^{\frac d2}_{2,1})}=0\andf
\lim_{\ep\to0}\|(\Id-K_\ep)\nabla \wt n\|_{L^2(\R_+;\dot B^{\frac d2}_{2,1})}=0.$$
Plugging all this information in \eqref{eq:dn} completes the proof of \eqref{eq:convn}. 
\end{proof}

\subsection{Convergence of System \eqref{Kepsilon}  to the porous media equation}

The convergence of the density, solution of System \eqref{Kepsilon}
to the solution of  the porous media equation \eqref{eq:pme} when both $\f\to\infty$ and $\ep\to0$ may be deduced  from Theorems \ref{thm:fric1} and \ref{thm:fric2}.
Let $(\rho_{\f,\ep},u_{\f,\ep})$ be the solution of \eqref{Kepsilon}
and  $(\check \rho_{\f,\ep},\check u_{\f,\ep})$ be the corresponding rescaled solution (see \eqref{eq:diffusive}).
Let $r_\ep$ be the solution of \eqref{eq:a1} with data $\rho_0=1+a_0$ and, 
finally, $n$ the solution to \eqref{eq:pme} with the same data (for simplicity).
We have
$$\check \rho_{\f,\ep}-n=(\check \rho_{\f,\ep}-r_\ep)+(r_\ep-n).$$
Hence, in light of  Theorems \ref{thm:fric1} and \ref{thm:fric2}, one 
may conclude to the following result:
\begin{thm}\label{thm:fric3} Take $a_0$ and $u_0$ as in Theorem \ref{thm:fric1}.
Let $n$ be the solution to \eqref{eq:pme} with data $1+a_0.$ Let  $(\check \rho_{\f,\ep},\check u_{\f,\ep})$ be the solution of \eqref{Kepsilon} after rescaling \eqref{eq:diffusive}. 
Then 
$$\check \rho_{\f,\ep}-n \to 0\ \hbox{ in }\ L^\infty(\R_+;\dot B^{\frac d2-1}_{2,1}+ \dot B^{\frac d2}_{2,1})
\ \hbox{ as }\ \f\to\infty\andf \ep\to0.$$
\end{thm}


\section{Appendix}


\subsection{Commutator estimates}\label{ss:com}

As a first, we recall two product laws in Besov spaces that we used repeatedly in the paper
(the reader may refer to \cite[Chap. 2]{BCD} for more details):
\begin{eqnarray}\label{eq:prod1}
\|fg\|_{\dot B^\sigma_{2,1}}\leq C\|f\|_{\dot B^{\frac d2}_{2,1}}\|g\|_{\dot B^\sigma_{2,1}},\qquad -d/2<\sigma\leq d/2,\\\label{eq:prod2}
\|fg\|_{\dot B^\sigma_{2,1}}\leq 
C\Bigl(\|f\|_{L^\infty}\|g\|_{\dot B^\sigma_{2,1}}+
\|g\|_{L^\infty}\|f\|_{\dot B^\sigma_{2,1}}\Bigr),\qquad \sigma>0.
\end{eqnarray}
The latter inequality is often  combined with the embedding
\begin{equation}\label{eq:keyembedding}
\dot B^{\frac d2}_{2,1}\hookrightarrow L^\infty.\end{equation}
In the rest of this part, we focus on  the commutators estimates 
that we used for handling the terms
$v\cdot\nabla a,$ $v\cdot \nabla u,$  $b\,\div u$ 
and $c\nabla K_\ep a$ in System \eqref{eq:eulerconvbis}.
\smallbreak
The  following commutator estimate belongs to the mathematical folklore 
(see  \cite[Chap.~2]{BCD}):
\begin{lem}
Let  $ -d/2<\sigma\leq 1+d/2.$ Then, for $v \in \dot B^{d/2 +1}_{2,1}$ and 
$z \in \dot B^{\sigma}_{2,1},$ it  holds that
\begin{equation}\label{com:I}
\| [v,\ddj]\nabla z\|_{L^2}\leq Cc_j 2^{-j\sigma}\|\nabla v\|_{\dot B^{\frac d2}_{2,1}}\|z\|_{\dot B^\sigma_{2,1}},
\end{equation}
where    $(c_j)_{j\in\Z}$ denotes a nonnegative sequence with sum equal to $1.$ 
\end{lem}

The next commutator estimate is connected to the operator $L_\ep$ satisfying \eqref{eq:Kep}.
\begin{lem}   
Let $\sigma >-d/2.$
There exists a constant $C$ \emph{independent of $\ep$} such that we have:
\begin{equation}\label{com:II}
    \|[c,\partial_k L_\ep ] h \|_{\dot B^\sigma_{2,1}} \leq C\bigl( 
   \|c\|_{\dot B^{\frac d2+1}_{2,1}} \|h\|_{\dot B^{\sigma}_{2,1}}
    +\|h\|_{L^\infty} \|c\|_{\dot B^{\sigma +1}_{2,1}}\bigr),\qquad k\in\{1,\cdots,d\}.
\end{equation}
Moreover, if $-d/2<\sigma\leq d/2,$ then the second term is not needed.     
\medbreak
We  also have 
\begin{equation}\label{com:IIb}
    \|[c,L_\ep ] h \|_{\dot B^\sigma_{2,1}} \leq C
   \bigl(\|c\|_{\dot B^{\frac d2+1}_{2,1}} \|h\|_{\dot B^{\sigma-1}_{2,1}}
   +\|h\|_{L^\infty}\|c\|_{\dot B^\sigma_{2,1}}\bigr), \end{equation}
and the second term is not needed for $-d/2<\sigma\leq d/2+1.$
\end{lem}   
{\bf Proof.}
  One can take advantage  of the following (simplified) Bony decomposition:
\begin{equation}\label{eq:bony}fg=T_fg+T'_gf\with T_fg:=\sum_j \dot S_{j-1}f\,\ddj g\andf T'_{g}f:=\sum_j\dot S_{j+2}g\,\ddj f.\end{equation}
Now, using the paraproduct operators $T$ and $T',$ we have the decomposition:
\begin{equation}\label{eq:LF4b}[c,\partial_k L_\ep]h=[T_c,\partial_k L_\ep]h +T'_{\partial_k  L_\ep h}c-\partial_k L_\ep   T'_{h}c. \end{equation}
The last two terms may be bounded according to continuity results for the paraproduct
(see  \cite[Chap. 2]{BCD}):
if $\sigma>-d/2,$ then we have  (using  \eqref{eq:Kep})
\begin{equation}\label{eq:LF5}   \|T'_{\partial_k L_\ep h}c\|_{\dot B^\sigma_{2,1}}\lesssim
    \|\partial_k L_\ep h\|_{\dot B^{\sigma-1}_{2,1}}\|c\|_{\dot B^{d/2+1}_{2,1}}
\lesssim \| h\|_{\dot B^\sigma_{2,1}} \|c\|_{\dot B^{d/2+1}_{2,1}}\end{equation}
and, if $\sigma>0$
\begin{equation}\label{eq:LF6}
 \|\partial_k L_\ep T'_{h}c\|_{\dot B^\sigma_{2,1}}\lesssim
\|T'_{h}c\|_{\dot B^{\sigma+1}_{2,1}}\lesssim
 \|h\|_{L^\infty}\|c\|_{\dot B^{\sigma+1}_{2,1}}.
\end{equation}
Note that for $-d/2<\sigma\leq d/2,$ then we  have 
\begin{equation}\label{eq:LF6a}
 \|\partial_k L_\ep T'_{h}c\|_{\dot B^\sigma_{2,1}}\lesssim
\|T'_{h}c\|_{\dot B^{\sigma+1}_{2,1}}\lesssim
 \|h\|_{\dot B^\sigma_{2,1}}\|c\|_{\dot B^{d/2+1}_{2,1}}.
\end{equation}
For the first term  in the right-hand side of \eqref{eq:LF4b} we write that, by definition of paraproduct,
$$[T_c,\partial_k L_\ep]h=\sum_j [\dot S_{j-1}c,\partial_k L_\ep]\ddj h.$$
Now, from  the mean value formula, we gather 
$$[\dot S_{j-1}c,\partial_k L_\ep]\ddj h(x)=\int_0^1\!\!\int_{\R^d} \nabla\dot S_{j-1}c(y+\tau (x-y))\cdot (x-y)\: \partial_k L_\ep(x-y)\ddj h(y)\,dy.$$
Hence, for all $j\in\Z,$
$$\|[\dot S_{j-1}c,\partial_k L_\ep]\ddj h\|_{L^2} \leq  \|z\partial_k L_\ep\|_{L^1}
\|\nabla\dot S_{j-1}c\|_{L^\infty} \|\ddj h\|_{L^2}.$$
From this, Condition \eqref{eq:Kep} and Lemma 2.23 in \cite{BCD}, we deduce that
\begin{equation}\label{eq:LF7}
\|[T_c,\partial_k L_\ep]h\|_{\dot B^\sigma_{2,1}}
\lesssim \|\nabla c\|_{L^\infty}\|h\|_{\dot B^\sigma_{2,1}}
\end{equation}
and, in light of \eqref{eq:keyembedding}, we conclude to \eqref{com:II} 
\medbreak
Proving \eqref{com:IIb} is similar : we start from the decomposition
$$[c,L_\ep]h=[T_c,L_\ep]h +T'_{L_\ep h}c- L_\ep   T'_{h}c. $$
The last two terms may be bounded by means of standard continuity results for the paraproduct operator. 
To handle the first one, we introduce the function $L_{\ep,j'}:=\cF^{-1}(\wh L_\ep \varphi(2^{-j'}\cdot))$ and  write that 
\begin{align*}
[\dot S_{j-1}c,L_\ep]\ddj h(x)
&=\sum_{j'\sim j}[\dot S_{j-1}c,\dot\Delta_{j'}L_\ep]\ddj h(x)\\
&=\sum_{j'\sim j}
\int_0^1\!\!\int_{\R^d} \nabla\dot S_{j-1}c((y+\tau (x-y)))\cdot (x-y)\: L_{\ep,j'}(x-y)\ddj h(y)\,dy.    \end{align*}
Hence, for all $j\in\Z,$
\begin{equation}\label{eq:hj0}\|[\dot S_{j-1}c,L_\ep]\ddj h\|_{L^2} \leq  \sum_{j'\sim j}\|z L_{\ep,j'}\|_{L^1}
\|\nabla\dot S_{j-1}c\|_{L^\infty} \|\ddj h\|_{L^2}.\end{equation}
Since we have
$$\cF(z L_{\ep,j'})(\xi)=i2^{-j'}\Bigl(\wh L_\ep(\xi)\nabla\varphi(2^{-j'}\xi)+
\psi(2^{-j'}\xi)\:\xi\cdot\nabla \wh L_\ep (\xi)\Bigr)\with \psi(\eta):= |\eta|^{-2} \eta \varphi(\eta),$$
we get from convolution inequalities and \eqref{eq:Kep} that
\begin{equation}\label{eq:hj}\sup_{\ep,j'} 2^{j'}\|z L_{\ep,j'}\|_{L^1}<\infty.\end{equation}
So we have
$$2^{j\sigma}\|[\dot S_{j-1}c,L_\ep]\ddj h\|_{L^2} \leq C \|\nabla c\|_{L^\infty}
2^{j(\sigma-1)}\|\ddj h\|_{L^2},$$
and it is now easy to  complete the proof of \eqref{com:IIb}. 
\qed

\begin{lem} Assume that $\sigma>-d/2.$
Then  we have
\begin{equation}\label{com:III}
\|[L_\ep\ddj, c]h\|_{L^2}\leq Cc_j2^{-j\sigma}\bigl(\|c\|_{\dot B^{\frac d2+1}_{2,1}}\|h\|_{\dot B^{\sigma-1}_{2,1}}
+\|c\|_{\dot B^{\sigma}_{2,1}}\|h\|_{L^\infty}\bigr)\end{equation}
and the second term is not needed if $\sigma\leq d/2+1.$
\end{lem}

{\bf Proof.} For conciseness, we only treat the case $\sigma\leq d/2+1$
(the easy adaptations for $\sigma>d/2+1$ are left to the reader). 
One can  mimic the proof of \eqref{com:I} proposed in  
\cite{BCD}: using Bony's decomposition \eqref{eq:bony}, we write that 
\begin{equation}\label{eq:decdec} [L_\ep\ddj, c]h= [L_\ep\ddj, T_c]h
+L_\ep\ddj T'_{h}c-T'_{h}L_\ep\ddj c.\end{equation}
For the last term,  we have
$$T'_{h}L_\ep\ddj c=\sum_{|j'-j|\leq1} \dot S_{j+2}h\, \dot\Delta_{j'}L_\ep\ddj c. $$
Hence, since $\sigma-1-d/2\leq0,$ we have, thanks to Bernstein inequality, the definition of
Besov space $\dot B^{\sigma-1-\frac d2}_{\infty,1}$
and embedding $\dot B^{\sigma-1}_{2,1}
\hookrightarrow\dot B^{\sigma-1-\frac d2}_{\infty,\infty},$
$$\begin{aligned}\|T'_{h}L_\ep\ddj c\|_{L^2}&\lesssim \|\dot S_{j+2} h\|_{L^\infty} 
\|\ddj c\|_{L^2}\\&\lesssim 2^{-j(\sigma-1-\frac d2)}\|h\|_{\dot B^{\sigma-1-\frac d2}_{\infty,1}}\|\ddj c\|_{L^2}\\&\lesssim 2^{-j\sigma}\|h\|_{\dot B^{\sigma-1}_{2,1}}
\:2^{j(1+d/2)}\|\ddj c\|_{L^2}.\end{aligned}$$
The last but one term of \eqref{eq:decdec} may be bounded thanks 
to the fact that 
$$T':\dot B^{\sigma-1}_{2,1}\times\dot B^{\frac d2+1}_{2,1}\to \dot B^{\sigma}_{2,1},\qquad
-d/2<\sigma\leq d/2+1.$$
For bounding the first term, one can write that  both $\dot S_{j'-1}c\, h_{j'}$ and $h_{j'}$ are localized
in an annulus of size $2^{j'}.$ Hence, by definition of the paraproduct, we have
$$[L_\ep\ddj, T_c]h=\sum_{j'\sim j} [L_\ep\ddj, \dot S_{j'-1}c] h_{j'}.$$
The mean value formula ensures that for all $x\in\R^d$, we have
$$[L_\ep\ddj, \dot S_{j'-1}c] h_{j'}(x)=\int_{\R^d}\int_0^1L_{\ep,j}(x-y)\:
\nabla\dot S_{j'-1}c(x+\tau(y-x))\cdot(y-x)\:  h_{j'}(y)\,d\tau\,dy,$$
whence
$$\|[L_\ep\ddj, T_c]h\|_{L^2}\leq 
\sum_{j'\sim j}\|z L_{\ep,j}\|_{L^1}\|\nabla\dot S_{j'-1}c\|_{L^\infty}\|h_{j'}\|_{L^2}.$$
Hence, owing to \eqref{eq:hj}
\begin{equation}\label{eq:Lepddj}
  \|[L_\ep\ddj, T_c]h\|_{L^2}\leq Cc_j2^{-j\sigma}\|\nabla c\|_{L^\infty}\|h\|_{\dot B^{\sigma}_{2,1}},  
\end{equation} which, combined with \eqref{eq:keyembedding}, completes the proof of the lemma. 
\qed 

\begin{lem} Let $\sigma\in(-d/2,d/2+1].$
Let $z$ be in $\dot B^{d/2}_{2,1}\cap \dot B^{\sigma}_{2,1}$ and 
$b$ be a scalar function such that  $\nabla b \in \dot B^{d/2}_{2,1}$ and
$\nabla L_\ep b \in \dot B^{\sigma}_{2,1}.$ Then, 
\begin{equation}\label{com:IV}
\|\nabla L_\ep[b,\ddj]z \|_{L^2} \leq Cc_j 2^{-j\sigma}
\bigl(\|\nabla b\|_{\dot B^{\frac d2}_{2,1}}\|z\|_{\dot B^{\sigma}_{2,1}}
+\|z\|_{L^\infty}\|\nabla L_\ep b\|_{\dot B^{\sigma}_{2,1}}\bigr),\qquad j\in\Z.
\end{equation}
Furthermore, the second term is not needed if $-d/2<\sigma\leq d/2.$
\end{lem}
\noindent 
{\bf Proof.}
We start with  the decomposition
$$\begin{aligned}
\nabla L_\ep[b,\ddj]z &= L_\ep[b,\ddj]\nabla z+  L_\ep[\nabla b,\ddj]z\\
&= L_\ep[b,\ddj]\nabla z   +  L_\ep (\nabla b \, z_j) - \ddj L_\ep(\nabla b\, z)\\
&=R_j^{11}+R_j^{12}+R_j^{13}+R_j^{14}+R_j^{15}\end{aligned}$$
with $R_j^{11}:= L_\ep[b,\ddj]\nabla z,\quad\!\!
 R_j^{12}:= L_\ep (\nabla b \, z_j),\quad\!\!
  R_j^{13}:=- \ddj L_\ep T'_{\nabla b}z,\quad\!\!
 R_j^{14}:=-T_{z} L_\ep \ddj\nabla b$ and 
 $R_j^{15}:= [T_z,L_\ep\ddj ] \nabla b.$
 \medbreak\noindent
Since $0\leq\wh L_\ep\leq1,$  the commutator estimate \eqref{com:I}  ensures that
$$\|R_j^{11}\|_{L^2}\leq \|[b,\ddj]\nabla z\|_{L^2}
\leq Cc_j 2^{-j\sigma}\|\nabla b\|_{\dot B^{\frac d2}_{2,1}}
\|z\|_{\dot B^{\sigma}_{2,1}}.$$
Next, owing to \eqref{eq:keyembedding},
$$
\|R_j^{12}\|_{L^2}\leq \|\nabla b\, z_j\|_{L^2}\leq \|\nabla b\|_{L^\infty}
\|z_j\|_{L^2}\leq Cc_j 2^{-j\sigma}\|\nabla b\|_{\dot B^{\frac d2}_{2,1}}
\|z\|_{\dot B^{\sigma}_{2,1}}
$$
 and, because  $T':\dot B^{\frac d2}_{2,1}\times\dot B^{\sigma}_{2,1}\to \dot B^{\sigma}_{2,1}$
 for $\sigma>-d/2,$ we have
$$\|R_j^{13}\|_{L^2}\leq c_j2^{-j\sigma}\|T'_{\nabla b}\, z\|_{\dot B^{\sigma}_{2,1}}\leq Cc_j2^{-j\sigma}
\|\nabla b\|_{\dot B^{\frac d2}_{2,1}}\|z\|_{\dot B^{\sigma}_{2,1}}.$$
 Next, since 
 $$R_{j}^{14}=\sum_{|j'-j|\leq1}
 \dot S_{j'-1} z\: L_\ep \dot\Delta_{j'} \nabla b_j,$$
 we have 
 $$
  \|R_{j}^{14}\|_{L^2}\leq C\| z\|_{L^\infty}\|L_\ep\nabla b_j\|_{L^2}\leq Cc_j2^{-j\sigma}\| z\|_{L^\infty}\|L_\ep\nabla b\|_{\dot B^\sigma_{2,1}}.
  $$
  Note that if $\sigma\leq d/2,$ then we also have for $|j'-j|\leq1,$ 
 $$\| \dot S_{j'-1} z\|_{L^\infty}\lesssim 2^{-j(\sigma-d/2)}\|z\|_{\dot B^{\sigma-d/2}_{\infty,1}}
 \lesssim 2^{-j\sigma}\|z\|_{\dot B^{\sigma}_{2,1}} $$
so that 
  $$
  \|R_{j}^{14}\|_{L^2}\leq Cc_j2^{-j\sigma}\|\nabla b\|_{\dot B^{\frac d2}_{2,1}}\|z\|_{\dot B^{\sigma}_{2,1}}.  $$
 The term  $R_j^{15}$ may be treated by a small variation of \eqref{eq:Lepddj}.  We get
$$\|R_{j}^{15}\|_{L^2}\leq Cc_j2^{-j\sigma}\|\nabla b\|_{\dot B^{\frac d2}_{2,1}}\|\nabla z\|_{B^{\sigma-1-\frac d2}_{\infty,1}}.$$
In the end, remembering the embedding $B^{\sigma-1}_{2,1}\hookrightarrow 
B^{\sigma-1-\frac d2}_{\infty,1},$ we obtain  \eqref{com:IV}.

\qed

 {\bf An alternative proof of  \eqref{com:IV}. } 
Inequality \eqref{com:IV}  can be alternatively demonstrated by means of an  integral representation. In contrast to employing the para-decomposition, our approach necessitates the explicit elucidation of the paramount terms that engender limitations on the regularity. This particular scenario demands a more intricate analysis, wherein we focus on a three-dimensional space and impose constraints on the regularity with $\sigma = 5/2$.

Note that $\nabla L_\ep[b,\ddj]\div u=  L_\ep[ \nabla b,\ddj]\div u + L_\ep[b,\ddj] \nabla \div u$. The second part can be easily treated by the commutator rule (\ref{com:I})
$$
    2^{(d/2+1)j}\|L_\ep[b,\ddj] \nabla \div u\|_{L^2} \lesssim
    c_j \|\nabla b\|_{L^\infty} \|\div u\|_{\dot B^{d/2+1}_{2,1}} \with \sum_j c_j =1.
$$
The first part can be seen as follows
$$
   L_\ep[ \nabla b,\ddj]\div u = 
   L_\ep \nabla b \ddj \div u -    L_\ep \ddj (\nabla b \. \div u).
$$
Above, the first term is bounded by
$$
    2^{(d/2+1)j}\|L_\ep \nabla b \ddj \div u\|_{L^2} \lesssim 
    \|\nabla b\|_{L^\infty}
    2^{(d/2+1)j} \|\nabla u_j\|_{L^2}.
$$
Since $(\ddj \nabla b) \div u$ is of a good form and
$$
    2^{(d/2+1)j}\|L_\ep (\ddj \nabla b) \div u\|_{L^2} \lesssim 
    2^{(d/2+1)j}\|\ddj \nabla b\|_{L^2} \|\div u\|_{L^\infty}, 
$$
the most difficult term we  consider in the following form 
$$
    L_\ep \ddj (\nabla b \.\div u) - L_\ep (\ddj \nabla b) \div u.
$$
Let $L_\ep^j = \dot \Delta_j L_\ep$, then we restate the above term
\begin{multline*}
    \int_{\R^d} L^j_\ep (z) \nabla b(x-z) (\div u (x-z) -\div u(x))dz \\=
    \int_{\R^d} L^j_\ep (z) \nabla b(x-z)
    \left[\nabla^2u (x) z + \int_0^1 (\nabla ^2u(x-tz) -\nabla^2u(x)) z\, dt \right]  dz = K_1 +K_2.
\end{multline*}
Let fix $d=3$. In order to get the general case, it is required to apply the induction method to get the bound for arbitrary dimension.
First, we  find the bound for 
$$
    2^{5/2 j} \| K_1\|_{L^2} \leq 2^{5/2 j }\| \int_{\R^3} z L^j_\ep (z) \nabla b(x-z) dz \|_{L^2} \|\nabla^2u\|_{L^\infty} \leq C 2^{5/2 j}\|b_j\|_{L^2} \|\nabla^2u\|_{\dot B^{3/2}_{2,1}}.
$$
Note that by definition
$$
  \int_{\R^3} z L^j_\ep (z) \nabla b(x-z) dz = \int_{\R^3} z L_\ep (z) \nabla \ddj b(x-z) dz,  
$$
hence  by (\ref{eq:Kep})
$$
    \|\int_{\R^3} z L^j_\ep (z) \nabla b(x-z) dz \|_{L^2} \lesssim 
    \| \xi  \partial_{\xi}\hat L(\xi)
    \hat{b_j}\|_{L^2} \lesssim \|b_j\|_{L^2}.
$$

The term $K_2$ still needs to be restated. So (we  use the transform $(t,s)\to (ts,s)$)
\begin{multline*}
    K_2 = \int_{\R^3} z L^j_\ep(z) \nabla b(x-z) \int_0^1 \int_0^1 
    \nabla^3u(x-ts z) tz \, ds dt dz \\
    =\int_{\R^3} z^2 L^j_\ep(z) \nabla b(x-z) [\nabla^3 u(x) +\int_0^1
    (\nabla^3u(x-s z)   - \nabla^3u(x) )\,ds  ]dz = K_{21} + K_{22}.
\end{multline*}
And then 
\begin{multline*}
    2^{5/2 j} \| \int_{\R^3} z^2 L^j_\ep (z) \nabla b(x-z) dz     \nabla^3 u(x)\|_{L^2} \leq  
    2^{5/2 j }
    \|\int_{\R^3} z^2 L^j_\ep (z) \nabla b(x-z) dz \|_{L^6} \|\nabla^3 u\|_{L^3} \\
    \leq C 2^{5/2 j } 2^{-j} \|L_\ep b_j\|_{L^6} \|\nabla^3u\|_{L^3} \leq
    C 2^{5/2 j} \| L_\ep b_j\|_{L^2} \|\nabla^3 u \|_{B^{1/2}_{2,1}}.
\end{multline*}
Note the above case requires the highest regularity in both terms.

Term $K_{22}$ requires some more care. In an direct way we get
\begin{multline*}
    2^{5/2 j} \| \int_{\R^3} z^2 L^j_\ep (z) \nabla b(x-z) 
    \int_0^1 (\nabla^3u(x-s z) -\nabla^3 u (x))\, ds dz\|_{L^2} \\ \leq 
    2^{5/2 j} \| \int_{\R^3} z^2 L^j_\ep (z) \nabla b(x-z) 
    \int_0^1 \frac{(\nabla^3u(x-s z) -\nabla^3 u (x))}{(s|z|)^{1/2}}  s^{1+1/2} |z|^{1/2}\,\frac{|z|ds}{s|z|}  dz\|_{L^2} \\ \leq 
    C 2^{5/2 j} \| z^{2+1/2} L^j_\ep \|_{L^1} 
    \|\nabla b\|_{L^\infty} 
   \sup_{z\in \R^3} \int_0^\infty \frac{\| \nabla^3u (x - h \hat{e}_z) -
    \nabla^3 u(x)\|_2} {h^{1/2}} \frac{dh}{h}\\ \leq 
    C \|\nabla b\|_{L^\infty} \|\nabla^3 u\|_{B^{1/2}_{2,1}}.
\end{multline*}
By the assumption (\ref{eq:Kep}) we easily deduce that
$$ 2^{5/2 j}\|z^{5/2} L_\ep^j\|_{L^1} \leq \mbox{uniformly  bounded in} j \andf \ep.$$
The right-hand side is independent of $j$. The $\ell^1$ summability is required. So we proved the existence of a map from 
$B^{1/2}_{2,1}(\R^3) \to B^{5/2}_{2,\infty}(\R^3)$, but it is not enough. 
Fortunately we can use interpolation. Note that $1/2$ can be replaced by any $\alpha$ close to $1/2$, a bit bigger and a bit smaller, then we get the map
    $T: \dot B^{\alpha}_{2,1}(\R^3) \to \dot B^{\alpha}_{2,\infty}(\R^3),$
so then
\begin{equation*}
    T: \dot B^{1/2}_{2,1}(\R^3)= (\dot B^{1/2-\sigma}_{2,1}(\R^3);
    \dot B^{1/2 +\sigma}_{2,1} (\R^3))_{1/2,1} \to 
(\dot B^{1/2-\sigma}_{2,\infty}(\R^3);
    \dot B^{1/2 +\sigma}_{2,\infty}(\R^3))_{1/2,1} = 
    \dot B^{1/2}_{2,1}(\R^3)
\end{equation*}
with the suitable desired estimates. We are done.  This approach takes more space, but in some special cases can deliver faster answers concerning the limitation on the required regularity. \qed

\begin{lem} There exists a constant $C$ independent of $j\in\Z$ and $\ep>0$
such that the following inequality holds:
\begin{equation}\label{com:V}
\|[L_\ep,c]\ddj z\|_{L^2} \leq C2^{-j}\|\nabla c\|_{L^\infty}\|\ddj z\|_{L^2}.
\end{equation}
\end{lem}
{\bf Proof.} It is based on the decomposition
\begin{equation}\label{eq:decompo}[L_\ep,c]\ddj z = [L_\ep,\dot S_{j-1}c]\ddj z
+L_\ep\bigl((\Id-\dot S_{j-1})c\, \ddj z\bigr)
-(\Id-\dot S_{j-1})c\, L_\ep\ddj z,
\end{equation}
and on the fact that, in light of the properties of localization of 
$\ddj$ and $\dot S_{j-1},$ we have
$$[L_\ep,\dot S_{j-1}c]\ddj z=\sum_{j'\sim j}\,[L_\ep\dot\Delta_{j'},\dot S_{j-1}c]\ddj z.$$
Hence, \eqref{eq:hj0} and \eqref{eq:hj} guarantee that 
$$\|L_\ep,\dot S_{j-1}c]\ddj z\|_{L^2}\leq 2^{-j}\|\nabla\dot S_{j-1}c\|_{L^\infty}\|\ddj z\|_{L^2}$$
and, since
$$\|(\Id-\dot S_{j-1})c\|_{L^\infty} \leq C2^{-j}\|\nabla c\|_{L^\infty},$$
the other two terms of \eqref{eq:decompo} also satisfy 
the desired inequality.
\qed

\begin{lem}
We have for $\sigma\in(-d/2-1, d/2+1],$
\begin{multline}\label{eq:Rj2a}
\| \nabla L_\ep\ddj(v\cdot\nabla a)-v\cdot\nabla (\nabla  L_\ep \ddj a)\|_{L^2}\\
\leq Cc_j2^{-j\sigma}
\bigl(\|a\|_{\dot B^{\sigma}_{2,1}}\|v\|_{\dot B^{\frac d2+2}_{2,1}}
+\|v\|_{\dot B^{\frac d2+1}_{2,1}}
\|\nabla L_\ep a\|_{\dot B^{\sigma}_{2,1}}\bigr)\cdotp
\end{multline}
\end{lem}

{\bf Proof.}
Using again Bony's decomposition \eqref{eq:bony} and the fact that   $\dot S_{j'-1}v\cdot\nabla a_{j'}$ is localized in an annulus of size $2^{j'},$ we may write 
$$
\begin{aligned}
\nabla L_\ep\ddj (v\cdot\nabla a)&=\nabla L_\ep \ddj T'_{\nabla a}\cdot v+\nabla L_\ep \ddj T_v\cdot\nabla a\\
&=\nabla L_\ep \ddj T'_{\nabla a}\cdot v+\nabla L_\ep \ddj \sum_{j'\sim j} \dot S_{j'-1} v \cdot \nabla a_{j'}\\
&=R_j^{21}+R_j^{22}+R_j^{23}+R_j^{24}+v\cdot\nabla(\nabla L_\ep a_j)\end{aligned}
$$
$$\begin{aligned}
\with&R_j^{21}:=\nabla L_\ep T'_{\nabla a}\cdot v,\qquad R_j^{22}:=\nabla L_\ep \ddj \sum_{j'\sim j} (\dot S_{j'-1}-\dot S_{j-1})v \cdot \nabla a_{j'},\\ 
&R_j^{23}:= \sum_{j'\sim j}[\nabla L_\ep \ddj,\dot S_{j-1}v] \cdot \nabla a_{j'}\andf
R_j^{24}:=(\dot S_{j-1}-\Id)v \cdot(\nabla(\nabla L_\ep a_j)).\end{aligned}$$
For  $R_j^{21},$  we use that  $T':\dot B^{\sigma-1}_{2,1}\times \dot B^{\frac d2+2}_{2,1}\to \dot B^{\sigma+1}_{2,1}$ for $-d/2-1<\sigma\leq d/2+1$ and that 
$\nabla L_\ep:\dot B^{\sigma+1}_{2,1}\to \dot B^{\sigma}_{2,1}$ uniformly with respect to $\ep$
to  get
\begin{equation}\label{eq:Rj21}
\|R_j^{21}\|_{L^2}
\leq Cc_j2^{-j\sigma}\|\nabla a\|_{\dot B^{\sigma-1}_{2,1}}\|v\|_{\dot B^{\frac d2+2}_{2,1}}.
\end{equation}
Next, by Bernstein inequality and the fact that $0\leq\wh L_\ep\leq1,$
we have
$$\begin{aligned}
\|R_j^{22}\|_{L^2}&\lesssim 2^j \sum_{j''\sim j'\sim j}
2^{-2j''-j'(\sigma-1)}\bigl(2^{2j''}\|v_{j''}\|_{L^\infty}\bigr) \bigl(2^{j'(\sigma-1)}\|\nabla a_{j'}\|_{L^2}\bigr)\\
&\lesssim 2^{-j\sigma} \|v\|_{\dot B^2_{\infty,\infty}}
\sum_{j'\sim j}\bigl(2^{j'(\sigma-1)}\|\nabla a_{j'}\|_{L^2}\bigr)\cdotp
\end{aligned}$$
Hence we have, owing to embedding $\dot B^{\frac d2+2}_{2,1}\hookrightarrow \dot B^2_{\infty,\infty},$
\begin{equation}\label{eq:Rj22}
\|R_j^{22}\|_{L^2} \leq Cc_j2^{-j\sigma}\|v\|_{\dot B^{\frac d2+2}_{2,1}}
\|\nabla a\|_{\dot B^{\sigma-1}_{2,1}}.\end{equation}
To bound  $R_j^{24},$ we use the fact that
$$
\|(\dot S_{j-1}-\Id)v\|_{L^\infty}\lesssim 2^{-j}
 \|v\|_{\dot B^1_{\infty,\infty}}.$$
Hence, combining with  the  embedding $\dot B^{\frac d2+1}_{2,1}\hookrightarrow \dot B^1_{\infty,\infty}$ and Bernstein inequality, 
\begin{equation}\label{eq:Rj24}
\|R_j^{24}\|_{L^2}\leq C\|v\|_{\dot B^{\frac d2+1}_{2,1}}
\|\nabla L_\ep a_j\|_{L^2}.
\end{equation}
To handle  $R_j^{23},$ we have to go to the
second order in the Taylor expansion. Using again the notation $L_{\ep,j}=\cF^{-1}(\wh L_\ep\,\varphi(2^{-j}\cdot)),$
we write that for all $k\in\{1,\cdots,d\},$ we have, with the summation convention,
$$\displaylines{
[\partial_kL_\ep\ddj,\dot S_{j-1}v^\ell]\partial_\ell a_{j'}
= R_{jj'k}^{231}+ R_{jj'k}^{232}\cr
\with R_{jj'k}^{231}(x):=
\int_{\R^d}\partial_k L_{\ep,j}(y-x)\: (y-x)\cdot \nabla\dot S_{j-1}v^\ell(x)
\:\partial_\ell a_{j'}(y)\,dy\andf\cr
R_{jj'k}^{232}(x):=\int_{\R^d}\biggl(\int_0^1(1-\tau)
D^2\dot S_{j-1}v^\ell(x+\tau(y-x))(y-x,y-x)\,d\tau\biggr)
\partial_k L_{\ep,j}(x-y)\:\partial_\ell a_{j'}(y)\,dy.}$$
First, by using H\"older inequality, we have
$$\|R_{jj'k}^{231}\|_{L^2}\leq 
\|\nabla\dot S_{j-1}v^\ell\|_{L^\infty}
\|z\partial_k L_{\ep,j}\star\partial_\ell a_{j'}\|_{L^2}.$$
Denoting $h_j:=\cF^{-1}(\varphi(2^{-j}\cdot)),$ we have the identity
$$z\partial_{z_k}(L_\ep\star h_j)= \partial_{z_k}L_\ep\star (z h_j)
- L_\ep\star \partial_{z_k}(z h_j)$$
and thus 
$$z \partial_k L_{\ep,j}\star \partial_\ell a_{j'}
= (z h_j)\star L_\ep\partial_k\partial_\ell a_{j'}
- \partial_k(z h_j)\star \partial_\ell L_\ep a_{j'}.$$
Since 
$$\|zh_j\|_{L^1}=2^{-j}\|z h_0\|_{L^1}\andf 
\|\partial_k(z h_j)\|_{L^1}= \|\partial_k(z h_0\|_{L^1},$$
we deduce (using once Bernstein inequality) that
\begin{equation}\label{eq:Rj231}
   \|R_{jj'k}^{231}\|_{L^2}\lesssim \|\nabla v\|_{L^\infty}\|L_\ep\nabla a_{j'}\|_{L^2}. \end{equation}
For the other term, we have
$$\|R_{jj'k}^{232}\|_{L^2}\lesssim \|\nabla^2 v\|_{L^\infty}
\|(z\otimes z)\nabla L_{\ep,j}\|_{L^1}\|\nabla a_j\|_{L^2},$$
and one can show that
$$\|(z\otimes z)\nabla L_{\ep,j}\|_{L^1}\lesssim 2^{-j}
\|(z\otimes z)\nabla^2 L_{\ep}\|_{L^1}.$$
Indeed, if we set $\wt\varphi(\xi)=-i\xi |\xi|^{-2}\varphi(\xi)$
and $\wt h_0:=\cF^{-1}\wt\varphi,$ then $h_0=\div \wt h_0,$
and thus $\wt h_j=2^{-j}\div\wt h_j$ for all $j\in\Z.$
Consequently, we have
$$(z\otimes z)\nabla L_{\ep,j}= 2^{-j}(z\otimes z)(\Delta L_\ep \star  \wt h_j).$$ 
Hence 
\begin{equation}\label{eq:Rj232}\|R_{jj'k}^{232}\|_{L^2}\lesssim \|\nabla^2 v\|_{L^\infty}\|a_{j'}\|_{L^2}.\end{equation}
Putting \eqref{eq:Rj231} and \eqref{eq:Rj232} together yields
\begin{equation}\label{eq:Rj23}
\|R_j^{23}\|_{L^2}
\leq C\sum_{j'\sim j} 
\bigl(\|\nabla v\|_{L^\infty}\|L_\ep\nabla a_{j'}\|_{L^2}
+\|\nabla^2 v\|_{L^\infty}\|a_{j'}\|_{L^2}\bigr).
\end{equation}
Hence, one can conclude from  \eqref{eq:Rj21},
\eqref{eq:Rj22}, \eqref{eq:Rj24} and \eqref{eq:Rj23} that \eqref{eq:Rj2a} holds true. 
\qed 


\subsection{The general pressure case}\label{s:pressure}

Here we  explain how to close the estimates for all time in the general pressure case, 
that is for System \eqref{Kepsilon-PP}.
Denoting $a:=\rho-1,$ this corresponds to System \eqref{eq:eulerconvbis} with $v=u,$ $b=a$ and $c=F(K_\ep a)$
with $F(z):= \cN(1+z)-1.$
We assume that\footnote{The second assumption can be achieved after renormalization.} 
\begin{equation}\label{eq:F}F(0)=0\andf F'(0)=1. \end{equation}
We plan to use  Inequality \eqref{eq:linear} with $\sigma=d/2+1.$ 
Note that, at some point, we will have to bound the $L^\infty$ norm of $c_t+\div((1+c)v)$ with $c=F(K_\ep a).$
To do this, we observe from the first equation of \eqref{Kepsilon-PP} that
$$\partial_t(K_\ep(a))+\div\bigl((1\!+\!K_\ep a)u\bigr) 
= \cR_\ep:=  K_\ep a \,\div u + \sum_j [u^j,\partial_j K_\ep]a +(\Id-K_\ep)\div u,$$
whence 
$$\partial_t(F(K_\ep a))+\div\bigl((1\!+\!F(K_\ep a))u\bigr) = \bigl(1\!+\! F(K_\ep a)-(1\!+\!K_\ep a)F'(K_\ep a)\bigr)\div u +F'(K_\ep a) \cR_\ep. $$
Under condition \eqref{eq:smallbc}  with $b=a$ and thanks to hypothesis \eqref{eq:F}, it is obvious that
$$\|\bigl(1\!+\! F(K_\ep a)-(1\!+\!K_\ep a)F'(K_\ep a)\bigr)\div u\|_{L^\infty}\lesssim \|\div u\|_{L^\infty}\|a\|_{L^\infty}.$$
Next, using first order Taylor formula, we readily get
$$\Bigl\|\sum_j [u^j,\partial_j K_\ep]a\Bigr\|_{L^\infty}\leq \|z\nabla K_\ep\|_{L^1} \|\nabla u\|_{L^\infty}\|a\|_{L^\infty}.$$
Hence, keeping  Assumption \eqref{eq:Kep} in mind and assuming e.g. that $|a|\leq 1/4$, we conclude that 
\begin{equation}\label{eq:ct}
\bigl\|\partial_t(K_\ep(a))+\div\bigl((1\!+\!K_\ep a)u\bigr) \bigr\|_{L^\infty}\lesssim\|\nabla u\|_{L^\infty}.
\end{equation}
Now, denoting $X:=X^{\frac d2+1}_{a,u}$ and $H:=H^{\frac d2+1}_{a,u},$ using \eqref{eq:ct}
and observing that in the case $v=u$ and $b=a$ all the terms in lines two and 
three are of type  $\|\nabla u\|_{\dot B^{\frac d2}_{2,1}\cap
\dot B^{\frac d2+1}_{2,1}}X,$ Inequality \eqref{eq:linear} with $\sigma=d/2+1$ reduces to 
\begin{multline*}   
X(t)+\int_0^tH\,d\tau \lesssim X(0)+\int_0^t \|\nabla u\|_{\dot B^{\frac d2}_{2,1}\cap
\dot B^{\frac d2+1}_{2,1}}X\,d\tau\\
 +\!\int_0^t\!\Bigl(\|F(K_\ep a)\|_{\dot B^{\frac d2}_{2,1}}
 \bigl(\|\nabla L_\ep a^h\|_{\dot B^{\frac d2+1}_{2,1}} +\|\nabla K_\ep a^\ell\|_{\dot B^{\frac d2+1}_{2,1}}
 +\|\nabla^2 K_\ep a^\ell\|_{\dot B^{\frac d2+1}_{2,1}}\bigr) \\
 +\|F(K_\ep a)\|_{\dot B^{\frac d2+1}_{2,1}}\bigl(\|\nabla L_\ep a^h\|_{\dot B^{\frac d2}_{2,1}} +\|\nabla K_\ep a^\ell\|_{\dot B^{\frac d2}_{2,1}} +\|\nabla^2 K_\ep a^\ell\|_{\dot B^{\frac d2}_{2,1}}\bigr)\\
 +\|F(K_\ep a)\|_{\dot B^{\frac d2+1}_{2,1}}\|(L_\ep a,\nabla L_\ep a)\|_{\dot B^{\frac d2+1}_{2,1}}
  +\|F(K_\ep a)\|_{\dot B^{\frac d2+2}_{2,1}}\|\nabla L_\ep a\|_{\dot B^{\frac d2}_{2,1}}\Bigr)d\tau.\end{multline*}
Since $F(0)=0,$ the right-hand side may be simplified thanks to the following composition inequality that is
valid whenever $\|z\|_{L^\infty}$ is small enough and $s>0$:
$$\|F(z)\|_{\dot B^s_{2,1}}\lesssim \|z\|_{\dot B^s_{2,1}}.$$
In the end, after a few simplifications, we discover that
\begin{multline*}   
X(t)+\int_0^tH\,d\tau \lesssim X(0)+\int_0^t \|\nabla u\|_{\dot B^{\frac d2}_{2,1}\cap
\dot B^{\frac d2+1}_{2,1}}X\,d\tau\\
 +\!\int_0^t\!\Bigl(\|K_\ep a\|_{\dot B^{\frac d2}_{2,1}}
 \bigl(\|\nabla L_\ep a^h\|_{\dot B^{\frac d2+1}_{2,1}} +\|\nabla K_\ep a^\ell\|_{\dot B^{\frac d2+1}_{2,1}}
 +\|\nabla^2 K_\ep a^\ell\|_{\dot B^{\frac d2+1}_{2,1}}\bigr) \\
 +\|K_\ep a\|_{\dot B^{\frac d2+1}_{2,1}}\bigl(\|L_\ep a\|_{\dot B^{\frac d2+1}_{2,1}} +\|\nabla K_\ep a^\ell\|_{\dot B^{\frac d2+1}_{2,1}}\bigr)\\
 +\|K_\ep a\|_{\dot B^{\frac d2+1}_{2,1}}\|\nabla L_\ep a\|_{\dot B^{\frac d2+1}_{2,1}}
  +\|K_\ep a\|_{\dot B^{\frac d2+2}_{2,1}}\|\nabla L_\ep a\|_{\dot B^{\frac d2}_{2,1}}\Bigr)d\tau.\end{multline*}
We observe that all the products  in the integrals of the right-hand side may be bounded
by $HX$ except, maybe,  
$$\|K_\ep a\|_{\dot B^{\frac d2+1}_{2,1}}^\ell\|L_\ep a\|_{\dot B^{\frac d2+1}_{2,1}}^\ell
\andf \|K_\ep a\|_{\dot B^{\frac d2+1}_{2,1}}^\ell\|\nabla L_\ep a\|_{\dot B^{\frac d2+1}_{2,1}}^\ell.$$ 
However, by Cauchy-Schwarz inequality in the Fourier space and the fact that $K_\ep =L_\ep\ast L_\ep$, we notice that
$$\begin{aligned}\|K_\ep a\|_{\dot B^{\frac d2+1}_{2,1}}^\ell&\leq \|L_\ep a\|_{\dot B^{\frac d2+1}_{2,1}}^\ell\leq 
\sqrt{\|a\|_{\dot B^{\frac d2}_{2,1}}^\ell\|\nabla^2 K_\ep a\|_{\dot B^{\frac d2}_{2,1}}^\ell}\leq \sqrt{XH}\\
\|\nabla L_\ep a\|_{\dot B^{\frac d2+1}_{2,1}}^\ell&\leq 
\sqrt{\|a\|_{\dot B^{\frac d2+1}_{2,1}}^\ell\|\nabla^2 K_\ep a\|_{\dot B^{\frac d2+1}_{2,1}}^\ell}\leq \sqrt{XH}.\end{aligned}$$
Hence, we conclude that for some universal constant $C\geq1,$ we have for all $t\geq0$
$$X(t)+\int_0^t H\,d\tau \leq C\biggl(X(0)+\int_0^t HX\,d\tau\biggr)\cdotp$$
Now, provided $2C^2 X(0)<1,$ we get the following global-in-time and uniform in $\ep$ control: 
$$X(t)+\frac12\int_0^t H\,d\tau \leq CX(0).$$
Granted with  the above a priori estimate, remembering Remark \ref{r:fundamental}
and mimicking the proof of Theorem \eqref{thm:GWP}, 
one ends up with the following global uniform  well-posedness result for 
System \eqref{Kepsilon-PP}.
\begin{thm}
Let $\cN$ be any smooth  function defined  
on some neighborhood of $1$ and such that $\cN(1)=\cN'(1)=1.$
Assume that $d\geq2$ and take initial data $\rho_0=1+a_0$ and $u_0$
such that  
$$u_0\in \dot B^{\frac d2}_{2,1}\cap \dot B^{\frac d2+2}_{2,1},
\quad a_0\in \dot B^{\frac d2-1}_{2,1}\cap \dot B^{\frac d2+1}_{2,1}\andf
\nabla^2L_\ep a_0 \in \dot B^{\frac d2}_{2,1}.$$
There exists an absolute  positive constant $\alpha_0$ such that if 
\begin{equation*}
\|u_0\|_{\dot B^{\frac d2+1}_{2,1}\cap \dot B^{\frac d2+2}_{2,1}}
+\|a_0\|_{\dot B^{\frac d2}_{2,1}\cap \dot B^{\frac d2+1}_{2,1}}
+\|\nabla^2L_\ep a_0 \|_{\dot B^{\frac d2}_{2,1}}\leq\alpha_0,
\end{equation*}
then System \eqref{Kepsilon-PP} with $\f=1$ supplemented with initial data 
$(\rho_0,u_0)$ admits a unique global classical
solution $(\rho,u)$ such that  $(a,u)$ with 
$a:=\rho-1$  belongs to the space $E_{K_\ep}^{\frac d2+1}$
defined in \eqref{def:E}. Furthermore, $a\in\cC(\R_+;\dot B^{\frac d2-1}_{2,1})$
and Inequality \eqref{est:1} is satisfied.
%
\end{thm}
In this general pressure setting, it  is also possible to prove  convergence to the compressible Euler
System \eqref{Euler-P} and  asymptotic results when the friction coefficient $\f$ tends to $\infty,$
in the spirit of Theorems \ref{thm:conv}, \ref{thm:fric1},  \ref{thm:fric2} and \ref{thm:fric3}. 
The details are left to the reader.


\subsection{From the micro to the macro scale}

In this part we aim at sketching the connection between \eqref{eq:boltz} and \eqref{Kepsilon}. 
In case the number $N$ of particles is large in \eqref{eq:boltz}, it is customary to treat the distribution of particles in terms of  measures. By  performing the so-called mean field limit, we are led to the 
following kinetic  equation:
\begin{equation}\label{eq:kin}
f_t +v\cdot \nabla_x f + \div_v (F(f)f)=0
\end{equation}
where, for some suitable  kernel $K_\ep,$
\begin{equation*}
    F(f)(t,x,v)= \f v +\int_{\R^d}
    \nabla K_\ep(x-y)f(t,y,w)\,dy.
\end{equation*}
Note  that the solution to \eqref{eq:boltz} may be seen as 
a \emph{measure solution} to \eqref{eq:kin}. Indeed, 
 the weak formulation of \eqref{eq:kin} reads, 
for all test function $\phi \in \mathcal{D}(\R^d \times [0,T)),$ 
\begin{equation*}
    \int_0^T \int_{\R^d}\! \int_{\R^d} 
    f (\partial_t + v\cdot \nabla_x + F(f)\nabla_v) \phi \,dx\,dv\,dt =- \int_{\R^d} \int_{\R^d} f_0(x,v)\phi(0,x,v)\,dx\,dv.
\end{equation*}
Hence, if we set $$f := \sum_k  \delta_{x_k(t)} \otimes 
\delta_{v_k(t)},$$
then we have \begin{multline*}
\sum_k     \int_0^T \int_{\R^d}\! \int_{\R^d} 
    \left(\partial_t + v_k(t)\cdot \nabla_x + F(\sum_l \delta_{x_l(t)} \otimes 
\delta_{v_l(t)})\nabla_v \right) \phi(x_k(t),v_k(t))\,dx\,dv
\\
=- \int_{\R^d} \int_{\R^d} f_0(x,v)\phi(0,x,v)\,dx\,dv,
\end{multline*}
since 
\begin{equation*}
    \frac{d}{dt} \phi(t, x_k(t),v_k(t))=\partial_t \phi+
    \frac{dx_k}{dt} \cdot \nabla_x \phi + \frac{dv_k}{dt} \cdot \nabla_v \phi.
\end{equation*}
Next,  let us  explain  how \eqref{Kepsilon} can be obtained from \eqref{eq:kin}.
The idea is to assume that we are in  the \emph{mono-kinetic regime}, namely
\begin{equation}\label{mono-kin}
    f(t,x,v) = \rho(t,x) \otimes \delta_{v=u(t,x)}
\end{equation}
for some nonnegative function $\rho(t,x)$ and vector-field $u(t,x).$ In other words,  all the particles at point $x$ at time  $t$ have the same velocity $u(t,x),$ and their density is $\rho(t,x)$. Then, first integrating over  the $v$-coordinate the equation (\ref{eq:kin}) we obtain the simple continuity law
\begin{equation}\label{eq:kin1}
    \rho_t + \div(\rho u) =0.
\end{equation}
Second,  multiplying \eqref{eq:kin} by $v$  and  integrating over $v$ gives:
$$\begin{aligned}
    0&=\int_{\R^d} vf_t\, dv  + \int_{\R^d} v \otimes v\cdot \nabla_x f \,dv - \int_{\R^d} v \,\div_v ((\f v + \nabla K_\ep \ast_x \rho)f) \,dv
    \\&=
    \frac{d}{dt} \int_{\R^d} vf\, dv  + \div_x \int_{\R^d} v \otimes v \, f \,dv + d\int_{\R^d}  (\f v + \nabla K_\ep \ast_x \rho \, f)  \,dv.
\end{aligned}$$
Taking into account the ansatz \eqref{mono-kin} we obtain then
\begin{equation}\label{eq:kin2}
    \partial_t(\rho u) + \div (\rho u \otimes u)+ d\f \rho u + d\rho \nabla K_\ep \ast \rho=0.
\end{equation}
After a suitable rescaling of the constant parameters, using \eqref{eq:kin1} to (\ref{eq:kin2}) we obtain the original system \eqref{Kepsilon}.


\bigbreak

\noindent
\textbf{Acknowledgement:} The paper has been partly supported by the Polish National Science Centre’s Grant No. 2018/30/M/ST1/00340 (HARMONIA).


\end{document}